\documentclass[a4paper,11pt]{amsart}
\usepackage{placeins}
\usepackage[utf8]{inputenc}
\usepackage[T1]{fontenc}
\usepackage[english]{babel}
\usepackage{amssymb,amsfonts,amsmath,amsthm}
\usepackage{mathtools}
\usepackage{graphicx}
\usepackage{xcolor}
\usepackage[margin=3.5cm]{geometry}
\usepackage[hidelinks]{hyperref}

\makeatletter
\newcommand{\datahoje}{\two@digits{\day}-\two@digits{\month}-\number\year}
\def\@setdate{\@date\@addpunct.}
\makeatother

\newtheorem{theorem}{Theorem}[section]
\numberwithin{equation}{section}
\newtheorem{proposition}[theorem]{Proposition}
\newtheorem{lemma}[theorem]{Lemma}
\newtheorem{corollary}[theorem]{Corollary}

\newtheorem{lettertheorem}{Theorem}

\newtheorem{lettercorollary}[lettertheorem]{Corollary}
\theoremstyle{definition}
\newtheorem{definition}[theorem]{Definition}
\theoremstyle{remark}
\newtheorem{remark}[theorem]{Remark}

\newcommand{\R}{\mathbb{R}}

\newcommand{\s}{\mathbb{S}}
\newcommand{\N}{\mathbb{N}}

\newcommand{\QQ}{\mathcal{Q}}

\newcommand{\ran}{\rangle}
\newcommand{\lan}{\langle}

\newcommand{\Z}{\mathbb{Z}}
\DeclareMathOperator{\sgn}{sgn}
\DeclareMathOperator{\dive}{div}
\DeclareMathOperator{\Hess}{Hess}
\newcommand{\Gpq}{O(p{+}1)\times O(q{+}1)}

\title[Bi-rotational self-shrinkers]{The global structure of bi-rotational self-shrinkers}
\author{Gregório Silva Neto}
\address{Instituto de Matemática, Universidade Federal de Alagoas, Maceió, 57072-900, Brazil}
\email{gregorio@im.ufal.br}
\date{}


\begin{document}

\begin{abstract}
We describe the global structure of the self-shrinkers of the mean curvature flow in $\R^{p+q+2}$ invariant under $\Gpq$, without assuming properness or embeddedness. The main tool is a pair of piecewise monotone quantities along the profile curves, in the spirit of Hsiang's analysis of constant mean curvature hypersurfaces. We prove that every such self-shrinker other than the two invariant cylinders intersects both of them transversely, and all of them except the Clifford minimal cone are complete. Apart from those of the cylinders, each end is either asymptotic to a cone or meets both cylinders infinitely many times. Consequently, a complete noncompact example is proper if and only if all its ends are asymptotically conical, and this is equivalent to polynomial volume growth, to finite entropy, and to finitely many intersections with the cylinders. The reduced flow preserves a finite measure, and for almost every initial condition the profile curve is either closed or generates a complete nonproper self-shrinker of infinite entropy. We also bound the winding of the profile curve by its number of cylinder crossings and localize its self-intersections.
\end{abstract}

\maketitle

\tableofcontents

\section{Introduction}\label{sec:intro}

The mean curvature flow of an immersion $X_0:\Sigma\to\R^{n+1}$ is a family of immersions $X:\Sigma\times[0,T)\to\R^{n+1}$ satisfying
\[
\left\{\begin{aligned}
\left(\dfrac{\partial X}{\partial t}(p,t)\right)^\perp&= H(p,t)\nu(p,t)\\
X(p,0) &= X_0,
\end{aligned}\right.
\]
where $H$ is the mean curvature and $\nu$ the inner unit normal. By the first variation formula
\[
\left.\frac{d}{dt}\mbox{Area}(t)\right|_{t=0}=-\int_\Sigma f H d\Sigma
\]
where $f$ is the normal speed of the variation, the flow is the gradient flow of area, and since $H\nu=\Delta_{g(t)}X$ it is a parabolic system of heat type.

Self-similar solutions move by homotheties or isometries. They are the simplest solutions of the flow and model its singularities: suitable parabolic rescalings at a singularity converge to a self-shrinker, smoothly and subsequentially at type I singularities by Huisken's monotonicity formula \cite[Theorem~3.5]{Huisken}, and in the weak sense of tangent flows in general \cite{Ilmanen95,White97} (see also \cite{CM}).

An immersion $X_0$ moves homothetically if
\begin{equation}\label{eq:soliton}
H=c\lan X,\nu\ran, \ c\in\R.
\end{equation}
It is a self-shrinker if $c<0$, a self-expander if $c>0$, and minimal if $c=0$. We drop the subscript $0$ from now on.

Let $G=\Gpq$, $p,q\ge1$, act in the standard way on $\R^{p+1}\times\R^{q+1}=\R^{n+1}$. The orbit through $(U,V)$ is $\s^{p}(|U|)\times\s^{q}(|V|)$, and the orbit space is
\begin{equation}
\overline{\QQ}:=\pi(\R^{n+1})=\{(x,y)\in\R^2;\ x\geq 0,\ y\geq 0\},
\end{equation}
where $\pi(U,V)=(|U|,|V|)$ (see \cite[Sections~2--4]{BdCH} and \cite{Hsiang}). 
Let
\[
\QQ:=\{(x,y)\in\R^{2};\ x>0,\ y>0\}
\]
be the interior of $\overline{\QQ}$, where the system below is defined. 

A $G$-invariant hypersurface can be parametrized by
\begin{equation}\label{eq:param}
X(s,\theta_1,\theta_2)=(x(s)\Psi_1(\theta_1),y(s)\Psi_2(\theta_2)),
\end{equation}
where $\Psi_1$ and $\Psi_2$ parametrize $\s^{p}\subset\R^{p+1}$ and $\s^{q}\subset\R^{q+1}$, and $\gamma(s)=(x(s),y(s))$ is the profile curve. We call such hypersurfaces bi-rotational.

If $\gamma$ is parametrized by the arclength $s$, then \eqref{eq:soliton} is equivalent to
\begin{equation}\label{lambda-s}
x'(s)y''(s)-x''(s)y'(s)=-y'(s)\left(\frac{p}{x(s)}+cx(s)\right) + x'(s)\left(\frac{q}{y(s)}+cy(s)\right),
\end{equation}
and, writing $\gamma'(s)=(\cos(\sigma(s)),\sin(\sigma(s)))$, to the system
\begin{equation}\label{eq:flow}
\left\{
\begin{aligned}
x' &= \cos\sigma;\\
y' &= \sin\sigma;\\
\sigma' &= -\left(\frac{p}{x}+cx\right)\sin\sigma+\left(\frac{q}{y}+cy\right)\cos\sigma,
\end{aligned}\right.
\end{equation}
on the phase space $M:=\QQ\times\s^1$. A solution $\Gamma(s)=(x(s),y(s),\sigma(s))$, called an orbit or trajectory, satisfies $\frac{d}{ds}\Gamma=F(\Gamma)$, where \[F(x,y,\sigma)=\left(\cos\sigma,\sin\sigma,-\left(\frac{p}{x}+cx\right)\sin\sigma+\left(\frac{q}{y}+cy\right)\cos\sigma\right),\] and the profile curve $\gamma$ is its projection on $\QQ$. Throughout, a bi-rotational self-shrinker is the hypersurface generated by a maximal trajectory, closed off smoothly where its profile curve meets an axis (Proposition~\ref{lem:suave}).

The classification of self-shrinkers is a central problem and remains difficult even in codimension one. On the rigidity side, halfspace theorems for properly immersed shrinkers go back to Pigola and Rimoldi \cite{PR}, who assumed a bound on the mean curvature and a contact point at finite distance. Cavalcante and Espinar \cite[Theorems~1.2 and~1.3]{CE} removed these assumptions, inside and outside the solid cylinder, Impera, Pigola and Rimoldi \cite[Theorem~A]{IPR} gave a potential-theoretic proof, and Colding and Minicozzi \cite{CMfrankel} proved a strong Frankel theorem, which implies that every properly embedded shrinker meets every generalized cylinder. On the construction side, Angenent \cite{Angenent} produced the first nontrivial closed rotational shrinker (see also \cite{DN}), Kleene and M\o ller \cite{KM} analyzed the rotational class, and Drugan and Kleene \cite{DrKl} constructed infinitely many complete immersed rotational examples. McGrath \cite{McGrath} constructed closed $SO(n)\times SO(n)$-invariant examples of type $\s^{n-1}\times\s^{n-1}\times\s^1$, and Riedler \cite[Theorem~A]{Riedler} obtained closed embedded examples of type $\s^1\times M_0$ for every isoparametric hypersurface $M_0\subset\s^n$ with equal multiplicities. Chang and Spruck \cite{ChangSpruck} constructed self-shrinking ends asymptotic to the cone over any non-minimal isoparametric hypersurface of the sphere, which are bi-rotational for the products $\s^p(\cos\psi)\times\s^q(\sin\psi)$. Park \cite[Theorem~1.3]{Park-fol} constructed bi-rotational asymptotically conical trumpets and closing disks which, together with the cylinders and the Clifford cone, foliate the complement of a compact region. This foliation describes the far field, whereas the recurrent nonproper dynamics established here occurs in the bounded core. Mean curvature flow near bi-rotational minimal quadratic cones was studied in \cite{Velazquez,Park-Simons}. For $p=q$, the examples of McGrath and the examples of Riedler give closed orbits of \eqref{eq:flow}, and for $p\ne q$ existence remains open. These works do not describe every maximal trajectory of the bi-rotational phase space, in particular the nonproper recurrent ones.

For $c<0$, system \eqref{eq:flow} has the following explicit solutions:

\begin{itemize}
\item[(i)] the sphere $\s^{p+q+1}(r_0),$ for 
\[
r_0=\sqrt{\frac{p+q+1}{-c}}\quad \mbox{and}\quad (x,y,\sigma)=(r_0\cos(r_0^{-1}s),r_0\sin(r_0^{-1}s),\tfrac\pi2+r_0^{-1}s),\quad s\in\left(0,\tfrac{\pi r_0}{2}\right);
\]
\item[(ii)] the cylinders $\mathcal{C}_{\mathrm v}=\s^p(x_*)\times\R^{q+1}$ and $\mathcal{C}_{\mathrm h}=\s^q(y_*)\times \R^{p+1},$ for 
\[
(x,y,\sigma)=(x_*,s,\pi/2)\quad\mbox{and}\quad (x,y,\sigma)=(s,y_*,0),
\]
respectively, where
\[
x_*:=\sqrt{-p/c},\quad y_*:=\sqrt{-q/c};
\]
\item[(iii)] the minimal Clifford cone 
\[
\mathcal{C}=\{(u,v)\in \R^{p+1}\times\R^{q+1};|u|=\sqrt{p/q}|v|\},
\] 
for 
\[(x,y,\sigma)=(s\cos\psi_0,s\sin\psi_0,\psi_0),
\] 
where $\tan\psi_0=\sqrt{q/p}.$
\end{itemize}

Here the Clifford cone means its regular part, and its vertex accounts for the exception in the completeness theorem.

The cylinders $\mathcal C_{\mathrm v}$ and $\mathcal C_{\mathrm h}$ are represented in $\QQ$ by the lines
\[
x=x_* \quad\mbox{and}\quad y=y_*,
\] 
respectively, called separatrices, which meet at the vertex $V=(x_*,y_*)$.

\subsection{The results}

The picture that emerges is the following. The two cylinders organize the whole phase space: every other trajectory crosses both of them, and an end that crosses them only finitely often is asymptotic to a cone. The remaining ends return infinitely often to a bounded region, and this recurrent behavior is the prevalent one. The crossings also control the winding of the profile curve and the location of its self-intersections.

No equivariant self-shrinker other than the two distinguished cylinders avoids them. Intersection with the cylinders is known for properly embedded shrinkers \cite{CMfrankel}, and non-confinement for properly immersed ones \cite{CE,IPR}. Theorem~\ref{thmA} requires neither properness nor embeddedness, and it also gives transversality.

\begin{lettertheorem}[Equivariant halfspace]\label{thmA}
Let $\Sigma$ be a bi-rotational self-shrinker distinct from the two cylinders
\[
\mathcal C_{\mathrm h}=\s^{q}(y_*)\times\R^{p+1},
\qquad
\mathcal C_{\mathrm v}=\s^{p}(x_*)\times\R^{q+1}.
\]
Then $\Sigma$ intersects both, and every intersection is transverse. Equivalently, the only equivariant self-shrinker contained in one of the two closed sides of $\mathcal C_{\mathrm h}$, that is, in the closure of a connected component of $\R^{n+1}\setminus\mathcal C_{\mathrm h}$, is $\mathcal C_{\mathrm h}$ itself, and analogously for $\mathcal C_{\mathrm v}$. 
\end{lettertheorem}

To describe the behavior at infinity, we use polar coordinates in the profile plane
\[
x=\rho\cos\psi,\quad\mbox{and}\quad y=\rho\sin\psi,
\]
for $\rho>0$ and $\psi\in(0,\tfrac\pi2)$.

\begin{lettertheorem}[Completeness and dichotomy of the ends]\label{thmB}
Every bi-rotational self-shrinker $\Sigma$, associated to a maximal solution of \eqref{eq:flow} and distinct from the Clifford cone, is complete. Moreover, if $\Sigma$ is distinct from the cylinders, then $\Sigma$ is compact or each end of $\Sigma$ is of one of two types:
\begin{enumerate}
\item[(a)] asymptotically conical: $\rho\to\infty$, the polar angle converges to a limit $\psi_\infty\in(0,\tfrac\pi2)$, and the tangent becomes radial;
\item[(b)] non-escaping: infinite arclength with $\rho\not\to\infty$.
\end{enumerate}
Moreover, in case (a):
\begin{itemize}
\item[(i)] the end is a nondegenerate graph over each of the two coordinate axes, hence embedded;

\item[(ii)] parametrizing the end toward infinity, both coordinates are eventually strictly increasing, with
\[
\frac{dy}{dx}\longrightarrow\tan\psi_\infty\in(0,\infty),
\quad\mbox{and}\quad
\frac{dx}{dy}\longrightarrow\cot\psi_\infty\in(0,\infty).
\]
\end{itemize}
\end{lettertheorem}

In the rotational case, Kleene and M\o ller \cite[Theorem~3]{KM} proved that every self-shrinking end which is a graph over a half-line of the axis of rotation is either the cylinder or one of their asymptotically conical trumpets, and Park \cite{Park-fol} obtained the bi-rotational analogue. Theorem~\ref{thmB} assumes no graphical structure, properness or embeddedness: the graphical structure in (i) is part of the conclusion, and no end is cylindrical. For general shrinkers in $\R^3$ of finite topology, every end is asymptotic to a regular cone or to a round cylinder \cite{WangAsym}.

\begin{remark}[Which alternative is typical]\label{rem:generic}
Alternative (b) is the complement of (a) among the ends of infinite arclength, and it is the prevalent alternative in the sense of measure: by Theorem~\ref{thmH} below, for $\mathfrak m$-almost every phase point the trajectory is either closed or has two non-escaping ends, and escaping ends form an $\mathfrak m$-null set. Theorem~\ref{thmB} should therefore be read conditionally: an escaping end is asymptotically conical. Whether non-escaping ends are bounded, or can be injective, is not decided here.
\end{remark}

\begin{remark}[The family of escaping ends]\label{rem:escaping}
For each $\psi_\infty\in(0,\tfrac\pi2)$ there is, up to reversal, exactly one maximal trajectory with an escaping end of direction $\psi_\infty$ (Proposition~\ref{prop:fins}): the Clifford ray for $\psi_0$, and the continuation of an end of Chang and Spruck \cite{ChangSpruck} otherwise, which is a complete boundaryless self-shrinker, not necessarily proper (Corollary~\ref{cor:completion}). The escaping ends therefore form a one-parameter family, parametrized by their limit direction, and a trajectory with two escaping ends of distinct directions is assigned to both. Uniqueness follows from an elementary argument for the profile equation, analogous to the uniqueness theorem of Wang \cite{Wang}.
\end{remark}

A crossing is an intersection of $\gamma$ with a separatrix. Let $N_y$ and $N_x$ be the numbers of crossings of $\{y=y_*\}$ and $\{x=x_*\}$, possibly infinite (over one period, for a closed orbit), and $N_y[a,b]$, $N_x[a,b]$ those at instants $s\in(a,b)$. The crossings divide $\gamma$ into excursions, terminal when they contain an endpoint of $\gamma$ and intermediate otherwise. 

Along a subarc $\gamma|_{[a,b]}$, we write $\Delta\sigma:=\sigma(b)-\sigma(a)$ for the variation of the tangent angle, with $\sigma$ taken as a continuous real-valued function along $\gamma$, and we set $n[a,b]:=\lfloor|\Delta\sigma|/2\pi\rfloor$, and $n(\gamma)$ denotes this quantity over one period of a closed orbit, or between the endpoints when the tangent angle has limits there (Definition~\ref{def:winding}).

A closed orbit returns to the same phase point and generates a compact hypersurface. An apparent loop returns to the same position with a distinct tangent (Definition~\ref{def:lacos}), and its simple subarc is a loop (Definition~\ref{def:loop}). Apparent loops occur because a trajectory, which never self-intersects in $\QQ\times\s^1$, may have a self-intersecting projection. The parametrization \eqref{eq:param} is injective if and only if $\gamma$ is simple, and an embedding if, in addition, it is proper.

\begin{lettertheorem}[Winding and crossings]\label{thmC}
Let $\gamma$ be the profile curve of a bi-rotational self-shrinker $\Sigma$, with $N_x$ and $N_y$ crossings of the separatrices $\{x=x_*\}$ and $\{y=y_*\}$, respectively. Then,
\begin{enumerate}
\item[(i)] whenever the global winding number $n(\gamma)$ is defined (for a closed orbit, $n(\gamma)$, $N_x$ and $N_y$ are computed over one period),
\begin{equation}\label{eq:cont-intro}
n(\gamma)\ \le\ \min(N_x,N_y)-1.
\end{equation}
In particular $\min(N_x,N_y)\le1$ forbids a single full turn;
\item[(ii)] between two consecutive crossings of a separatrix the tangent sweeps less than a full turn (i.e., $|\Delta\sigma|<2\pi$), and along a terminal excursion less than half a turn (i.e., $|\Delta\sigma|<\pi$);
\item[(iii)] every non-escaping end crosses both separatrices infinitely often and, along every compact subarc $\gamma|_{[a,b]}$, the number of full turns of the tangent satisfies
\[
n[a,b]\le\min(N_x[a,b],N_y[a,b]).
\]
\end{enumerate}
\end{lettertheorem}

\begin{remark}[Scope of Theorem~\ref{thmC}]\label{rem:escopoC}

Item (i) applies to closed orbits and to trajectories whose tangent angle has limits at both endpoints, that is, to compact and to proper examples. For the recurrent trajectories of Theorem~\ref{thmH} only item (iii) is available. We do not know whether the bound \eqref{eq:cont-intro} is attained.

\end{remark}

Theorems~\ref{thmA}, \ref{thmB} and \ref{thmC} combine into the following dichotomy: a complete noncompact bi-rotational self-shrinker other than the cylinders either has only conical ends or meets both cylinders infinitely often. The dichotomy coincides with the standard finiteness invariants of the theory.

\begin{lettercorollary}[Tame or wild]\label{thmD}
Let $\Sigma$ be a complete, noncompact bi-rotational self-shrinker, distinct from the two cylinders. Then the following are equivalent:
\begin{enumerate}
\item[(a)] every end of $\Sigma$ is asymptotically
conical, and then each is embedded;
\item[(b)] $\Sigma$ is proper;
\item[(c)] $\Sigma$ has polynomial volume growth;
\item[(d)] $\Sigma$ has finite entropy;
\item[(e)] $\Sigma$ meets each of the two cylinders in finitely many connected components.
\end{enumerate}
When they fail, $\Sigma$ has a non-escaping end, is not proper, has infinite entropy, and meets both cylinders transversely in infinitely many connected components. 

In particular, a properly immersed complete noncompact bi-rotational self-shrinker has only asymptotically conical embedded ends.
\end{lettercorollary}

The equivalence (b)$\Leftrightarrow$(c)$\Leftrightarrow$(d) is a particular case of the theorem of Cheng and Zhou \cite[Theorem~1.3]{ChengZhou}, valid for every complete self-shrinker, combined with the standard relation between Gaussian area and entropy \cite[\S7.2]{CM}. The content specific to the bi-rotational class is the equivalence of (a), (b) and (e), whose nontrivial direction (e)$\Rightarrow$(a) rests on Theorem~\ref{prop:recorrente}.

The self-intersections of $\gamma$ are controlled by the same combinatorics and depend on the horizontal strips
\[
\mathcal S^+=\{(x,y)\in\QQ;y>y_*\}\quad\mbox{and}\quad \mathcal S^-=\{(x,y)\in\QQ;0<y<y_*\},
\]
and the vertical strips
\[
\mathcal V^+=\{(x,y)\in\QQ;x>x_*\}\quad\mbox{and}\quad \mathcal V^-=\{(x,y)\in\QQ;0<x<x_*\},
\]
whose intersections are the four quadrants.

\begin{lettertheorem}[Immersion versus embedding]\label{thmE}
Let $\Sigma$ be a bi-rotational self-shrinker distinct from the cylinders and $\gamma$ its profile curve. Then,
\begin{enumerate}
\item[(i)] any loop of $\gamma$ confined to a quadrant has $\pi<|\Delta\sigma|<\tfrac{3\pi}{2}$;
\item[(ii)] no two loops of $\gamma$ lie in the same excursion of a strip, and \emph{a fortiori} not in the same excursion within a quadrant;
\item[(iii)] no loop of $\gamma$ lies in a terminal excursion, whence a loop confined to a horizontal strip forces $N_y\ge2$, a loop confined to a vertical strip forces $N_x\ge2$, a quadrant loop forces $\min(N_x,N_y)\ge2$, and the numbers of loops confined to horizontal and to vertical strips are at most $N_y-1$ and $N_x-1$, respectively ($N_y$ and $N_x$ if $\gamma$ is closed).
\end{enumerate}
\end{lettertheorem}

The round sphere $\s^{p+q+1}(r_0)$, $r_0=\sqrt{(p+q+1)/|c|}$, also organizes the trajectories. Its profile arc $\{\rho=r_0\}$ separates the region where $\rho$ can only have local minima from the region where it can only have local maxima.

\begin{lettertheorem}[The sphere as a third separatrix]\label{thmF}
Let $\Sigma$ be a bi-rotational self-shrinker distinct from the round sphere, $\gamma$ be its profile curve, and let $\rho=|\gamma|$. Then $\gamma$ meets the circle $\{\rho=r_0\}$ transversely, every critical point of $\rho$ is nondegenerate, every local minimum of $\rho$ lies inside the sphere and every local maximum outside it, and between two consecutive crossings of the sphere, $\rho$ has exactly one critical point. In particular:
\begin{enumerate}
\item[(i)] the crossings of the sphere and the critical points of $\rho$ alternate along $\gamma$;
\item[(ii)] every loop of $\gamma$ contained in the open ball bounded by the sphere contains exactly one critical point of $\rho$, a minimum, and every loop outside the closed ball contains exactly one, a maximum;
\item[(iii)] if $\Sigma\cong\s^{1}\times\s^{p}\times\s^{q}$, it meets the sphere in an even number $\ge2$ of connected components.
\end{enumerate}
\end{lettertheorem}

The topology of $\Sigma$ is determined by the endpoints of its profile curve. Each endpoint on an axis closes off one factor of $\s^{p}\times\s^{q}$ smoothly (Proposition~\ref{lem:suave}), a closed orbit contributes a circle factor, and the remaining endpoints are ends of infinite length. This gives the following list, together with two constraints on the compact examples.

\begin{lettertheorem}[Topology]\label{thmG}
Every bi-rotational self-shrinker $\Sigma$ is diffeomorphic to one of $\s^{1}\times\s^{p}\times\s^{q}$, $\s^{p+q+1}$, $\s^{q+1}\times\s^{p}$, $\s^{p+1}\times\s^{q}$ in the compact case, the last two coinciding when $p=q$, and to one of $\R\times\s^{p}\times\s^{q}$, $\R^{p+1}\times\s^{q}$, $\R^{q+1}\times\s^{p}$ otherwise, according to the number of endpoints on the axes. Moreover,
\begin{enumerate}
\item[(i)] every compact $\Sigma$ encircles the vertex: its profile curve crosses both separatrices, so that $V$ lies in the interior of the rectangle $[x_{\min},x_{\max}]\times[y_{\min},y_{\max}]$, where $x_{\min}$, $x_{\max}$, $y_{\min}$ and $y_{\max}$ are the extreme values of the coordinates along the profile curve;
\item[(ii)] if $\Sigma\cong\s^{1}\times\s^{p}\times\s^{q}$, each separatrix is crossed an even number $\ge2$ of times.
\end{enumerate}
\end{lettertheorem}

The previous theorems describe the possible behaviors of a single trajectory. The last one describes the typical trajectory. Its main tool is an invariant measure of finite total mass on the phase space, which makes the Poincaré recurrence theorem available. Let
\[
\mathfrak m:=\Phi dx dy d\sigma
\]
on the phase space $\mathcal Q\times\s^1$, where
\begin{equation}\label{geodesic}
\Phi:=x^{p}y^{q}e^{c(x^{2}+y^{2})/2}.
\end{equation}

\begin{lettertheorem}[Almost every trajectory is recurrent]\label{thmH}
The flow of \eqref{eq:flow} preserves $\mathfrak m$, and $\mathfrak m(\mathcal Q\times\s^1)<\infty$. For $\mathfrak m$-almost every initial condition $v=(x_0,y_0,\sigma_0)\in\mathcal Q\times\s^1$, the solution through $v$ is defined for all $s\in\R$ and is recurrent in both time directions. Consequently, for $\mathfrak m$-almost every $v$, the hypersurface $\Sigma_v$ generated by the trajectory of $v$ satisfies exactly one of:
\begin{enumerate}
\item[(i)] its profile curve is closed and $\Sigma_v\cong\s^{1}\times\s^{p}\times\s^{q}$;
\item[(ii)] $\Sigma_v$ is complete, noncompact, not proper, and of infinite entropy, and both of its ends are non-escaping, each meeting both cylinders transversely in infinitely many connected components.
\end{enumerate}
Moreover, 
\begin{enumerate}
    \item[(a)] the initial conditions whose trajectories reach an axis or the origin, or have an asymptotically conical end, form an $\mathfrak m$-null set;
    \item[(b)] the union of the closed orbits is either $\mathfrak m$-null or contains a nonempty open set.
 \end{enumerate}
\end{lettertheorem}

Whether the second alternative for the closed orbits can occur is left open. If it does not occur, (ii) holds for $\mathfrak m$-almost every initial phase point. Complete nonproper self-shrinkers were already known: Halldorsson \cite[Theorem~5.1]{Halldorsson} showed that a shrinking plane curve other than a line through the origin or the circle is either closed \cite{AbreschLanger} or dense in an annulus, and for a dense curve $\Gamma$ the product $\Gamma\times\R^{n-1}$ is a complete self-shrinker without polynomial volume growth \cite[Proposition~1.1]{ChengOgata}.

\subsection{The method}\label{sec:metodo}

Equation \eqref{lambda-s} is the geodesic equation of the conformal metric
\[
g_\Phi=\Phi^{2}\left(dx^{2}+dy^{2}\right),
\]
with $\Phi$ as in \eqref{geodesic}. For $c<0$, the potential $u=\log\Phi$ is strictly concave with maximum at $V$, and $\Hess u(V)=2c\,\mathrm{Id}$. If $r$ denotes the Euclidean distance to $V$, its Taylor expansion gives $u=u(V)+cr^{2}+O(r^{3})$. This suggests the auxiliary rotationally symmetric metric
\[
g_{\mathrm{rad}}=\Phi(V)^{2}e^{2cr^{2}}(dr^{2}+r^{2}d\theta^{2}),
\]
whose Clairaut integral is $\mathcal I(r)\sin\vartheta$, where $\vartheta$ is the angle with the meridians and $\mathcal I(r)=\Phi(V)r e^{cr^{2}}$. Since $\mathcal I(r)$ tends to zero both as $r\to0$ and as $r\to\infty$, geodesics of this auxiliary model with nonzero Clairaut constant remain in an annulus, whereas only meridians can reach its radial ends. The auxiliary model thus gives geometric intuition for confinement, and for $g_\Phi$ itself the corresponding statement is the recurrence of Theorem~\ref{thmH}.

In the exact system, the role of the Clairaut integral is played by two piecewise monotone quantities. Writing $\Phi=f(x)g(y)$, with
\[
f(x)=x^{p}e^{cx^{2}/2}
\quad\mbox{and}\quad
g(y)=y^{q}e^{cy^{2}/2},
\]
they are
\[
I=f(x)\sin\sigma
\quad\mbox{and}\quad
J=g(y)\cos\sigma,
\]
linked by the mirror symmetry. Their derivatives along the flow are
\begin{equation}\label{eq:cruzadas}
I'=B(y) f(x)\cos^{2}\sigma
\quad
\mbox{and}
\quad
J'=A(x) g(y)\sin^{2}\sigma,
\end{equation}
where
\[
A(x)=cx+\frac{p}{x} \qquad \mbox{and}\qquad B(y)=cy+\frac{q}{y}.
\]
Thus $I$ and $J$ are monotone in the horizontal and vertical strips, respectively, and the signs of their derivatives change at the separatrices. They are not conserved, but the bounds $|I|\le f(x)$ and $|J|\le g(y)$, with $f$ and $g$ vanishing at both ends of $(0,\infty)$, reproduce the turning mechanism.

The functions $I$ and $J$ are modeled on the monotone functions introduced by Hsiang, Teng and Yu \cite[Proposition~3]{HTY} and used by Hsiang \cite[Proposition~4]{Hsiang} for the three-dimensional system of bi-rotational hypersurfaces of constant mean curvature. In that system the terms $-cx\sin\sigma+cy\cos\sigma$ are replaced by a constant, and no sign change occurs. The sign changes of $A$ and $B$, which occur only for $c<0$, produce simultaneously the two cylinders, the winding, and the recurrent trajectories. The case $c>0$ is treated in the companion paper \cite{SNexpanders}.

A third quantity, $\mathcal F=\Phi\lan\gamma,\gamma'\ran$, satisfies $\mathcal F'=(p+q+1+c\rho^{2})\Phi$ (Lemma~\ref{lem:terceira}). It increases inside the round sphere $\rho^{2}<(p+q+1)/|c|$ and decreases outside it, so the sphere is a third separatrix for the radial motion (Theorem~\ref{thmF}). It also governs the behavior near the origin (Proposition~\ref{lem:origem}).

\subsection{Organization}
Section~\ref{sec:prelim} sets up the profile system, the functions $I$ and $J$, the four regions and the behavior at the axes. Sections~\ref{sec:winding} and~\ref{sec:excursoes} prove Theorems~\ref{thmE} and~\ref{thmC}, Section~\ref{sec:cross} proves Theorems~\ref{thmA} and~\ref{thmB}, Section~\ref{sec:completas} proves Corollary~\ref{thmD} and Theorems~\ref{thmF} and~\ref{thmG}, and Section~\ref{sec:generico} proves Theorem~\ref{thmH}.

The statements of the introduction are proved as follows.

\begin{itemize}\setlength{\itemsep}{1pt}
\item[]\textbf{Theorem~\ref{thmA}} is Theorem~\ref{thm:cross}, for the horizontal cylinder, Corollary~\ref{cor:crossI}, for the vertical one, and Lemma~\ref{lem:transv}, for the
transversality.
\item[]\textbf{Theorem~\ref{thmB}.} The completeness is
Theorem~\ref{thm:completude}. Item (a) is Theorem~\ref{cor:conico-direto}. The exclusion of cylindrical and
coordinate ends is given by Propositions~\ref{lem:fuga} and \ref{prop:nao-eixo}. The embeddedness is Corollary~\ref{cor:end-mergulhado} and the graph property is proven in Corollary~\ref{cor:grafico}. Along an asymptotically conical end one also has $C\rho^{-2}\le|\mathcal A|^{2}=O(\rho^{-2})$ for some $C>0$ (Proposition~\ref{prop:decaimento}).
\item[]\textbf{Theorem~\ref{thmC}.} Item \textup{(i)} is Theorem~\ref{prop:contagem}. Item
\textup{(ii)} is given by Lemmas~\ref{lem:excursao} and \ref{lem:excterminal} and item \textup{(iii)} is Theorem~\ref{prop:recorrente}, together with Corollary~\ref{cor:porcao}.
\item[]\textbf{Corollary~\ref{thmD}.} The equivalences (a)$\Leftrightarrow$(b)$\Leftrightarrow$(e) are given by
Corollary~\ref{cor:dicotomia} and the equivalences
(b)$\Leftrightarrow$(c)$\Leftrightarrow$(d) are proven in Corollary~\ref{cor:entropia}.
\item[]\textbf{Theorem~\ref{thmE}.} Item (i) is Proposition~\ref{thm:cota32}, item (ii) is Theorem~\ref{thm:unico} and item (iii) is
Proposition~\ref{thm:terminal}, with Corollaries~\ref{cor:localiza} and \ref{cor:superpos}.
\item[]\textbf{Theorem~\ref{thmF}} is Theorem~\ref{thm:esfera}.
\item[]\textbf{Theorem~\ref{thmG}.} The diffeomorphism class is given by Proposition~\ref{prop:topcompacta} and Theorem~\ref{thm:completas}. Item \textup{(i)} is
Corollary~\ref{thm:compacta-vertice} and the parity statement in item (ii) is Corollary~\ref{cor:noloop2}.
\item[]\textbf{Theorem~\ref{thmH}} is Theorem~\ref{thm:generico}, with Proposition~\ref{prop:fechadas}.
\end{itemize}

\subsection*{Acknowledgements}
The author was partially supported by CNPq, Brazil, grant number 306189/2023-2.

\subsection*{AI statement}
Generative AI tools (Claude, Anthropic, and ChatGPT, OpenAI) were used to assist with language editing, organization of the exposition, numerical experiments, suggestions for some arguments, and the discussion and checking of selected calculations. All AI-assisted material was reviewed and edited by the author, who takes full responsibility for the mathematical statements, proofs, and references.

\subsection*{Data availability statement}
No data are associated with this manuscript. The figures that display computed trajectories are based on numerical solutions of \eqref{eq:flow}. These computations are not certified, and no proof relies on them.

\subsection*{Conflict of interest statement}
The author declares that he has no conflict of interest.

\section{The profile system}\label{sec:prelim}

This section collects the tools used throughout the paper: the system \eqref{eq:flow} and its symmetries, the functions $I$ and $J$ and the signs of their derivatives, the partition of $\mathcal Q$ by the two separatrices, and the behavior of a trajectory that approaches a coordinate axis.

\subsection{The system and its symmetries}\label{sec:sistema}

We parametrize $\gamma$ by arclength and write $\sigma$ for the tangent angle, so that
$x'=\cos\sigma$, $y'=\sin\sigma$. The unit normal is $N=(-\sin\sigma,\cos\sigma)$ in the orbit
plane.

\begin{lemma}[The system]\label{prop:sistema}
The hypersurface \eqref{eq:param} satisfies \eqref{eq:soliton} if and only if its profile curve $\gamma(s)=(x(s),y(s))$, parametrized by arclength, satisfies \eqref{lambda-s} and the three-dimensional curve $\Gamma(s)=(x(s),y(s),\sigma(s))$ is a solution of the system \eqref{eq:flow} on the domain $\mathcal Q\times\s^1$.
\end{lemma}
\begin{proof}
The principal curvatures of \eqref{eq:param} are $\kappa_0=\sigma'$ (the curvature of the profile curve), plus $p$ copies of $\kappa_1=\tfrac{\sin\sigma}{x}$ (from the directions tangent to $\s^{p}(x)$) and $q$ copies of $\kappa_2=-\tfrac{\cos\sigma}{y}$ (from the tangents to $\s^{q}(y)$). Hence
\[
H=\sigma'+p\frac{\sin\sigma}{x}-q\frac{\cos\sigma}{y}.
\]
On the other hand, $\langle X,N\rangle=-x\sin\sigma+y\cos\sigma$, so that \eqref{eq:soliton},
which in scalar form is $H=c\langle X,N\rangle$, is equivalent to
\[
\sigma'+p\frac{\sin\sigma}{x}-q\frac{\cos\sigma}{y}=c(-x\sin\sigma+y\cos\sigma),
\]
i.e.
\[
\sigma'=-\sin\sigma\left(\frac{p}{x}+cx\right)+\cos\sigma\left(\frac{q}{y}+cy\right),
\]
which is \eqref{eq:flow}.
\end{proof}

We also record the norm of the second fundamental form in terms of the profile curve, which is used in the analysis of the ends.

\begin{lemma}[Second fundamental form in terms of the profile curve]\label{prop:2forma}
The squared norm $|\mathcal A|^{2}$ of the second fundamental form of a bi-rotational hypersurface $\Sigma^n$ has the expression
\begin{equation}\label{eq:A2}
|\mathcal A|^{2}=\kappa_0^2+p\kappa_1^2+q\kappa_2^2
=(\sigma')^{2}+p\frac{\sin^{2}\sigma}{x^{2}}+q\frac{\cos^{2}\sigma}{y^{2}}.
\end{equation}
Along the horizontal and the vertical cylinders, $|\mathcal A|^{2}\equiv|c|$  is constant. \end{lemma}
\begin{proof}
The curvatures are those of Lemma~\ref{prop:sistema} and \eqref{eq:A2} is the sum of the squares with the multiplicities. On the horizontal cylinder, $\sigma\equiv0$ and $y\equiv y_*$, which gives $\kappa_0=0$, $\kappa_1=0$, and $\kappa_2=-1/y_*$. Thus, $|\mathcal A|^{2}=q/y_*^2$ and
$y_*^2=-q/c$ implies $|\mathcal A|^{2}=-c$. The case of the vertical cylinder is analogous.
\end{proof}

The field
\[
F(x,y,\sigma)= \left(\cos\sigma,\sin\sigma,-\left(\frac{p}{x}+cx\right)\sin\sigma+\left(\frac{q}{y}+cy\right)\cos\sigma\right)
\]
associated with the system \eqref{eq:flow} (i.e., $\frac{d}{ds}(x,y,\sigma)=F(x,y,\sigma)$) is smooth on $\mathcal Q\times\s^1$, with uniqueness of solutions. 

Every solution extends to a maximal one, defined on an open interval $(s_-,s_+)$ with $-\infty\le s_-<s_+\le\infty$. We call $\alpha$ and $\omega$ the two endpoints of the trajectory, corresponding to $s\to s_-$ and to $s\to s_+$, respectively. They are ends of the parameter interval, not necessarily points of the plane, and a statement at $\alpha$ or at $\omega$ refers to the limit $s\to s_-$ or $s\to s_+$. When the profile curve converges as $s\to s_\pm$, the same letter also denotes its limit point. The notation is borrowed from the $\alpha$- and $\omega$-limit sets of dynamical systems, which describe the same two directions of time.

The field $F$ has two symmetries we will use repeatedly.

\begin{lemma}[Reversibility and mirror]\label{lem:simetrias}
System \eqref{eq:flow} is invariant under:
\begin{enumerate}
\item[(i)] \emph{(reversibility)} $s\mapsto-s$, $\sigma\mapsto\sigma+\pi$, which preserves the
image of $\gamma$ and swaps the endpoints $\alpha$ and $\omega$;
\item[(ii)] \emph{(mirror)} $(x,y,\sigma;p,q)\mapsto(y,x,\tfrac\pi2-\sigma;q,p)$, which swaps
the roles of the two axes.
\end{enumerate}
\end{lemma}
\begin{proof}
(i) Under $s\mapsto-s$ and $\sigma\mapsto\sigma+\pi$: $\cos\sigma\mapsto-\cos\sigma$ and
$\sin\sigma\mapsto-\sin\sigma$, so that $x'=\cos\sigma$ and $y'=\sin\sigma$ are preserved (two
signs that cancel), and the right-hand side of $\sigma'$ changes sign, as does $\sigma'$ itself.
(ii) Under the swap, $\cos\sigma\leftrightarrow\sin\sigma$ and $A\leftrightarrow B$, which
permutes the first two equations and preserves the third up to sign, absorbed by
$\sigma\mapsto\tfrac\pi2-\sigma$.
\end{proof}

Since the first two components of the field form a unit vector, the flow has no stationary points.

\begin{lemma}[Absence of equilibria]\label{lem:noeq}
The field \eqref{eq:flow} has no equilibrium points in $\mathcal Q\times\s^1$.
\end{lemma}
\begin{proof}
Indeed, since $(x')^2+(y')^2=\cos^2\sigma+\sin^2\sigma=1\ne0,$ $x'$ and $y'$ cannot vanish simultaneously.
\end{proof}

\subsection{The functions \texorpdfstring{$I$}{I} and \texorpdfstring{$J$}{J}}\label{sec:duas}

This subsection contains the engine of the whole text. We introduce the two functions, compute their derivatives along the flow, and read off the sign diagram. As introduced above, we write
\[
A(x)=cx+\frac{p}{x}\quad\mbox{and}\quad B(y)=cy+\frac{q}{y},
\]
and so, the third equation of system \eqref{eq:flow} becomes
\[
\sigma'=-A(x)\sin\sigma + B(y)\cos\sigma.
\]
The functions $I$ and $J$ are motivated by the two planar subsystems of \eqref{eq:flow}. If $y$ is frozen, the pair $(\sigma,x)$ evolves by the field $W_1=(-A(x)\sin\sigma,\cos\sigma)$, and $f(x)=x^{p}e^{cx^{2}/2}$ is an integrating factor of $W_1$: since $f'=Af$,
\[
\dive(fW_1)=\partial_\sigma\left(-fA\sin\sigma\right)+\partial_x\left(f\cos\sigma\right)=\cos\sigma\,(f'-Af)=0,
\]
and the corresponding first integral is $f(x)\sin\sigma$. In the same way, $g(y)=y^{q}e^{cy^{2}/2}$ satisfies $g'=Bg$ and is an integrating factor of the field $W_2=(B(y)\cos\sigma,\sin\sigma)$ of the $(\sigma,y)$ subsystem, with first integral $g(y)\cos\sigma$. For the full system these quantities are no longer conserved, but the signs of their derivatives are controlled (Lemma~\ref{prop:dIdJ}). This motivates the following

\begin{definition}[The functions $I$ and $J$]\label{def:IJ}
Along a trajectory of \eqref{eq:flow}, we define
\begin{equation}\label{eq:defIJ}
\left\{\begin{aligned}    
I(s)&:=f(x(s))\sin(\sigma(s))=(x(s))^{p}e^{c(x(s))^{2}/2}\sin(\sigma(s)),\\
J(s)&:=g(y(s))\cos(\sigma(s))=(y(s))^{q}e^{c(y(s))^{2}/2}\cos(\sigma(s)).
\end{aligned}\right.
\end{equation}
\end{definition}

The following elementary relation between $I$ and $J$ has some fundamental consequences.

\begin{lemma}[Pythagorean identity]\label{prop:pit}
At every point of $\mathcal Q\times\s^1$,
\begin{equation}\label{eq:pit}
\frac{I^{2}}{f(x)^{2}}+\frac{J^{2}}{g(y)^{2}}=\sin^{2}\sigma+\cos^{2}\sigma=1 .
\end{equation}
In particular,
\begin{itemize}
\item[(i)] $|I|\le f(x)$ and $|J|\le g(y)$;
\item[(ii)] $I$ and $J$ never vanish simultaneously;
\item[(iii)] $I=0$ forces $\cos\sigma=\pm1$ (horizontal tangent) and $J=0$ forces $\sin\sigma=\pm1$ (vertical tangent).
\end{itemize}
\end{lemma}
\begin{proof}
Immediate from \eqref{eq:defIJ}.
\end{proof}

The next lemma gives the derivatives of $I$ and $J$ along the trajectories of \eqref{eq:flow}.

\begin{lemma}[Fundamental identities]\label{prop:dIdJ}
Along the trajectories of \eqref{eq:flow}, we have
\begin{equation}\label{eq:dJ}
\left\{\begin{aligned}
\frac{dI}{ds}&=B(y)f(x)\cos^2\sigma,\\
\frac{dJ}{ds}&=A(x)g(y)\sin^2\sigma.
\end{aligned}\right.
\end{equation}
\end{lemma}
\begin{proof}
From $f'=Af$, $x'=\cos\sigma$, and $\sigma'=-A\sin\sigma+B\cos\sigma$, we have
\[
\begin{aligned}
\frac{dI}{ds}&=f'(x)x'\sin\sigma+f(x)(\cos\sigma)\sigma'\\
&=f[A\cos\sigma\sin\sigma+(\cos\sigma)\sigma']\\
&= f[A\cos\sigma\sin\sigma+\cos\sigma(-A\sin\sigma+B\cos\sigma)]\\
&=Bf\cos^2 \sigma.
\end{aligned}
\]
Analogously, from $g'=Bg$ and $y'=\sin\sigma$,
\[
\begin{aligned}
\frac{dJ}{ds}
&=g'(y)y'\cos\sigma-g(y)\sin\sigma\sigma'\\
&=g\sin\sigma[B\cos\sigma-\sigma']\\
&=g\sin\sigma[A\sin\sigma]\\
&=Ag\sin^{2}\sigma,
\end{aligned}
\]
where we used
\[
B\cos\sigma-\sigma'=B\cos\sigma-(-A\sin\sigma+B\cos\sigma)=A\sin\sigma.
\]
\end{proof}

The following elementary result will play a fundamental role in the study of the monotonicity of the functions $I$ and $J$:

\begin{lemma}\label{lem:bound}
Let $c<0$, $f(x)=x^{p}e^{cx^{2}/2}$ and $g(y)=y^{q}e^{cy^{2}/2}$. Then $f$ increases on $(0,x_*)$ and decreases on $(x_*,\infty)$, with global
maximum
\[
M_f:=f(x_*)=\left(\frac{-p}{c}\right)^{p/2}e^{-p/2}.
\]
Hence, $f(x)\to0$ as $x\to0^{+}$ and as $x\to\infty$, and $|I|\le M_f$ along every trajectory. Symmetrically for $g$, with $M_g=g(y_*)=(-q/c)^{q/2}e^{-q/2}$ and $|J|\le M_g$.
\end{lemma}
\begin{proof}
$f'=A(x)f=\tfrac{p+cx^{2}}{x}f$ with $f>0$, so the sign of $f'$ is that of $p+cx^{2}$, which is
positive on $(0,x_*)$ and negative on $(x_*,\infty)$, with
\[
f(x_*)=\left(\frac{-p}{c}\right)^{p/2}e^{c(-p/c)/2}=\left(\frac{-p}{c}\right)^{p/2}e^{-p/2}.
\]
The limits come from the fact that $f\sim x^{p}\to0$ as $x\to0^{+}$, and $\log f=p\log x-\tfrac{|c|}{2}x^{2}\to-\infty$
as $x\to\infty$. The bounds on $I$ and $J$ follow from
$|\sin\sigma|,|\cos\sigma|\le1$. The weights $f$ and $g$ are shown in Figure~\ref{fig:gauss}.
\end{proof}

\begin{figure}[ht]
\centering
\includegraphics[width=\textwidth]{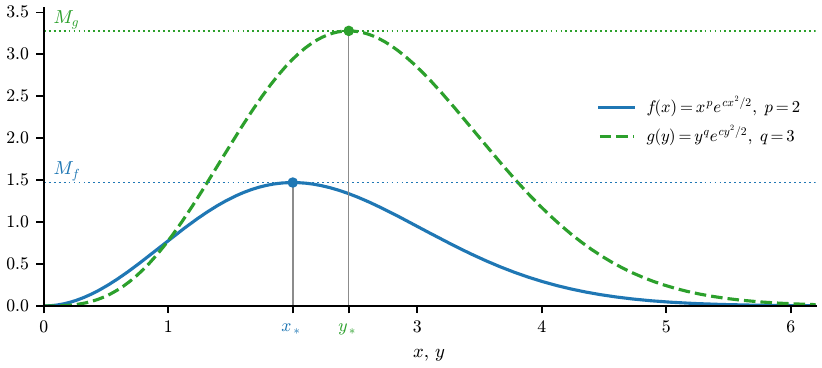}
\caption{The weights $f(x)=x^{p}e^{cx^{2}/2}$ and $g(y)=y^{q}e^{cy^{2}/2}$ for $c<0$, here with
$p=2$, $q=3$ and $c=-\tfrac12$: each has an interior maximum, at $x_*$ and at $y_*$, and tends
to $0$ at both ends. Hence the uniform bounds $|I|\le M_f$ and $|J|\le M_g$ of
Lemma~\ref{lem:bound}, which are what confines the trajectories.}
\label{fig:gauss}
\end{figure}

\subsection{The four regions and the two cylinders}\label{sec:regioes}

\begin{definition}[Partition of the phase space]\label{def:regioes}
For $c<0$, the lines $\{x=x_*\}$ and $\{y=y_*\}$ split $\mathcal Q$ into four open regions, which we denote according to the signs of $(A,B)$:
\[
\begin{array}{lcccc}
\text{region} & \text{definition} & A & B & (I',J')\\[2pt]
\hline
\rule{0pt}{2.6ex}%
\mathrm{Q_{++}} & x<x_*,\ y<y_* & >0 & >0 & (\ge0,\ \ge0)\\[2pt]
\mathrm{Q_{-+}} & x>x_*,\ y<y_* & <0 & >0 & (\ge0,\ \le0)\\[2pt]
\mathrm{Q_{+-}} & x<x_*,\ y>y_* & >0 & <0 & (\le0,\ \ge0)\\[2pt]
\mathrm{Q_{--}} & x>x_*,\ y>y_* & <0 & <0 & (\le0,\ \le0).
\end{array}
\]
We also write
\[
\mathcal S^{+}:=\{y>y_*\}\quad\mbox{and}\quad \mathcal S^{-}:=\{0<y<y_*\}
\]
for the two horizontal strips determined by $y_*$ alone (a union of two quadrants each), and similarly 
\[
\mathcal V^{+}:=\{x>x_*\}\quad\mbox{and}\quad\mathcal V^{-}:=\{0<x<x_*\}
\]
for the vertical strips. The superscript records the sign of $y-y_*$ (resp. $x-x_*$), and therefore the sign opposite to that of $B$ (resp.\ $A$).
\end{definition}

The signs of $A$ and $B$ in the four quadrants determine the monotonicity of $I$ and $J$.

\begin{lemma}[Sign diagram]\label{prop:sinais}
In each of the four regions, $I$ and $J$ are monotone (in the weak sense) with the signs above.
Moreover, $I'=0$ exactly on $\{y=y_*\}\cup Z_I$, where $Z_I:=\{\cos\sigma=0\}$, and $J'=0$
exactly on $\{x=x_*\}\cup Z_J$, where $Z_J:=\{\sin\sigma=0\}$.
\end{lemma}
\begin{proof}
Direct from \eqref{eq:dJ}, since $f,g>0$ and the zeros of $A$ and $B$ are $x_*$ and $y_*$. The resulting sign diagram is shown in Figure~\ref{fig:quatro}.
\end{proof}

\begin{figure}[ht]
\centering
\includegraphics[width=0.62\textwidth]{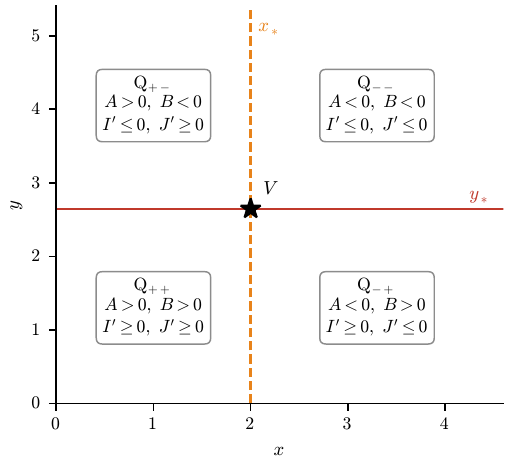}
\caption{The four regions determined by $x=x_*$ and $y=y_*$ ($c<0$), with the signs of $I'$ and $J'$ in each, the vertex $V$ and the two cylinders (horizontal in red, vertical dashed in
orange). This sign diagram underlies the monotonicity arguments used throughout the paper.}
\label{fig:quatro}
\end{figure}

The next result states that the trajectories $\{y=y_*\}$ and $\{x=x_*\}$, corresponding to the horizontal and vertical cylinders, respectively, are the only trajectories of the system \eqref{eq:flow} along which $I$ or $J$ is constant.

\begin{proposition}[The two cylinders]\label{prop:cyl}
Let $c<0$. Then:
\begin{enumerate}
\item[(a)] $\{y=y_*,\ \sigma=0\}$ is a trajectory of \eqref{eq:flow}, with image
$\s^{q}(y_*)\times\R^{p+1}$, along which $I\equiv0$ and $J\equiv M_g$;
\item[(b)] $\{x=x_*,\ \sigma=\tfrac\pi2\}$ is a trajectory of \eqref{eq:flow}, with image
$\s^{p}(x_*)\times\R^{q+1}$, along which $I\equiv M_f$ and $J\equiv0$.
\end{enumerate}
These are the only trajectories along which $I$ or $J$ is constant.
\end{proposition}
\begin{proof}
The verification of items (a) and (b) is immediate. To prove the uniqueness, if $I\equiv$ const., then $B(y)\cos^{2}\sigma\equiv0$ along the trajectory, by \eqref{eq:dJ}. If $\cos\sigma\equiv0$ on an interval, then $x'\equiv0.$ Thus, by the arclength equation $(x')^2+(y')^2=1$, we have $(x,y)=(d_0,s+d_1)$ for some constants $d_0,d_1\in\R$ and $\gamma$ is a vertical line. This implies that $\sigma\equiv\pi/2$ and $\sigma'\equiv\mp A(x)$ must vanish, forcing $x\equiv x_*$: this is (b). Otherwise $B(y)=0$ on a dense set, hence $y\equiv y_*$ and
$y'=\sin\sigma\equiv0$: this is (a). The case of $J$ is the mirror image.
\end{proof}

\begin{remark}[The vertex]\label{rem:vertex}
Let $c<0$. At the vertex $V=(x_*,y_*)$ one has $A(x_*)=B(y_*)=0$, and therefore $\sigma'=0$. It is the only point of $\mathcal Q$ at which $\sigma'=0$ for every tangent direction $\sigma$. Both cylinders pass through $V$, the horizontal one with $\sigma=0$ and the vertical one with $\sigma=\tfrac\pi2$. In particular, the trajectory of the vertical cylinder crosses the level $\{y=y_*\}$ at $V$, and therefore is contained neither in $\mathcal S^{+}$ nor in $\mathcal S^{-}$.
\end{remark}

The next lemma shows that any trajectory crosses the separatrices transversely.

\begin{lemma}[Automatic transversality of crossings]\label{lem:transv}
Let $c<0$ and let $\gamma$ be a trajectory distinct from the horizontal (resp.\ vertical)
cylinder. Then every contact of $\gamma$ with $\{y=y_*\}$ (resp.\ $\{x=x_*\}$) is transverse,
i.e.\ $\sin\sigma\ne0$ (resp.\ $\cos\sigma\ne0$) at the contact point. In particular, the
crossings are isolated.
\end{lemma}
\begin{proof}
A tangential contact with $\{y=y_*\}$ would have $\sin\sigma=0$, i.e.\ $\sigma\in\{0,\pi\}$. But
$\{y=y_*,\sigma=0\}$ and $\{y=y_*,\sigma=\pi\}$ are the two orientations of the trajectory of the
horizontal cylinder (Proposition~\ref{prop:cyl}(a) and Lemma~\ref{lem:simetrias}(i)). By uniqueness of
solutions, $\gamma$ would coincide with it. The case of the level $\{x=x_*\}$ is the mirror
image.
\end{proof}

In the next lemma we show that $I$ and $J$ are in fact strictly monotone in the regions determined by the separatrices.

\begin{lemma}[Structure of $Z_I$ and $Z_J$]\label{lem:Z}
Let $c<0$. On $Z_I=\{\cos\sigma=0\}$ one has $\sin\sigma=\pm1$ and, along \eqref{eq:flow},
\begin{equation}\label{eq:sigmaZ}
\left.\sigma'\right|_{Z_I}=\mp A(x)=\mp\frac{p+cx^{2}}{x}.
\end{equation}
Consequently,
\begin{enumerate}
\item[(i)] $\sigma'|_{Z_I}=0$ if and only if $x=x_*$ and the largest invariant subset of $Z_I$ is the trajectory of the vertical cylinder, on which $I\equiv\pm M_f$;
\item[(ii)] off the vertical cylinder, each trajectory meets $Z_I$ only at isolated instants,
with simple zeros of $\cos\sigma$.
\end{enumerate}
The mirror statement holds for $Z_J=\{\sin\sigma=0\}$ and the horizontal cylinder.
\end{lemma}
\begin{proof}
On $Z_I$, $\cos\sigma=0$ and $\sigma'=-A\sin\sigma+B\cdot0=\mp A(x)$, which is \eqref{eq:sigmaZ}.
This vanishes if and only if $p+cx^{2}=0$, i.e.\ $x=x_*$. For a trajectory to remain in $Z_I$ on an
interval it is necessary that $\cos\sigma\equiv0$, whence $x'=\cos\sigma\equiv0$ ($x$ constant
$=x_0$) and $\sigma'\equiv0$. By \eqref{eq:sigmaZ} this forces $A(x_0)=0$, i.e., $x_0=x_*$ and the hypersurface is the vertical cylinder, on which $I=f(x_*)\sin(\pm\tfrac\pi2)=\pm M_f$. This proves (i).
Off it, $\sigma'|_{Z_I}=\mp A(x)\ne0$ and hence
\[
\frac{d}{ds}\cos\sigma=-(\sin\sigma)\sigma'=\pm A(x)\ne0
\]
at the points of $Z_I$. Therefore, the zeros of $\cos\sigma$ are simple and isolated, proving (ii).
\end{proof}

Since these zeros are isolated, the monotonicity of Lemma~\ref{prop:sinais} is strict inside the strips.

\begin{corollary}[Strict monotonicity]\label{cor:estrita}
Let $c<0$ and let $\gamma$ be a trajectory of \eqref{lambda-s} distinct from the two separatrices.
\begin{enumerate}
\item[(i)] On an interval where $\gamma$ remains in a horizontal strip $\mathcal S^{\pm}$, $I$
is strictly monotone: increasing in $\mathcal S^{-}$, decreasing in $\mathcal S^{+}$.
\item[(ii)] On an interval where $\gamma$ remains in a vertical strip $\mathcal V^{\pm}$, $J$
is strictly monotone: increasing in $\mathcal V^{-}$, decreasing in $\mathcal V^{+}$.
\item[(iii)] On an interval where $\gamma$ remains in a quadrant, both are strictly
monotone, with the signs of Definition~\ref{def:regioes}.
\end{enumerate}
In particular $\gamma$ is neither periodic nor recurrent within a strip or quadrant.
\end{corollary}
\begin{proof}
In each strip or quadrant, the signs of $A$ and $B$ are fixed (Lemma~\ref{prop:sinais}), and the
derivatives vanish only on $Z_I$ and $Z_J$, visited at isolated instants (Lemma~\ref{lem:Z}(ii)).
Now, a continuous monotone function whose derivative vanishes only at isolated points is
strictly monotone. Recurrence or periodicity would make the function reassume a value.
\end{proof}

\subsection{Endpoints on the axes}\label{sec:cusp}

We now describe the behavior at an endpoint reaching one of the axes, in the signed form used in the boundary arguments below.

\begin{proposition}[Orthogonal endpoint, signed]\label{lem:cusp}
Let $c\ne0$ and let $\gamma$ be a maximal trajectory of \eqref{eq:flow} defined in an open interval $(s_-,s_+).$ Suppose that $\gamma$ approaches the $x$-axis at the endpoint $\omega$,
i.e., $y\to0^{+}$ and $x\to x_0>0$ as $s\to s_+$. Then
\[
\sin\sigma\to-1\quad(\mbox{downward vertical tangent}),\quad s_+<\infty,
\]
and therefore $I\to-f(x_0)<0$. If $\gamma$ approaches the $x$-axis at the endpoint $\alpha$, i.e., $y\to0^{+}$ and $x\to x_0>0$ as $s\to s_-$, then $\sin\sigma\to+1$,
$s_->-\infty$ and $I\to f(x_0)>0$. 

Symmetrically, approaching the $y$-axis ($x\to0^{+}$,$y\to y_0>0$), one has $\cos\sigma\to\mp1$ at the endpoints $\omega$/$\alpha$, with $J\to\mp g(y_0)$, also in finite time.
\end{proposition}
\begin{proof}
We treat the endpoint $\omega$ on the $x$-axis. The endpoint $\alpha$ follows by reversibility (Lemma~\ref{lem:simetrias}(i)) and the $y$-axis by the mirror image, with $J$ in place of $I$.

Near the $x$-axis one has $B(y)=\tfrac{q}{y}+cy>0$, as the term $q/y$ dominates. Hence, by \eqref{eq:dJ}, $I$ is nondecreasing and it is bounded, since $|I|\le f(x)$, with $x\to x_0$. Therefore, $I_+=\lim I$ exists and, as $f(x)\to f(x_0)>0$,
\[
\sin\sigma \longrightarrow \zeta:=\frac{I_+}{f(x_0)}\in[-1,1].
\]
We rule out the values $\zeta\ne-1$.

\emph{Case $\zeta>0$.} Then $y'=\sin\sigma\ge\zeta/2>0$ near the endpoint, so $y$ is increasing and cannot tend to $0^{+}$, a contradiction.

\emph{Case $|\zeta|<1$.} We show that this case is impossible. Since $\sin\sigma\to\zeta$ with $|\zeta|<1$, there exist $\varepsilon>0$ and a neighborhood $(s_1,s_+)$ of the endpoint in which
\begin{equation}\label{eq:coslb}
|\cos\sigma(s)| \ge k_0:=\sqrt{1-\zeta^2}-\varepsilon >\ 0,
\end{equation}
since $\cos^2\sigma=1-\sin^2\sigma\to1-\zeta^2>0$. In particular, $\cos\sigma$ has constant sign in $(s_1,s_+)$, shrinking the neighborhood if necessary. Then fix the sign, say $\cos\sigma\ge k_0>0$ in $(s_1,s_+)$ (the other case is identical).

We estimate $\sigma'$ in this neighborhood. From $A(x)\to A(x_0)$ it follows that $|A(x)|\le M_A$ for some constant $M_A$, while
\[
B(y)=\frac{q}{y}+cy \ge \frac{q}{2y}
\]
for $y$ small. Hence, using $|\sin\sigma|\le1$ and $\cos\sigma\ge k_0$,
\[
\sigma'=-A(x)\sin\sigma+B(y)\cos\sigma \ge -M_A+\frac{q}{2y}k_0 .
\]
As soon as $y<\dfrac{q k_0}{4M_A}$ (which occurs near the endpoint, since $y\to0$), we obtain
\begin{equation}\label{eq:sigmalb}
\sigma'(s) \ge \frac{qk_0}{4y(s)}>0
\end{equation}
in this neighborhood. In particular $\sigma$ is increasing, with derivative of fixed sign. It remains to show that 
\begin{equation}\label{eq:infinity}
\displaystyle\int^{s_+}\frac{ds}{y(s)}=+\infty,
\end{equation}
 
which by \eqref{eq:sigmalb} implies $\sigma(s)\to+\infty$, making $\cos\sigma(s)$ oscillate and contradicting its constant sign on $(s_1,s_+)$.
We distinguish the subcases by the value of $\zeta$.

\smallskip
\emph{Subcase $\zeta<0$.} Since $y'=\sin\sigma\to\zeta<0$, there exists $s_2\geq s_1$ such that, for
$s\in(s_2,s_+)$,
\begin{equation}\label{eq:ylinha}
\frac{3\zeta}{2}\le y'(s) \le \frac{\zeta}{2}<0.
\end{equation}
In particular, $y$ is strictly decreasing with derivative bounded away from $0$. As $y>0$ and decreasing, the limit $\lim_{s\to s_+^-}y(s)\ge0$ exists, and the final parameter $s_+$ is finite. Indeed, if $s_+=\infty$, integrating the upper bound $y'\le\zeta/2$ on $(s_2,s)$ would give $y(s)\le y(s_2)+\tfrac{\zeta}{2}(s-s_2)\to-\infty$, contradicting $y>0$. Hence $s_+<\infty$, and since $y\to0$ there (it is the endpoint on the axis), we may integrate \eqref{eq:ylinha} from $s$ to $s_+$ to obtain, using $y(s_+)=0$,
\[
y(s)=y(s)-y(s_+)=-\int_s^{s_+}y'(u)du.
\]
From $\tfrac{3\zeta}{2}\le y'\le\tfrac{\zeta}{2}<0$ (that is, $\tfrac{|\zeta|}{2}\le -y'\le\tfrac{3|\zeta|}{2}$) we obtain the two-sided bound
\begin{equation}\label{eq:yenq}
\frac{|\zeta|}{2}(s_+-s) \le y(s) \le \frac{3|\zeta|}{2}(s_+-s).
\end{equation}
This gives
\[
\int^{s_+}\frac{ds}{y(s)} \ge \frac{2}{3|\zeta|}\int^{s_+}\frac{ds}{s_+-s}=+\infty.
\]
\smallskip
\emph{Subcase $\zeta=0$.} Now $y'=\sin\sigma\to0$ with $y\to0^{+}$. We will prove that the divergence of the integral holds whether $s_+$ is finite or not. Fix $\varepsilon\in(0,1)$ and take $s_2\geq s_1$ such that
\[
|y'(s)|=|\sin\sigma(s)|\le \varepsilon
\]
for all $s\in(s_2,s_+)$.
\begin{itemize}
\item If $s_+<\infty$, since $y(s)\to0$ as $s\to s_+$, the fundamental theorem of
calculus gives, for $s\in(s_2,s_+)$,
\[
\begin{aligned}
y(s)&=y(s)-\lim_{u\to s_+}y(u)=-\int_{s}^{s_+}y'(u)du\\
& \le \int_s^{s_+}|y'(u)|du \le \varepsilon(s_+-s),
\end{aligned}
\]
whence
\[
\int_{s_2}^{s_+}\frac{ds}{y(s)} \ge \frac{1}{\varepsilon}\int_{s_2}^{s_+}\frac{ds}{s_+-s}
=+\infty.
\]
\item \emph{If $s_+=\infty$}: from $y(s)\to0$ it follows that, for each $\delta>0$, there exists
$s_\delta$ with $0<y(s)\le\delta$ for all $s\ge s_\delta$, and hence
\[
\int_{s_\delta}^{\infty}\frac{ds}{y(s)} \ge \frac1\delta\int_{s_\delta}^{\infty}ds=+\infty.
\]
\end{itemize}

In either alternative, \eqref{eq:infinity} holds. This eliminates the case $|\zeta|<1$.

There remains $\zeta=-1$: $\sin\sigma\to-1$, $\cos\sigma\to0$, and $y'\to-1$ gives
$y(s)\sim(s_+-s)$ with $s_+<\infty$. Finally $I\to f(x_0)\cdot(-1)=-f(x_0)$.
\end{proof}

The next lemma shows that the hypothesis $x\to x_0$ of Proposition~\ref{lem:cusp} is automatic when $x$ stays away from $0$ and from $\infty$.

\begin{lemma}\label{rem:limites}
Let $c\ne0$ and let $\gamma$ be a maximal trajectory of \eqref{eq:flow}. If $y\to0$ as $s\to s_+$ while $x$ remains in a compact interval $[a,b]\subset(0,\infty)$, then $s_+<\infty$ and $x\to x_0$ for some $x_0\in[a,b]$. In particular, Proposition~\ref{lem:cusp} applies at $\omega$.
\end{lemma}
\begin{proof}
Near the endpoint we have $B(y)\ge q/(2y)$ and $|A(x)|\le M_A:=\max_{[a,b]}|A|$. At a zero of $\sin\sigma$ we have $\cos^{2}\sigma=1$, and \eqref{eq:flow} gives
\[
(\sin\sigma)'=\cos\sigma\,\sigma'=B(y)>0,
\]
so $\sin\sigma$ has at most one zero near the endpoint and $y$ is eventually monotone. Since $y\to0$, eventually $\sin\sigma<0$.

We claim that $\cos\sigma\to0$. Otherwise there are $\kappa\in(0,\tfrac12)$ and $s_k\to s_+$ with $|\cos\sigma(s_k)|\ge2\kappa$. Assume $\cos\sigma(s_k)\ge2\kappa$, so that $\sigma(s_k)\in(\tfrac{3\pi}2,2\pi)$ modulo $2\pi$. As long as $\cos\sigma\ge\kappa$ and $\sigma\in(\tfrac{3\pi}2,2\pi)$,
\[
\sigma'\ \ge\ -M_A+\frac{q\kappa}{2y}\ \ge\ \frac{q\kappa}{4y}\ \ge\ \frac{q\kappa}{4\delta_k},
\qquad \delta_k:=\sup_{s\ge s_k}y(s),
\]
for $k$ large, and $\cos\sigma$ increases with $\sigma$ on that interval. If $s_+=\infty$, then $\sigma$ reaches $2\pi$ after arclength at most $2\pi\delta_k/(q\kappa)$. If $s_+<\infty$, then $y(s)=-\int_s^{s_+}y'(\tau)\,d\tau\le s_+-s$, so that $\sigma'\ge q\kappa/(4(s_+-s))$ as long as $\sigma\in(\tfrac{3\pi}2,2\pi)$. This lower bound is not integrable up to $s_+$. In both cases $\sigma$ reaches $2\pi$ before $s_+$, where $\sin\sigma=0$, contradicting $\sin\sigma<0$. The case $\cos\sigma(s_k)\le-2\kappa$ is symmetric, with $\sigma$ decreasing to $\pi$.

Therefore $\cos\sigma\to0$ and $y'=\sin\sigma\to-1$, so $y$ reaches $0$ in finite arclength and $s_+<\infty$. Since $|x'|=|\cos\sigma|\le1$, the integral in $x(s)=x(\tilde s)+\int_{\tilde s}^{s}\cos\sigma\,d\tau$ converges as $s\to s_+$, for any fixed $\tilde s<s_+$, and $x$ converges to some $x_0\in[a,b]$.
\end{proof}

\section{Apparent loops}\label{sec:winding}

Unless stated otherwise, $c<0$ from now on. This section and the next study loops, excursions and crossings, and prove Theorems~\ref{thmE} and~\ref{thmC}.

\subsection{Notions of loop and some elementary results}\label{sec:lacos}
In order to study the winding of the trajectories of \eqref{lambda-s}, we distinguish three kinds of return:

\begin{definition}[Returns of the profile curve]\label{def:lacos}
Let $\Gamma\subset\QQ\times\s^1,$ given by $\Gamma(s)=(x(s),y(s),\sigma(s))$, be a trajectory of \eqref{eq:flow} and $\gamma$ be the profile curve of a bi-rotational self-shrinker given by the projection of $\Gamma$ on $\QQ,$ i.e., $\gamma(s)=(x(s),y(s))$ (in this case $\sigma(s)$ is the tangent angle of $\gamma(s)$). We distinguish three kinds of return:
\begin{enumerate}
\item[(L1)] \emph{Closed trajectory}: there exist $s_1<s_2$ with $\Gamma(s_1)=\Gamma(s_2)$, i.e., a return to the same point with the same tangent.
\item[(L2)] \emph{Apparent loop}: there exist $s_1<s_2$ with $\gamma(s_1)=\gamma(s_2)$ but
$\sigma(s_1)\ne\sigma(s_2)$ (mod $2\pi$), that is, a self-intersection of the projection with distinct tangents (i.e., $\Gamma(s_1)\ne\Gamma(s_2)$).
\item[(L3)] \emph{Full-turn segment}: $|\Delta\sigma|=2\pi$ on an interval, with no requirement of return.
\end{enumerate}
\end{definition}

The next two results show that closed orbits cannot occur arbitrarily, but are rather confined to a well-defined region of the phase plane $\QQ$.

\begin{lemma}[Closed trajectories do not fit in a strip]\label{prop:noloop}
Let $c<0$. No closed orbit is contained in $\mathcal S^{+}$ or in $\mathcal S^{-}$.
Moreover, every closed orbit crosses $\{y=y_*\}$ at an even number $\ge2$ of points.
\end{lemma}
\begin{proof}
Along a closed orbit, $I$ (which is single-valued on $\mathcal Q\times\s^1$) reassumes its value. If the trajectory were contained in $\mathcal S^{-}$ (the case $\mathcal S^{+}$ is analogous) and were distinct from the vertical cylinder, by Corollary~\ref{cor:estrita}, item (i), $I$ would be strictly increasing, which is a contradiction. Hence, every closed orbit touches $\{y=y_*\}$. Since it returns to the same side of the level and the crossings are transverse and isolated by Lemma~\ref{lem:transv}, their number is even and $\ge2$.
\end{proof}
Combining this with the mirror statement, we obtain
\begin{corollary}[Closed trajectories cross both cylinders]\label{cor:noloop2}
Let $c<0$. Every closed orbit crosses each of the levels $\{y=y_*\}$ and $\{x=x_*\}$ at an even number $\ge2$ of points.
\end{corollary}
\begin{proof}
The level $\{y=y_*\}$ is Lemma~\ref{prop:noloop}. For $\{x=x_*\}$, the argument is the mirror image, with the function $J$ in place of $I$: if the orbit were contained in $\mathcal V^{+}$ or
$\mathcal V^{-}$, then $J$ would be strictly monotone (Corollary~\ref{cor:estrita}(ii)) and could not
reassume its value.
\end{proof}

We also need the following two elementary lemmas:

\begin{lemma}[Umlaufsatz with one corner]\label{lem:umlauf}
Let $\gamma|_{[s_1,s_2]}$ be a simple regular arc in the plane with $\gamma(s_1)=\gamma(s_2)$, i.e., a loop with a single corner of exterior angle $\theta\in(-\pi,\pi)$. Then the total variation of the tangent angle satisfies $\Delta\sigma+\theta=\pm2\pi$. In particular,
\[
|\Delta\sigma|=|{\pm}2\pi-\theta|\ge 2\pi-|\theta|>\pi.
\]
\end{lemma}
\begin{proof}
This is Hopf's rotation-index theorem for simple closed curves with vertices: the total
variation of the tangent angle along the smooth arcs, plus the exterior angles at the vertices,
equals $\pm2\pi$. See \cite[\S4-5, Theorem of Turning Tangents]{doCarmoCurvas}.
\end{proof}

\begin{remark}[Applicability to profile curves]\label{rem:umlauf}
Let $\gamma(s_1)=\gamma(s_2)$, $s_1<s_2$, with $\sigma(s_1)\ne\sigma(s_2)$ mod $2\pi$.
\begin{enumerate}
\item[(i)] \emph{A simple subarc exists.} Since $|\gamma'|=1$ and the curvature $\sigma'$ is bounded on $[s_1,s_2]$, there is $\eta>0$ such that $\gamma$ is injective on every subinterval of length $\le\eta$. The set of pairs $(u,v)$ with $s_1\le u<v\le s_2$, $v-u\ge\eta$ and $\gamma(u)=\gamma(v)$ is compact and nonempty, and a pair minimizing $v-u$ gives a simple subarc $\gamma|_{[u,v]}$ with $\gamma(u)=\gamma(v)$. Moreover, $\sigma(u)\ne\sigma(v)$ mod $2\pi$. Otherwise $\Gamma(u)=\Gamma(v)$, and the trajectory is periodic, with minimal period $T\le v-u$. Since $\sigma(s_1)\ne\sigma(s_2)$ mod $2\pi$, the difference $s_2-s_1$ is not a multiple of $T$. With $k:=\lfloor(s_2-s_1)/T\rfloor\ge1$, the pair $(s_1,s_2-kT)$ lies in $[s_1,s_2]$, satisfies $\gamma(s_1)=\gamma(s_2-kT)$, and has gap $r:=s_2-s_1-kT\in(0,T)$. By the choice of $\eta$, $r\ge\eta$, and $r<T\le v-u$ contradicts the minimality of $v-u$.
\item[(ii)] \emph{The exterior angle satisfies $|\theta|<\pi$.} Suppose $\sigma(s_2)=\sigma(s_1)+\pi$ mod $2\pi$, i.e.\ $\Gamma(s_2)=R(\Gamma(s_1))$, with $R(x,y,\sigma)=(x,y,\sigma+\pi)$. By Lemma~\ref{lem:simetrias}, item (i), $s\mapsto R(\Gamma(-s))$ is a solution, and by uniqueness $\Gamma(s_2+t)=R(\Gamma(s_1-t))$ for all admissible $t$. For $t=(s_1-s_2)/2$ this gives $\Gamma(m)=R(\Gamma(m))$, for $m=(s_1+s_2)/2$, that is, $\sigma(m)=\sigma(m)+\pi$ mod $2\pi$, which is absurd. The same argument applies to the pair $(u,v)$ of item (i). The value $\theta=0$ is excluded for that pair by item (i).
\end{enumerate}
Hence Lemma~\ref{lem:umlauf} applies to the simple subarc of every apparent loop.
\end{remark}

\begin{definition}[Loop]\label{def:loop}
A \emph{loop} of $\gamma$ is a subarc $\gamma|_{[u,v]}$, $u<v$, with $\gamma(u)=\gamma(v)$ and $\sigma(u)\ne\sigma(v)$ (mod $2\pi$), such that $\gamma$ is injective on $[u,v)$. Two loops are distinct if their parameter intervals are distinct, and a loop is confined to a subset of $\mathcal Q$ if its image is contained in it.
\end{definition}

By Remark~\ref{rem:umlauf}, every apparent loop contains a loop, and Lemma~\ref{lem:umlauf} applies to every loop.

\begin{lemma}[Cardinal points: lower and upper bounds]\label{lem:cardeais}
Let $(\sigma_1,\sigma_2)\subset\R$ be an open interval of length $L$.
\begin{enumerate}
\item[(i)] If $L>\pi$, the interval contains at least two integer multiples of
$\tfrac\pi2$, and two consecutive ones contained in it have distinct parities, i.e., one
is of the form $m\pi$ (a zero of $\sin$) and the other $\tfrac\pi2+m\pi$ (a zero of $\cos$). In particular, the interval contains a zero of $\sin$ and a zero of $\cos$.
\item[(ii)] If the interval contains three multiples of $\tfrac\pi2$, then two of them
have the same parity. Hence, one of the functions $\sin$ or $\cos$ has two zeros in
the interval.
\item[(iii)] If $L\ge\tfrac{3\pi}{2}$ and neither endpoint $\sigma_1,\sigma_2$ is a multiple of
$\tfrac\pi2$, then the interval contains at least three multiples of $\tfrac\pi2$.
\end{enumerate}
\end{lemma}
\begin{proof}
Elementary from real analysis.
\end{proof}

\subsection{Apparent loops confined to a quadrant}

Along a loop the tangent turns by more than $\pi$, which forces zeros of $\sin\sigma$ and of $\cos\sigma$. Inside a quadrant, the monotonicity of $I$ and $J$ determines how these zeros occur.

\begin{lemma}[Necessary constraints]\label{prop:restricoes}
Let $c<0$ and let $\gamma|_{[s_1,s_2]}$ be a loop, so that $\gamma(s_1)=\gamma(s_2)$. Then there exist $t_I,t_J\in(s_1,s_2)$ with $\sin\sigma(t_I)=0$ and $\cos\sigma(t_J)=0$. 
If, in addition, the subarc is contained in a single region $\mathrm{Q_{\pm\pm}}$, then both functions $I$ and $J$ change sign on the subarc, each exactly once, in the sense dictated by the quadrant. For instance, in $\mathrm{Q_{++}}$, where $I$ and $J$ are strictly increasing, both pass from $-$ to $+$, whence
\begin{equation}\label{eq:sin-cos}    
\sin\sigma(s_1)<0<\sin\sigma(s_2),\qquad \cos\sigma(s_1)<0<\cos\sigma(s_2).
\end{equation}

\end{lemma}
\begin{proof}
The existence of the zeros follows from Lemma~\ref{lem:cardeais}(i), because Lemma~\ref{lem:umlauf} gives $|\Delta\sigma|>\pi$ along a loop. Figure~\ref{fig:laco2} shows a loop with exactly one zero of each. Now assume that the subarc lies in $\mathrm{Q_{\pm\pm}}$. In a quadrant, the signs of the functions $A(x)=\frac{p}{x}+cx$ and $B(y)=\frac{q}{y}+cy$ are fixed. Hence, by Corollary~\ref{cor:estrita}, item (iii), $I$ and $J$ are strictly monotone, as the subarc cannot lie on a cylinder, which has no self-intersections. Since $f,g>0$, one has $\sgn I=\sgn(\sin\sigma)$ and $\sgn J=\sgn(\cos\sigma)$. As a strictly monotone function changes sign at most once and, as we proved, it changes sign at least once, these functions change sign exactly once. As the point is the same at the instants $s_1$ and $s_2$, the weights coincide there, $f(x(s_1))=f(x(s_2))$ and $g(y(s_1))=g(y(s_2))$, so the inequalities $I(s_2)>I(s_1)$ and $J(s_2)>J(s_1)$ translate directly into the inequalities \eqref{eq:sin-cos}.
\end{proof}

\begin{figure}[htbp]
\centering
\includegraphics[width=0.9\textwidth]{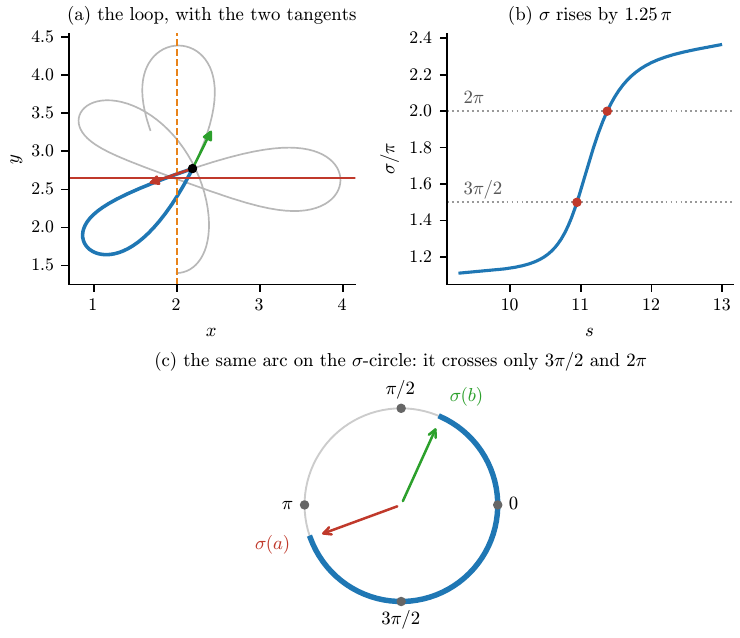}
\caption{A numerical simulation of an apparent loop ($p=4$, $q=7$, $c=-1$, initial condition $(2,1.4,0)$,
$s\in[9.29,12.97]$) whose tangent angle crosses exactly two cardinal points. In (a), the loop in
the plane, in bold, inside the trajectory, with the two tangents at the corner. The separatrices
$\{x=x_*\}$ and $\{y=y_*\}$ are the dashed and the solid lines. In (b), $\sigma(s)$ rising by
$1.2525\pi$ and crossing only $\tfrac{3\pi}{2}$ and $2\pi$. In (c), the same arc on the
$\sigma$-circle, leaving $\tfrac{\pi}{2}$ and $\pi$ out. By Lemma~\ref{prop:restricoes}, along an apparent loop each of $\sin\sigma$ and $\cos\sigma$ (equivalently, each of $I$ and $J$) vanishes at least once. This loop attains the minimum, with exactly one zero of each. It crosses both separatrices, so it is not confined to a quadrant.
}
\label{fig:laco2}
\end{figure}

In a quadrant these zeros are unique, and this bounds the turning of the tangent from above.

\begin{proposition}[Winding bound in a quadrant]\label{thm:cota32}
Let $c<0$ and let $\gamma|_{[s_1,s_2]}$ be a subarc contained in a single region $\mathrm{Q_{\pm\pm}}$, along which both functions $I$ and $J$ change sign. Then
\begin{equation}\label{eq:cota32}
|\Delta\sigma|<\frac{3\pi}{2}.
\end{equation}
Moreover, combined with Lemma~\ref{lem:umlauf}, this shows that a loop confined to a
quadrant satisfies
\begin{equation}\label{eq:janela}
\pi<|\Delta\sigma|<\frac{3\pi}{2}.
\end{equation}
\end{proposition}
\begin{proof}
In a quadrant, the functions $I$ and $J$ are strictly monotone along the subarc (Corollary~\ref{cor:estrita}(iii)). Since $f,g>0$, one has $\sgn I=\sgn(\sin\sigma)$ and $\sgn J=\sgn(\cos\sigma)$. Therefore each of the functions $\sin\sigma$, $\cos\sigma$ has at most one zero in the interior of $[s_1,s_2]$, since a second zero would force the corresponding monotone function to reassume the value $0$.

Suppose, for contradiction, that $|\Delta\sigma|\ge\tfrac{3\pi}{2}$. The open interval with endpoints $\sigma(s_1)$ and $\sigma(s_2)$ has length $|\Delta\sigma|\ge\tfrac{3\pi}2$, and every value in it is attained by $\sigma$ on $(s_1,s_2)$. Its endpoints are not multiples of $\tfrac\pi2$. Indeed, since $I$ and $J$ are strictly monotone and each changes sign on the subarc, their values at $s_1$ and $s_2$ are nonzero (and have opposite signs in each case). As $f,g>0$, neither $\sin\sigma$ nor $\cos\sigma$ vanishes at $s_1$ or $s_2$. By Lemma~\ref{lem:cardeais}(iii), the
interval then contains at least three multiples of $\tfrac\pi2$, and by
Lemma~\ref{lem:cardeais}(ii), two of them have the same parity, i.e.\ $\sin\sigma$ or
$\cos\sigma$ vanishes at two interior points. This contradicts the strict monotonicity of
$I$ or of $J$.
\end{proof}

\subsection{Uniqueness per strip}\label{sec:unico}
The next goal is to prove that each excursion of a strip contains at most one apparent loop.

\begin{definition}[Chord vector and its angle]\label{def:corda}
Let $\gamma|_{[u,v]}$ be a regular arc in the plane, parametrized by arclength, with tangent angle $\sigma$, that is, $\gamma'(s)=(\cos\sigma(s),\sin\sigma(s))$. We call the chord vector of $\gamma|_{[u,v]}$ the vector
\[
\gamma(v)-\gamma(u)=\int_u^{v}\gamma'(s)ds,
\]
which joins the initial point of the arc to the final one. When it is nonzero, its angle is the unique $\varphi\in\R/2\pi\Z$ such that
\[
\gamma(v)-\gamma(u)=|\gamma(v)-\gamma(u)|(\cos\varphi,\sin\varphi),
\]
that is, the polar angle of the direction of the segment $\overline{\gamma(u)\gamma(v)}$, measured in the same frame in which $\sigma$ is measured. 
\end{definition}

Note that the chord is the average of the tangents, in the literal sense of the integral equality above, while the following lemma is the precise version of that intuition.

\begin{lemma}[Direction of the chord: cone lemma]\label{lem:cone}
Let $\gamma|_{[u,v]}$ be a regular arc parametrized by arclength whose tangent angle stays in an interval $[\theta_-,\theta_+]$ with $\theta_+-\theta_-<\pi$. Then the chord vector $\gamma(v)-\gamma(u)$ is nonzero and its angle belongs to $[\theta_-,\theta_+]$ ($\mbox{mod}\ 2\pi$).
\end{lemma}
\begin{proof}
The set of vectors whose polar angle lies in $[\theta_-,\theta_+]$, together with the origin, is a closed
convex cone (since $\theta_+-\theta_-<\pi$) 
\[
\mathcal C=\{w;\ \langle w,u_0\rangle\geq|w|\cos\theta_0\},
\]
where
\[
u_0:=\left(\cos\left(\frac{\theta_-+\theta_+}{2}\right),\sin\left(\frac{\theta_-+\theta_+}{2}\right)\right) \quad\mbox{and}\quad \theta_0:=\frac{\theta_+-\theta_-}{2}<\frac{\pi}{2}
\]
are the axis and the half-angle of the cone, respectively. For $\sigma(s)\in[\theta_-,\theta_+]$ we have
\[
\langle T(s),u_0\rangle=\cos\left(\sigma(s)-\frac{\theta_++\theta_-}{2}\right)\geq\cos\theta_0>0,
\]
for every $s\in[u,v]$. After integrating, we obtain
\[
\langle\gamma(v)-\gamma(u),u_0\rangle = \int_u^v \langle T(s),u_0\rangle ds\geq (v-u)\cos\theta_0>0
\]
and the chord vector $\gamma(v)-\gamma(u)$ is nonzero. As each tangent $T(s)$ lies in $\mathcal C$ and $\gamma(v)-\gamma(u)=\int_u^vT ds$ is a limit of positive combinations of such vectors, it lies in $\mathcal C$.
\end{proof}

We can now prove the goal stated at the beginning of this subsection. The monotonicity of $I$ separates loops with disjoint parameter intervals, and Lemma~\ref{lem:cone} separates interlaced ones.

\begin{theorem}[At most one loop per strip excursion]\label{thm:unico}
Let $c<0$ and let $\gamma$ be a profile curve of a bi-rotational self-shrinker distinct from the cylinders. Then $\gamma$ cannot have two distinct loops contained in the same excursion of a horizontal strip $\mathcal S^{+}$ or $\mathcal S^{-}$, that is, in the same connected component of $\gamma^{-1}(\mathcal S^{\pm})$. The same holds for excursions of the vertical strips $\mathcal V^{\pm}$, and \emph{a fortiori} for each quadrant.
\end{theorem}
\begin{proof}
Let $\gamma|_{[a_1,b_1]}$ and $\gamma|_{[a_2,b_2]}$ be two distinct loops contained, say, in the same excursion $(e_1,e_2)$ of $\mathcal S^{-}$ (the remaining cases are identical, switching signs and, in the vertical case, replacing $I$ with $J$). We write $P:=\gamma(a_1)=\gamma(b_1)$ and $Q:=\gamma(a_2)=\gamma(b_2)$. Notice that it is possible to have $P=Q$, since distinct loops can have the same vertex. We may assume $a_1\le a_2$. If $a_1\le a_2<b_2\le b_1$ (nested intervals), then $\gamma(a_2)=\gamma(b_2)$ with $a_2\ne b_2$ in $[a_1,b_1]$, and simplicity of $\gamma|_{[a_1,b_1]}$ forces $\{a_2,b_2\}=\{a_1,b_1\}$, i.e.\ the loops coincide. If $a_1=a_2$ and $b_1<b_2$, then $\gamma(b_1)=\gamma(a_1)=\gamma(a_2)$ with $b_1\in(a_2,b_2)$, contradicting the simplicity of $\gamma|_{[a_2,b_2]}$. It remains to treat disjoint and interlaced intervals with $a_1<a_2$.

\emph{Case 1: disjoint intervals} ($b_1\le a_2$). By Lemma~\ref{prop:restricoes}, each subarc contains in its interior a zero of $\sin\sigma$, and these zeros are distinct by disjointness (note that this uses only the parameter intervals, not the
position of the vertices, so it covers the case $P=Q$ as well). Hence $I=f\sin\sigma$ vanishes at two distinct instants of the interval $(e_1,e_2)$. But $\gamma((e_1,e_2))\subset\mathcal S^{-}$, so by Corollary~\ref{cor:estrita}, item (i), $I$ is strictly increasing on $(e_1,e_2)$, hence injective there. This is a contradiction.

\emph{Case 2: interlaced intervals} ($a_1<a_2<b_1<b_2$). Here the zeros may coincide and the counting argument fails, so we use the geometry instead. Note first that, by hypothesis, the arcs  $\gamma([a_1,b_1])$ and $\gamma([a_2,b_2])$ both lie in $\mathcal S^{-}$. Since $a_2<b_1$, the two parameter intervals overlap, so that $[a_1,b_2]=[a_1,b_1]\cup[a_2,b_2]$ and therefore $\gamma([a_1,b_2])\subset\mathcal S^{-}$ as well. By Corollary~\ref{cor:estrita}, item (i), $I$ is strictly increasing on the whole parameter interval $[a_1,b_2]$ and has at most one zero $t_I$ there. By Lemma~\ref{lem:cardeais}, each of the two subarcs contains a zero of $\sin\sigma$ in its interior, hence a zero of $I$. Therefore, these two zeros coincide, and
\[
t_I\in(a_1,b_1)\cap(a_2,b_2)=(a_2,b_1).
\]
Since $\sgn I=\sgn(\sin\sigma)$ and $I$ is increasing with the single zero $t_I$, we get  $\sin\sigma<0$ on $[a_1,a_2]$ and $\sin\sigma>0$ on $[b_1,b_2]$. Consequently, $\sigma([a_1,a_2])$ is a compact connected subset of the open interval $(-\pi,0)$ modulo $2\pi$, so it is contained in some $[\theta_-,\theta_+]$ with $-\pi<\theta_-\le \theta_+<0$ and $\theta_+-\theta_-<\pi$. Analogously, $\sigma([b_1,b_2])\subset(0,\pi)$ has length $<\pi$. Lemma~\ref{lem:cone} therefore applies to both, and the chord vectors
\[
\gamma(a_2)-\gamma(a_1)=Q-P
\quad\text{and}\quad
\gamma(b_2)-\gamma(b_1)=Q-P,
\]
which are the same nonzero vector (if $P=Q$, then
$\gamma(a_1)=\gamma(a_2)$ with $a_1\ne a_2$, both in $[a_1,b_1)$, contradicting the simplicity of $\gamma|_{[a_1,b_1]}$), have angle $\varphi$ in $\sigma([a_1,a_2])$, where $\sin\varphi<0$, and in $\sigma([b_1,b_2])$, where $\sin\varphi>0$, respectively, which is a contradiction. The vertical case is identical with $J$, $\cos\sigma$ and the left and right half sectors (where $\cos\varphi<0$ and $\cos\varphi>0$, respectively).
\end{proof}

Counting the excursions, the previous theorem bounds the number of loops by the number of crossings.

\begin{corollary}[Counting the loops confined to a strip]\label{cor:superpos}
Let $c<0$ and let $\gamma$ be a trajectory distinct from the separatrices, with $N_x$ and $N_y$ crossings of $\{x=x_*\}$ and $\{y=y_*\}$. If $\gamma$ is not closed, then
\begin{equation}\label{eq:cotalacos}
\#\{\text{loops confined to a horizontal strip}\}\le N_y-1,
\end{equation}
and symmetrically the number of those confined to a vertical strip is at most $N_x-1$. If $\gamma$ is closed, the bounds are $N_y$ and $N_x$. In particular a trajectory with $N_y\le1$ has no loop confined to a horizontal strip.
\end{corollary}
\begin{proof}
The $N_y$ crossings divide $\gamma$ into exactly $N_y+1$ horizontal excursions, and a loop
confined to a strip has its subarc inside a single one. By Theorem~\ref{thm:unico}, each excursion contains at most one loop. This would give at most $N_y+1$ loops. But two of the excursions are terminal, the one containing $\alpha=\lim_{s\to s_-}\gamma(s)$ and the one containing $\omega=\lim_{s\to s_+}\gamma(s)$ (for the maximal interval of definition $(s_-,s_+)$), and by Proposition~\ref{thm:terminal}, neither contains a loop. There remain $N_y-1$ intermediate excursions, which gives \eqref{eq:cotalacos}. Note that $N_y\ge1$ always holds by Theorem~\ref{thm:cross} and that for $N_y=1$ both excursions are terminal and the bound is zero. If $\gamma$ is closed, the $N_y$ crossings of one period, taken to start at a crossing, divide it into $N_y$ excursions, none of them terminal, and Theorem~\ref{thm:unico} gives at most $N_y$ loops. The vertical case is the mirror image.
\end{proof}
\begin{remark}\label{rem:lacoquadrante}
    Numerical experiments suggest the existence of loops confined to a quadrant, as we can see in the middle and right panels of \cite[Figure~26, p.~28]{DLN}, which show numerical profile curves of immersed bi-rotational self-shrinkers in $\R^4$ ($p=q=1$). We performed some of these experiments, obtaining other loops of this kind. See Figure \ref{fig:lacoquadrante}. Therefore, the bound of Theorem \ref{thm:unico} seems to be sharp.

    \begin{figure}[htbp]
\centering
\includegraphics[width=\textwidth]{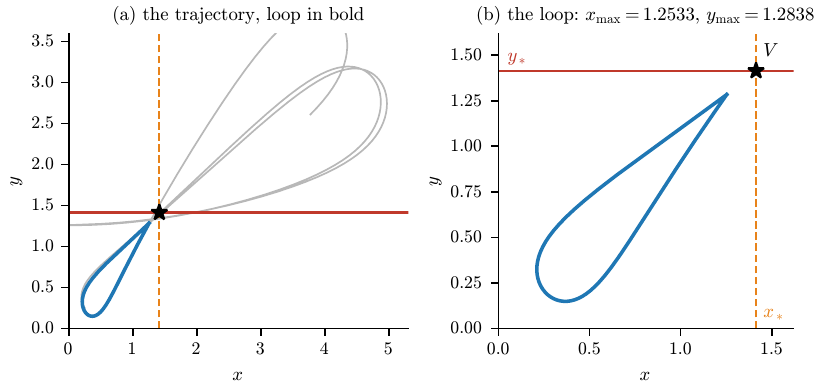}
\caption{An apparent loop confined to the quadrant
$\mathrm{Q_{++}}$, for $p=q=1$ and $c=-\tfrac12$. In (a), the trajectory issued from
$(0.250,0.250)$ on the diagonal, with tangent angle $\sigma=\tfrac{3\pi}4$, with the loop in bold.
The loop is the subarc $s\in[21.279,24.454]$. In (b), the loop alone, which stays strictly
below both separatrices $x=x_*$ and $y=y_*=\sqrt2$, whose intersection is the vertex $V$ (star).}
\label{fig:lacoquadrante}
\end{figure}
\end{remark}

\section{Excursions, crossings and winding}\label{sec:excursoes}

The crossings of the two separatrices divide each trajectory into excursions, and the global behavior is determined by their sequence. We first establish the boundary facts that govern a single excursion, then we eliminate the loops from the terminal excursions, and finally we bound the winding by the number of crossings.

\subsection{Behavior at the boundary of the strips}\label{sec:fronteira}

We collect the boundary facts used below. The first is classical, and we state it for completeness. For the proof, see \cite[Lemmas~6.3 and 6.5--6.7]{Teschl} or \cite[Chapter~I, Section~8, Theorem~8.1]{Hale}. See also \cite[Chapter~2, Theorem~5.2]{LaSalle}.

\begin{lemma}[Limit sets of precompact semi-orbits]\label{lem:omega}
If a positive semi-orbit of a $C^{1}$ field is contained in a compact subset $K$ of the domain, then it is defined for all $s\ge0$, its $\omega$-limit set is nonempty, compact, connected and invariant, contained in $K$, the orbit converges to it, and the complete trajectory of each of its points is defined for all time and remains in it.
\end{lemma}

The second fact describes how a trajectory confined to a horizontal strip can end. Since $I$ is monotone and bounded there, it has limits at both endpoints, and their values distinguish the possible endpoints.

\begin{proposition}[Limit values of $I$ in a strip]\label{lem:Jfront}
Let $c<0$ and let $\gamma$ be a trajectory of \eqref{lambda-s} confined to $\mathcal S^{+}$ or to $\mathcal S^{-}$, distinct from the cylinders, with maximal interval $(s_-,s_+)$. Let $I_\pm$ be the limits of $I$ at the endpoints (which exist, by monotonicity and boundedness, see Corollary~\ref{cor:estrita}(i) and Lemma~\ref{lem:bound}). Then each endpoint falls into one of the following scenarios:
\begin{enumerate}
\item[\textup{(F1)}] if $x\to\infty$ or $x\to0$ along some sequence $s_k\to s_\pm$, then
$I_\pm=0$, since $|I|\le f(x)\to0$;
\item[\textup{(F2)}] the scenario in which $x$ remains in a compact subset of $(0,\infty)$ and $y$ remains bounded and away from $0$ cannot occur;
\item[\textup{(F3)}] if the endpoint lies on the $x$-axis, at $(x_0,0)$ (only possible in
$\mathcal S^{-}$, and in finite time), then $I_-=f(x_0)>0$ at the endpoint $\alpha$ and
$I_+=-f(x_0)<0$ at the endpoint $\omega$;
\item[\textup{(F4)}] $x$ remains bounded, while $\limsup y=\infty$, which is possible only in $\mathcal S^{+}.$
\end{enumerate}
\end{proposition}

\begin{proof}
The four scenarios exhaust the possibilities. Indeed, if $x$ tends neither to $0$ nor to $\infty$ along any sequence $s_k\to s_\pm$, then $x$ remains in a compact subset of $(0,\infty)$ near the endpoint. Since $I$ is strictly monotone (Corollary~\ref{cor:estrita}(i)), it has at most one zero, and its zeros are those of $\sin\sigma=y'$. Hence $y$ is eventually monotone and converges to some $y_\infty\in[0,\infty]$. The value $y_\infty=0$ is possible only in $\mathcal S^{-}$, and it gives an endpoint on the $x$-axis by Lemma~\ref{rem:limites} and Proposition~\ref{lem:cusp}, which is (F3). A value $y_\infty\in(0,\infty)$ is (F2), and $y_\infty=\infty$ is (F4). Case (F1) follows from Lemma~\ref{lem:bound} and case (F3) is Proposition~\ref{lem:cusp}. The situation described in (F4) will be treated in Propositions~\ref{lem:fuga} and \ref{prop:nao-eixo} below. Let us prove (F2), supposing for contradiction that, near
the endpoint, $x(s)\in[a,b]\subset(0,\infty)$ and $y(s)\in[y_{\flat},y_{\sharp}]\subset(0,\infty)$.

The hypothesis says that the semi-orbit is contained in the compact set
$\mathcal K:=[a,b]\times[y_{\flat},y_{\sharp}]\times\s^1$, which is a compact subset of the domain $\mathcal Q\times\s^1$ of the field. The level $\{y=y_*\}$ lies in this domain, although it bounds the strip, and the semi-orbit may accumulate on it. Hence the semi-orbit is precompact and Lemma~\ref{lem:omega} applies: its limit set $\Lambda$ is nonempty, compact,
connected and invariant, with $\Lambda\subset\mathcal K$, and the complete trajectory of each
$z\in\Lambda$ is defined for all time and remains in $\Lambda$. Since $I$ is monotone and
bounded, $I\equiv I_\pm$ on $\Lambda$ by continuity. Hence $I'\equiv0$ on $\Lambda$, that is,
\[
\Lambda\ \subset\ Z_I\cup\{y=y_*\},\quad\mbox{where}\quad  Z_I=\{\cos\sigma=0\}
\]
(Lemma~\ref{prop:sinais}). Fix $z=(x(z),y(z))\in\Lambda$ and analyze the possibilities, which are
exhaustive.

\begin{enumerate}
\item[(i)] \emph{$z\in Z_I$, with $y(z)\ne y_*$ and $x(z)\neq x_*$.} By
Lemma~\ref{lem:Z} (ii), the trajectory of $z$ meets $Z_I$ only at isolated instants: it leaves $Z_I$ immediately. And, by continuity, it keeps $y\ne y_*$ at nearby instants. Hence the orbit leaves $Z_I\cup\{y=y_*\}$, contradicting $\Lambda\subset Z_I\cup\{y=y_*\}$ and the invariance of $\Lambda$.

\item[(ii)] \emph{$z$ belongs to the trajectory of the vertical cylinder} ($x\equiv x_*$,
$\cos\sigma=0$). Along it $y'=\sin\sigma=\pm1$, so that $y$ ranges over the whole interval $(0,\infty)$: the trajectory does not fit in $[y_{\flat},y_{\sharp}]$, nor in the closure of $\mathcal S^{+}$ or of $\mathcal S^{-}$ (Rem.~\ref{rem:vertex}). This contradicts
$\Lambda\subset\mathcal K$.

\item[(iii)] \emph{$z$ has $y(z)=y_*$ and $\sin\sigma(z)\ne0$.} Then $y'=\sin\sigma\ne0$ at $z$, and the trajectory leaves $\{y=y_*\}$ immediately. To remain in $Z_I\cup\{y=y_*\}$, it would have to be in $Z_I$ at all nearby instants, that is, to belong to the largest invariant subset of $Z_I$, which is exactly the vertical cylinder (Lemma~\ref{lem:Z}(i)), a case already excluded in (ii). If, on the contrary, $\cos\sigma(z)\ne0$, then by continuity the orbit has $y\ne y_*$ and $\cos\sigma\ne0$ just after $z$, leaving the union. In either alternative, we have a contradiction.

\item[(iv)] \emph{$z$ has $y(z)=y_*$ and $\sin\sigma(z)=0$.} In this case, $\sigma(z)\in\{0,\pi\}$ and $z$ belongs to the trajectory of the horizontal cylinder, which is indeed invariant and $I=f(x)\sin\sigma\equiv0$ on it. But the trajectory of the horizontal cylinder leaves any compact set, contradicting that $\Lambda\subset\mathcal K.$

\end{enumerate}

Since the alternatives (i) to (iv) were exhaustive, no point can belong to $\Lambda$, that is, $\Lambda=\emptyset$, contradicting the nonemptiness of the limit set of a precompact semi-orbit. Hence the hypothesis of (F2) is untenable, which proves (F2).
\end{proof}

\begin{remark}\label{rem:Jfront-end}
The proofs of (F1)--(F3) use only a semi-orbit. Hence they hold for an end of any trajectory distinct from the cylinders which is eventually contained in $\mathcal S^{+}$ or in $\mathcal S^{-}$: on the corresponding tail, $I$ is strictly monotone and converges, (F2) cannot occur, and in (F3) the existence of $x_0$ is guaranteed by Lemma~\ref{rem:limites}.
\end{remark}

The following result shows that bi-rotational self-shrinkers distinct from the cylinders cannot have ends asymptotic to the cylinders. This shows that case (F4) of Proposition~\ref{lem:Jfront} is also impossible.

\begin{proposition}[Absence of vertical escape]\label{lem:fuga}
Let $c<0$ and let $\gamma$ be a trajectory of \eqref{lambda-s} distinct from the vertical cylinder. There is no endpoint of $\gamma$ near which $\gamma$ is contained in $\mathcal S^{+}$, with $\limsup y=\infty$ and $x(s)\in[a,b]\subset(0,\infty)$. In particular, no trajectory is confined to $\mathcal S^{+}$ with these properties.
\end{proposition}

\begin{proof}
By reversibility (Lemma~\ref{lem:simetrias}(i)), we may assume that the endpoint is $\omega$, and we work on a semi-orbit $(s_0,s_+)$ contained in $\mathcal S^{+}$. Suppose, for contradiction, that $y$ is unbounded there, with
$x\in[a,b]$. Set
\[
\tau:=\sigma-\tfrac\pi2,
\qquad\text{so that}\qquad
\sin\sigma=\cos\tau,\quad \cos\sigma=-\sin\tau.
\]
Substituting into \eqref{eq:flow}, we obtain
\begin{equation}\label{eq:fuga-tau}
x'=-\sin\tau,\qquad y'=\cos\tau,\qquad
\tau'=\sigma'=-A(x)\cos\tau+\tilde B(y)\sin\tau,
\end{equation}
where
\[
\tilde B(y):=-B(y)=|c|y-\frac{q}{y}>0\quad\mbox{in}\quad\mathcal S^{+}.
\]
Note that $\tilde B$ is increasing and $\tilde B(y)\to\infty$. Thus, fix $y_1>y_*$ such that
\begin{equation}\label{eq:lamlb}
\tilde B(y) \ge \tfrac12|c|y\quad\mbox{for}\quad y\ge y_1 .
\end{equation}%
We claim that $\cos\tau>0$ along the whole semi-orbit. Indeed, by \eqref{eq:dJ}, and $B<0$ on $\mathcal S^{+}$, the function $I$ is strictly decreasing there and since $f>0$, the zeros of $I$ are exactly those of $\sin\sigma=\cos\tau$, so there is at most one. If there were
one, then $I<0$ after it, hence $y'=\sin\sigma<0$ and $y$ would be decreasing from that point on, contradicting that $\limsup y=\infty.$ If $I$ had no zero and were negative, the same conclusion would hold. 
Therefore $I>0$, that is, $\sin\sigma=\cos\tau>0$, on the whole semi-orbit, and $y$ is strictly increasing
throughout. 

Consequently, $y$ may be taken as the independent variable on the whole semi-orbit. Dividing the equations of \eqref{eq:fuga-tau} by $y'=\cos\tau$ gives
\begin{equation}\label{eq:fuga-red}
\frac{d\tau}{dy}=\tilde B(y)\tan\tau-A(x),\qquad \frac{dx}{dy}=-\tan\tau.
\end{equation}
Let $M_A:=\max_{[a,b]}|A|$.

\smallskip
\emph{Step 1 (the sliding manifold and its instability).} Replace $\tau$ by the variable
\[
u\ :=\ \tan\tau,
\]
which is legitimate as $|\tau|<\tfrac\pi2$ throughout. From
$\tfrac{du}{dy}=(1+u^{2})\tfrac{d\tau}{dy}$ and the first equation of \eqref{eq:fuga-red}, we have
\begin{equation}\label{eq:u-ode}
\frac{du}{dy}=(1+u^{2})\left(\tilde B(y) u-A(x)\right)
=(1+u^{2})\tilde B(y)\left(u-m\right),
\quad\mbox{for}\quad
m:=\frac{A(x)}{\tilde B(y)}.
\end{equation}
The curve $\{u=m\}$, where $du/dy$ vanishes, is the sliding manifold, and
\begin{equation}\label{eq:m-small}
|m|\le\frac{M_A}{\tilde B(y)}\le\frac{2M_A}{|c|y}\xrightarrow{y\to\infty}0,
\end{equation}
by \eqref{eq:lamlb}. So fix $y_1$ large enough that $|m(y)|\le\tfrac12$ for $y\ge y_1$. Thus the sliding manifold approaches $\{u=0\}$, the vertical direction.

We measure the departure from it by the deviation
\begin{equation}\label{eq:h-def}
\delta:=u-m .
\end{equation}
Differentiating and using \eqref{eq:u-ode}, we obtain
\begin{equation}\label{eq:h-ode}
\frac{d\delta}{dy}=(1+u^{2})\tilde B(y)\delta -\frac{dm}{dy}.
\end{equation}
Differentiating $m=A(x)/\tilde B(y)$ and using
$dx/dy=-u$ from \eqref{eq:fuga-red}, we have
\[
\frac{dm}{dy}=\frac{A'(x)}{\tilde B(y)}\frac{dx}{dy}-\frac{A(x)\tilde B'(y)}{\tilde B(y)^2}
=-\frac{A'(x) u}{\tilde B(y)}-\frac{A(x)\tilde B'(y)}{\tilde B(y)^{2}},
\]
whence, setting $M_A':=\max_{[a,b]}|A'|$ and noting that $\tilde B'(y)=|c|+q/y^{2}\le2|c|$ for $y$
large, together with $\tilde B(y)\ge\tfrac12|c|y$ from \eqref{eq:lamlb}, we obtain
\begin{equation}\label{eq:dm-bound}
\left|\frac{dm}{dy}\right|
 \le \frac{M_A'|u|}{\tilde B(y)}+\frac{2|c|M_A}{\tilde B(y)^{2}}
 \le \frac{2M_A'|u|}{|c|y}+\frac{8M_A}{|c|y^{2}}
\le \frac{K'\left(|u|+1\right)}{y},
\end{equation}
for a constant $K'$ depending only on $p,q,c,a,b$, for instance
\[
K'=\max\left\{\frac{2M_A'}{|c|},\ \frac{8M_A}{|c|y_1}\right\}.
\]

\emph{Step 2. One has $|u|<1$ on $[y_1,\infty)$, and
\begin{equation}\label{eq:h-small}
|\delta(y)| \le\ \frac{C_0}{y^{2}}
\qquad\text{for all }y\ge y_1,
\end{equation}
with $C_0$ depending only on $p,q,c,a,b$.}

First, let us prove that $|u|<1$. Suppose $|u(y_2)|\ge1$ for some $y_2\ge y_1$. Since $|m|\le\tfrac12$ there, $u-m$ has the sign of $u$, and whenever $|u|\geq 1,$
\[
|u-m|\ \ge\ |u|-|m|\ \ge\ |u|-\tfrac12\ \ge\ |u|-\tfrac{|u|}{2}\ =\ \tfrac{|u|}{2}.
\]
On the other hand, as $1+u^{2}\ge u^{2}$ and $\tilde B\ge\tfrac12|c|y$, equation \eqref{eq:u-ode} gives, for $u\ge1$,
\begin{equation}\label{eq:dudy1}
\frac{du}{dy}\ge u^{2}\cdot\frac{|c|y}{2}\cdot\frac u2=\frac{|c|y}{4}u^3 > 0 .
\end{equation}
Thus $u$ is increasing and the inequality $u\ge1$ perpetuates itself. Analogously, if $u\leq -1$ then
\begin{equation}\label{eq:dudy2}
\frac{du}{dy}=(1+u^{2})\tilde B(y)(u-m) \le -u^{2}\cdot\frac{|c|y}{2}\cdot\frac{|u|}{2} = -\frac{|c|y}{4}|u|^{3}<0,
\end{equation}
which gives that $u$ is decreasing and $u\leq-1$ also perpetuates itself. Thus, both inequalities give that $|u|\geq 1$ holds for every $y\ge y_2.$ Assuming $u(y_2)\geq 1$ and separating variables between $y_2$ and $y$ on \eqref{eq:dudy1}, we have
\[
\frac{|c|}{8}\left(y^{2}-y_2^{2}\right) \le \int_{u(y_2)}^{u(y)}\frac{dv}{v^{3}}
 \le \frac{1}{2u(y_2)^{2}}\le\frac12,
\]
but this is absurd as $y$ is assumed to be unbounded ($\limsup y=\infty$).
The case $u\le-1$ is symmetric.
Hence $|u|<1$ throughout.

Now, we prove \eqref{eq:h-small}. By the previous paragraph $|\delta|\le|u|+|m|\le\tfrac32$ is bounded and, by \eqref{eq:dm-bound},
\begin{equation}\label{eq:dm-bound2}
\left|\frac{dm}{dy}\right| \le \frac{2K'}{y}\quad\mbox{for every}\quad y\ge y_1.
\end{equation}
Equation \eqref{eq:h-ode} is linear in $\delta$ with positive coefficient
\[
a(y):=(1+u^{2})\tilde B(y)\ge\tfrac12|c|y.
\] 
Writing $\Theta(y):=\int_{y_1}^{y}a(r)dr$, integration of \eqref{eq:h-ode} gives
\begin{equation}\label{eq:h}
\delta(y)=e^{\Theta(y)}\left[\delta(y_1)-\int_{y_1}^{y}e^{-\Theta(t)}\frac{dm}{dy}(t)dt\right].
\end{equation}
Since 
\begin{equation}\label{eq:bound-A}
\Theta(t)\ge\tfrac{|c|}4(t^{2}-y_1^{2})\ge\tfrac{|c|y_1}{2}(t-y_1),
\end{equation}
the integral in \eqref{eq:h} converges as $y\to\infty$, to a limit $\ell$ say. If $\delta(y_1)\ne\ell$ then
\[
|\delta(y)|\sim|\delta(y_1)-\ell| e^{\Theta(y)}\to\infty,
\] 
contradicting the boundedness of $\delta$. Hence $\delta(y_1)=\ell$, and by \eqref{eq:dm-bound2} and \eqref{eq:bound-A},
\[
|\delta(y_1)|\ \le\ \sup_{t\ge y_1}\left|\frac{dm}{dy}\right|\int_{y_1}^{\infty}e^{-\frac{|c|y_1}{2}(t-y_1)}dt
\ \le\ \frac{2K'}{y_1}\cdot\frac{2}{|c| y_1}
\ =\ \frac{4K'}{|c| y_1^{2}} .
\]
The argument applies verbatim with $y_1$ replaced by any $y_0\ge y_1$, since \eqref{eq:dm-bound2} and \eqref{eq:bound-A} (with $y_0$ in place of $y_1$) remain valid on $[y_0,\infty)$. Hence $|\delta(y_0)|\le4K'/(|c| y_0^{2})$ for every $y_0\ge y_1$, which is \eqref{eq:h-small} with $C_0:=4K'/|c|$, and proves the claim.

\emph{Step 3. $x\to x_*$ as $y\to\infty$.} Substituting $\tan\tau=u=m+\delta$ into the second equation of \eqref{eq:fuga-red}, we obtain
\begin{equation}\label{eq:xdrift}
\frac{dx}{dy}=-\tan\tau=-\frac{A(x)}{\tilde B(y)}-\delta.
\end{equation}
Let $D:=x-x_*$. Since $c<0$ and $x_*^{2}=p/|c|$,
\[
A(x)=\frac{p+cx^{2}}{x}=-\frac{|c|(x^{2}-x_*^{2})}{x}=-|c|\frac{x+x_*}{x} D,
\]
so that \eqref{eq:xdrift} reads
\[
\frac{dD}{dy}=\frac{|c|}{\tilde B(y)}\frac{x+x_*}{x} D-\delta .
\]
Now $\dfrac{x+x_*}{x}\ge1$ and $\tilde B(y)=|c|y-\dfrac qy\le|c|y$, whence
\[
\frac{|c|}{\tilde B(y)}\frac{x+x_*}{x}\ \ge\ \frac1y,
\]
and multiplying by $D$ reverses the inequality when $D<0$. Since $|\delta|\le C_0/y^{2}$ by
\eqref{eq:h-small}, we obtain, for $y\ge y_1$,
\begin{equation}\label{eq:Ddrift}
\frac{dD}{dy} \ge \frac{D}{y}-\frac{C_0}{y^{2}}\quad\mbox{where }D>0,\quad\mbox{and}\quad
\frac{dD}{dy} \le \frac{D}{y}+\frac{C_0}{y^{2}}\quad\mbox{where }D<0 .
\end{equation}
Suppose that $D(y_2)\ge2C_0/y_2$ for some $y_2\ge y_1$. At the points where $D\ge2C_0/y$ one has $C_0/y^{2}\le D/(2y)$, hence
\[
\frac{dD}{dy} \ge \frac{D}{y}-\frac{D}{2y}=\frac{D}{2y} > 0 > \frac{d}{dy}\left(\frac{2C_0}{y}\right).
\]
Therefore $D\ge2C_0/y$ on $[y_2,\infty)$, and integrating $dD/dy\ge D/(2y)$ gives $D(y)\ge D(y_2)(y/y_2)^{1/2}\to\infty$, contradicting $x\le b$. Symmetrically, if $D(y_2)\le-2C_0/y_2$, then $C_0/y^2\le-D/(2y)$ and
\[
\frac{dD}{dy} \le \frac{D}{y}-\frac{D}{2y}=\frac{D}{2y} < 0 < \frac{d}{dy}\left(-\frac{2C_0}{y}\right)
\]
wherever $D\le-2C_0/y$, so that $D(y)\le D(y_2)(y/y_2)^{1/2}\to-\infty$, contradicting $x\ge a>0$. Hence $|x-x_*|<2C_0/y$ for every $y\ge y_1$, and $x\to x_*$.

\smallskip
\emph{Step 4 (the trajectory cannot approach $x_*$).} 
By Step 3, $x\to x_*$. We show that this is impossible. We use $I$ and the fact that the zero of $A$ is the maximum point of $f(x)=x^pe^{cx^2/2}$ (see Lemma \ref{lem:bound}). 

Since $x(y)\to x_*$ as $y\to\infty$, by \eqref{eq:xdrift},
\[
\frac{dx}{dy}=-\frac{A(x)}{\tilde B(y)}-\delta,
\]
where $A(x)\to A(x_*)=0$, $\tilde B(y)\to\infty$ and $|\delta|\le C_0y^{-2}\to0$ by
\eqref{eq:h-small}. Hence $\tfrac{dx}{dy}\to0$. Since $\tfrac{dx}{dy}=\cot\sigma$ and
$\sin\sigma=\cos\tau>0$ everywhere, this forces $\cos\sigma\to0$ and $\sin\sigma\to1$, whence
\[
I=f(x)\sin\sigma\longrightarrow f(x_*)\cdot1=M_f.
\]
On the other hand, since $I$ is decreasing in $\mathcal S^+$ along the trajectories (Corollary \ref{cor:estrita}), its convergence to $M_f$ implies that $I(y)\ge M_f$ for every $y$. But $I=f(x)\sin\sigma\le f(x)\le M_f$ always. Therefore $I\equiv M_f$, which forces $f(x)\equiv M_f$ and $\sin\sigma\equiv1$, that is, $x\equiv x_*$ and $\sigma\equiv\tfrac\pi2$, i.e., $\gamma$ is the profile curve of the vertical cylinder, excluded by hypothesis.

Therefore, in all cases we obtained a contradiction, which proves the Proposition.
\end{proof}

The mirror symmetry of Lemma~\ref{lem:simetrias}(ii) gives the analogous statement for the vertical strips.

\begin{corollary}[Absence of horizontal escape: the mirror of Proposition~\ref{lem:fuga}]\label{cor:fugaH}
Let $c<0$ and let $\gamma$ be confined to $\mathcal V^{+}=\{x>x_*\}$, distinct from the two
cylinders, with $y(s)\in[a,b]\subset(0,\infty)$ near one of the endpoints. Then $x$ is bounded in that neighborhood.
\end{corollary}
\begin{proof}
It is Proposition~\ref{lem:fuga}, carried over by the mirror symmetry of Lemma~\ref{lem:simetrias}, item (ii), which sends $(x,y,\sigma;p,q)$ to $(y,x,\tfrac\pi2-\sigma;q,p)$ and, therefore, $\mathcal S^{+}$ to $\mathcal V^{+}$, $I$ to $J$, $A$ to $B$ and $y_*$ to $x_*$. Since the symmetry is a bijection between solutions of \eqref{eq:flow} with parameters $(p,q)$ and $(q,p)$, and Proposition~\ref{lem:fuga} holds for all $p,q\ge1$, the conclusion carries over without change.
\end{proof}

Combining both Proposition \ref{lem:fuga} and Corollary \ref{cor:fugaH} we obtain

\begin{corollary}[Rigidity of the cylindrical ends: equivariant case]\label{cor:rigidez-end}
Let $c<0$. No trajectory distinct from the cylinders has an endpoint along which $y\to\infty$ and $x\to r$, or $x\to\infty$ and $y\to r$, for some $r>0$. In particular, no bi-rotational self-shrinker other than the cylinders has an end asymptotic to a cylinder $\s^{p}(r)\times\R^{q+1}$ or $\s^{q}(r)\times\R^{p+1}$.
\end{corollary}
\begin{proof}
In the first case, near the endpoint $y>y_*$ and $x$ lies in a compact subinterval of $(0,\infty)$, which Proposition~\ref{lem:fuga} excludes. The second case is excluded by Corollary~\ref{cor:fugaH}.
\end{proof}

Proposition~\ref{lem:fuga} assumes $x$ confined to a compact subset of $(0,\infty)$, so it says nothing when the trajectory drifts toward a coordinate axis while escaping. That complementary case is settled next, by a different argument, and the two together exclude every unbounded end running along an axis.

\begin{proposition}[Absence of escape along an axis]\label{prop:nao-eixo}
Let $c<0$ and let $\gamma$ be a trajectory distinct from the cylinders. There is no endpoint of
$\gamma$ along which
\[
x\to\infty\quad\text{and}\quad \liminf y=0,
\]
and none along which $y\to\infty$ and $\liminf x=0$.
\end{proposition}

\begin{proof}
We treat the first assertion. The second is the mirror image. Suppose, for contradiction, that $x\to\infty$ and $\liminf y=0$ at the endpoint $\omega$. The case $\alpha$ follows by reversibility (Lemma~\ref{lem:simetrias}, item (i)).

By hypothesis, $\liminf y=0$. Since
$g(y)=y^{q}e^{cy^{2}/2}\to0$ as $y\to0^{+}$, this gives
\begin{equation}\label{eq:liminf-g}
\liminf g(y)=0 .
\end{equation}

\smallskip
\emph{Step 1. $x$ is increasing, $\cos\sigma>0$, and $J\to0$ with $J\ge0$.} Near the endpoint, $x>x_*$ and $J$ is strictly decreasing by Corollary \ref{cor:estrita}. If $x$ had local maxima at a sequence $s_k\to s_+$, then $\cos\sigma(s_k)=0$ and $J(s_k)=0$ there, contradicting that $J$ is decreasing. Hence $x$ is increasing and $\cos\sigma>0$. Since $J$ is monotone and $|J|\le g(y)$, it follows from \eqref{eq:liminf-g} that $J\to0$, and therefore $J\ge0$.

\smallskip
\emph{Step 2. The endpoint is contained in $\mathcal S^{-}$.}
Suppose that there are upward crossings of $\{y=y_*\}$ at instants $s_k\to s_+$. There $y(s_k)=y_*$, hence $J(s_k)=g(y_*)\cos\sigma(s_k)$ with $g(y_*)>0$ constant, and from $J\to0$ we get $\cos\sigma(s_k)\to0$. Since $\cos\sigma>0$ and $\sin\sigma(s_k)>0$, this gives $\sigma(s_k)\to\tfrac\pi2^{-}$. 
As $x\to\infty$ implies $A(x)\to-\infty,$ we can fix $s_0$ and $a_0>0$ such that $|A(x(s))|\ge a_0>0$ for every $s\ge s_0.$ Now, let $L_B:=\sup\{|B'(y)|; \tfrac{y_*}{2}\leq y \leq 2y_*\}$ and put
\[
\delta:=\min\left\{\frac{y_*}{2},\frac{a_0}{4L_B}\right\}>0.
\]
Since $|y'(s)|=|\sin\sigma(s)|\le 1,$ by the mean value theorem, we have $|y(s)-y_*|\leq \delta$ in $[s_k,s_k+\delta]$. Thus, using the mean value theorem again gives
\[
|B(y(s))|=|B(y(s))-B(y_*)|\le L_B|y(s)-y_*|\le L_B\delta
\]
in $[s_k,s_k+\delta],$ as $s\in[s_k,s_k+\delta]$ implies $|y(s)-y_*|\leq \frac{y_*}{2}$, i.e., $y(s) \in [\tfrac{y_*}{2},\frac{3y_*}{2}]$. So, as long as $\sigma\in[\tfrac\pi3,\tfrac{2\pi}{3}]$, $\sin\sigma\geq \tfrac12$ and
\begin{equation}\label{eq:est-step-2}    
\sigma' = |A(x)|\sin\sigma+ B(y)\cos\sigma \geq \frac{a_0}{2} - L_B\delta\geq\frac{a_0}{2}-\frac{a_0}{4}=\frac{a_0}{4}>0,
\end{equation}
which is a lower bound independent of the index $k$. Now $\sigma(s_k)\to\tfrac\pi2^{-}$, so for $k$ large
$\sigma(s_k)\in[\tfrac\pi3,\tfrac\pi2)$ and, by integrating \eqref{eq:est-step-2}, the arclength needed to cover
$\tfrac\pi2-\sigma(s_k)$ at rate at least $\tfrac{a_0}{4}$ is at most
$4\left(\tfrac\pi2-\sigma(s_k)\right)/a_0$, which tends to $0$ and is therefore $<\delta$ for $k$
large. Hence $\sigma$ attains and exceeds $\tfrac\pi2$ within $[s_k,s_k+\delta]$. Therefore,
$\cos\sigma$ becomes negative and $x$ decreases, contradicting Step~1. Hence the crossings of $\{y=y_*\}$ do not accumulate at the endpoint, so the endpoint is eventually contained in $\mathcal S^{+}$ or in $\mathcal S^{-}$. The first is impossible, since there $y>y_*$ would contradict the hypothesis $\liminf y=0$. Therefore the endpoint is contained in $\mathcal S^{-}$.

\smallskip
\emph{Step 3. $y$ is decreasing. In particular, $y\to0^+.$}
In $\mathcal S^{-}$ one has that $I$ is strictly increasing by Corollary \ref{cor:estrita}. If $y$ had local maxima at a sequence $s_k\to s_+$, there $\sin\sigma(s_k)=0$ and $I(s_k)=0$. This contradicts the fact that $I$ is increasing. Hence
$y$ is monotone, and with $\liminf y=0$ we get $y\to0^{+}$.

\smallskip
\emph{Step 4. Reduction with $y$ as independent variable.}
By Steps 1 and 3, near the endpoint $y$ is strictly decreasing and $x$ strictly increasing. So we take $y\in(0,y_1]$ as the independent variable and set
\[
v:=-\frac{dx}{dy}=\frac{\cos\sigma}{-\sin\sigma}>0 .
\]
Dividing \eqref{eq:flow} by $y'=\sin\sigma$ one obtains the system
\begin{equation}\label{eq:w-ode}
\frac{dx}{dy}=-v,
\qquad
\frac{dv}{dy}=(1+v^2)(|A(x)|-B(y)v),
\quad\mbox{where}\quad |A|:=-A>0,
\end{equation}
and the hypothesis $x\to\infty$ is equivalent to
\begin{equation}\label{eq:w-diverge}
\int_{0}^{y_1}v(y)dy=+\infty.
\end{equation}
Write $v_*:=|A(x)|/B(y)>0$, where the right-hand side of \eqref{eq:w-ode} vanishes. We record two elementary bounds: since $|A(x)|=|c|x-p/x$ and $B(y)=\tfrac qy-|c|y$, we have
\begin{equation}\label{eq:AB-cotas}
|A(x)|\le|c|x,
\quad \mbox{and}\quad
B(y)\ge\frac{q}{2y} \quad \mbox{if}\quad y\le\frac{y_*}{\sqrt2}.
\end{equation}
Fix $\hat y\in(0,y_1]$ for which Steps 1--3 and \eqref{eq:AB-cotas} hold on $(0,\hat y]$.

\smallskip
\emph{Step 5: where $v\le2v_*$, the growth of $x$ is controlled.}
At these points, by \eqref{eq:AB-cotas},
\begin{equation}\label{eq:step5}
v\le \frac{2|A(x)|}{B(y)} \le\frac{2|c|x}{q/(2y)}=\frac{4|c|}{q}xy,
\end{equation}
whence
\begin{equation}\label{eq:estlog}
\frac{d}{dy}\log x=-\frac vx\ge-\frac{4|c|}{q}y .
\end{equation}
If $v\le2v_*$ on the whole interval $(0,\hat y]$, an integration from $y$ to $\hat y$ gives
\begin{equation}\label{eq:x-limitado}
\log x(y)\le\log x(\hat y)+\frac{2|c|}{q}\hat y^{2} \quad\mbox{for}\quad 0<y\le\hat y,
\end{equation}
and $x$ is bounded, contradicting $x\to\infty$.

\smallskip
\emph{Step 6. Where $v\ge2v_*$, the descent stops at the barrier $v=2v_*$.}
At these points $B v\ge2|A|$, hence $|A|-Bv\le-\tfrac12Bv$, and since $1+v^{2}\ge v^{2}$, from \eqref{eq:w-ode} and \eqref{eq:AB-cotas} we obtain
\begin{equation}\label{eq:w-blow}
\frac{dv}{dy} \le -\tfrac12 v^{3}B\le-\frac q4\frac{v^3}{y}.
\end{equation}
Suppose that $v\ge2v_*$ on a whole interval $(\beta,y_2]$. Separating variables in
\eqref{eq:w-blow} and integrating from $y$ to $y_2$, we have
\begin{equation}\label{eq:w-explode}
0 \le\frac{1}{2v(y)^{2}}\le \frac{1}{2v(y_2)^{2}}-\frac q4\log\frac{y_2}{y}
\quad\mbox{for}\quad \beta<y\le y_2.
\end{equation}
The right-hand side vanishes at $y_0:=y_2\exp\left(-2/(qv(y_2)^{2})\right)>0$ and is negative
below it. Hence $\beta\ge y_0>0$. Near the endpoint $I$ is strictly monotone with at most one zero, and $y$ decreases (Step~3), so $\sin\sigma<0$ on $(0,\hat y]$. Thus $v$ is finite and continuous there, and $v(\beta)=2v_*(\beta)$.

\smallskip
\emph{Step 7. Conclusion: Steps 5 and 6 imply that $x$ is bounded, reaching a contradiction.}
Let $\mathcal J:=\{y\in(0,\hat y]:v(y)\le2v_*(y)\}$. If $\mathcal J=(0,\hat y]$, Step~5 already gives the
contradiction. Otherwise, since the complement $\{y\in(0,\hat y]:v(y)>2v_*(y)\}$ is relatively open in $(0,\hat y],$ it is a countable union of disjoint open intervals
$(\beta_i,\gamma_i)$, and by Step~6 one has $\beta_i>0$ with $v(\beta_i)=2v_*(\beta_i)$.

On each $(\beta_i,\gamma_i)$ one has $v>2v_*>v_*$. Hence $dv/dy<0$ by \eqref{eq:w-ode}, i.e., the function $v$ is decreasing in $y$ and attains its maximum at the left endpoint. Thus, $v\le v(\beta_i)$ there, and by \eqref{eq:step5} applied at $\beta_i$, we obtain
\[
v(\beta_i)=2v_*(\beta_i)\le\frac{4|c|}{q}x(\beta_i)\beta_i.
\]
Therefore
\[
x(\beta_i)-x(\gamma_i)=\int_{\beta_i}^{\gamma_i}v dy
\ \le\ \frac{4|c|}{q} x(\beta_i)\beta_i(\gamma_i-\beta_i)
\ =:\ u_i x(\beta_i),
\]
that is, $x(\gamma_i)/x(\beta_i)\ge1-u_i$. Choosing $\hat y$ small enough that
$u_i\le\tfrac12$, and using $-\log(1-u)\le2u$ for $u\in[0,\tfrac12]$, we get
\[
\log\frac{x(\beta_i)}{x(\gamma_i)}\ \le\ 2u_i
=\frac{8|c|}{q}\beta_i(\gamma_i-\beta_i)
\ \le\ \frac{8|c|}{q}\int_{\beta_i}^{\gamma_i}t dt,
\]
where the last inequality is $\beta(\gamma-\beta)\le\tfrac12(\gamma^{2}-\beta^{2})$. Since the
intervals are disjoint,
\begin{equation}\label{eq:sumlog}
\sum_i\log\frac{x(\beta_i)}{x(\gamma_i)} \le \frac{8|c|}{q}\int_0^{\hat y}t dt
=\frac{4|c|}{q}\hat y^2.
\end{equation}
Now, since $\tfrac{d}{dy}\log x=-v/x$, integrating from $y$ to $\hat y$ we have
\[
\begin{aligned}
\log x(y)-\log x(\hat y)&=-\int_{y}^{\hat y}\tfrac{d}{dt}\log x dt\\
&=\int_{[y,\hat y]\cap \mathcal J}\frac vxdt+ \sum_i\int_{[y,\hat y]\cap(\beta_i,\gamma_i)}\left(-\tfrac{d}{dt}\log x\right) dt.\\
\end{aligned}
\]
Using \eqref{eq:estlog}, the first integral on the right-hand side can be estimated as
\[
\int_{[y,\hat y]\cap \mathcal J}\frac vx dt \le \frac{4|c|}{q}\int_{[y,\hat y]\cap \mathcal J}t dt \le \frac{4|c|}{q}\int_{0}^{\hat y}t dt
=\frac{2|c|}{q}\hat y^2.
\]
By \eqref{eq:sumlog}, the sum of integrals satisfies the estimate
\[
\begin{aligned}
\sum_i\int_{[y,\hat y]\cap(\beta_i,\gamma_i)}\left(-\tfrac{d}{dt}\log x\right)dt &\le \sum_i\int_{\beta_i}^{\gamma_i}\left(-\tfrac{d}{dt}\log x\right)dt\\
&=\sum_i\log\frac{x(\beta_i)}{x(\gamma_i)}\le \frac{4|c|}{q}\hat y^2.
\end{aligned}
\]
Adding the two bounds gives
\[
\log x(y) \le \log x(\hat y)+\frac{6|c|}{q}\hat y^2\quad\mbox{for all}\quad y\in(0,\hat y].
\]
Thus, $x$ is bounded, contradicting $x\to\infty$, or equivalently \eqref{eq:w-diverge}.
\end{proof}

An end of $\gamma$ is called non-escaping if it has infinite arclength and $\rho\not\to\infty$ along it. Recurrence, when used below, refers to the phase-space flow. The next result shows that a non-escaping end visits both strips infinitely often.

\begin{theorem}[Non-escaping ends cross both separatrices infinitely often]\label{prop:recorrente}
Let $c<0$ and let $\gamma$ be a trajectory distinct from the cylinders. If an end of $\gamma$ crosses $\{y=y_*\}$ only finitely many times, then it has finite arclength or $\rho\to\infty$ along it. The same holds for $\{x=x_*\}$. Consequently, every non-escaping end crosses each of the two separatrices infinitely often.
\end{theorem}

\begin{proof}
By reversibility we treat the end $s\to s_+$. The assertion for $\{x=x_*\}$ then follows from the mirror symmetry of Lemma~\ref{lem:simetrias}(ii) and the last assertion is the contrapositive of the first two. Since every contact of $\gamma$ with $\{y=y_*\}$ is transverse (Lemma~\ref{lem:transv}), an end with finitely many crossings of $\{y=y_*\}$ satisfies $\gamma((s_1,s_+))\subset\mathcal S^{+}$ or $\gamma((s_1,s_+))\subset\mathcal S^{-}$, for $s_1$ after the last crossing. We suppose that such an end is non-escaping and derive a contradiction.

\emph{Step 1: the extrema of $x$.} Suppose first that $x$ and $y$ are both monotone on some interval $(s_2,s_+)$. If one of them is unbounded, then it tends to $\infty$ and so does $\rho$. If both are bounded, then both converge, and the arclength of $(s_2,s_+)$ is at most the sum of their total variations, because $|x'|+|y'|=|\cos\sigma|+|\sin\sigma|\ge1$. In either case the end fails to be non-escaping.

By Remark~\ref{rem:Jfront-end}, $I$ is strictly monotone on $(s_1,s_+)$, and it is bounded by Lemma~\ref{lem:bound}, so it converges to some $I_+$. Being strictly monotone, $I$ has at most one zero on $(s_1,s_+)$. The zeros of $I$ are those of $\sin\sigma=y'$, so $y$ is monotone on some $(s_2,s_+)$. By the previous paragraph, $x$ is monotone on no such interval, that is, $x'=\cos\sigma$ changes sign infinitely often. Its zeros are isolated by Lemma~\ref{lem:Z}(ii) and between two consecutive zeros $x'$ keeps a constant sign, so these zeros form a sequence $t_k\to s_+$ at which $x$ attains local extrema, maxima and minima alternating.

\emph{Step 2: $I_+=0$.} Suppose that $I_+\ne0$. Since $|I|\le f(x)$, we have $f(x)\ge|I_+|/2$ on some $(s_2,s_+)$. By Lemma~\ref{lem:bound}, $f$ tends to $0$ both at $0$ and at $\infty$, so the set $\{x\in(0,\infty); f(x)\ge|I_+|/2\}$ is a compact subinterval $[a,b]$ of $(0,\infty)$, and $x$ remains in $[a,b]$. By Step~1, $y$ is eventually monotone, and it converges to some $y_\infty\in[0,\infty]$. If $y_\infty=\infty$, then $\rho\to\infty$. If $y_\infty=0$, then $s_+<\infty$ by Lemma~\ref{rem:limites}, since $x$ stays in $[a,b]$. If $y_\infty\in(0,\infty)$, then $x$ remains in $[a,b]$ and $y$ in a compact subinterval of $(0,\infty)$, which is scenario (F2) of Proposition~\ref{lem:Jfront}, excluded by Remark~\ref{rem:Jfront-end}. All three cases contradict the assumption that the end is non-escaping. Hence $I_+=0$, and $I$, being strictly monotone with limit $0$, has constant sign on $(s_1,s_+)$.

\emph{Step 3: the case $\mathcal S^{-}$.} Here $B>0$ on the strip, so $I$ is increasing, and by Step~2 it increases to $0$. Thus $I<0$, that is, $\sin\sigma<0$, and $y$ decreases to some $y_\infty\in[0,y_*)$.

We first locate the extrema of Step~1. At $t_k$ we have $\cos\sigma=0$ and $\sin\sigma=-1$, hence
\[
\sigma'=-A(x)\sin\sigma+B(y)\cos\sigma=A(x),\qquad
x''=(\cos\sigma)'=-\sin\sigma\sigma'=A(x).
\]
At a maximum $x''\le0$, so $A(x)\le0$. The equality $A(x(t_k))=0$ would give $\Gamma(t_k)=(x_*,y(t_k),\tfrac{3\pi}{2})$, a point of the trajectory of the vertical cylinder, and $\gamma$ would be that cylinder by uniqueness of solutions. Hence $x(t_k)>x_*$ at every maximum, and in the same way $x(t_k)<x_*$ at every minimum. Moreover $|I(t_k)|=f(x(t_k))\to0$, and $f$ decreases to $0$ on $(x_*,\infty)$ by Lemma~\ref{lem:bound}, so, writing $\tau_k$ for the maxima among the $t_k$, the maximum values $\hat x_k:=x(\tau_k)$ tend to $\infty$.

Now let $T_k<\tau_k$ be a minimum and the next maximum, with $\hat x_k\ge2x_*$, which holds for all large $k$. On $(T_k,\tau_k)$ there is no zero of $\cos\sigma$, so $\cos\sigma>0$ and $x$ increases from $x(T_k)<x_*$ to $\hat x_k$. Let $u_k\in(T_k,\tau_k)$ be the instant where $x(u_k)=x_*$. On $[u_k,\tau_k]$ we have $x_*\le x\le \hat x_k$, hence
\[
A(x)\le0,\qquad |A(x)|=|c|x-\frac px\le|c|\hat x_k .
\]
On $[u_k,\tau_k)$ both $\cos\sigma>0$ and $\sin\sigma<0$, so the continuous lift of $\sigma$ may be chosen in $(\tfrac{3\pi}{2},2\pi)$, with $\sigma(\tau_k)=\tfrac{3\pi}{2}$. Writing $\eta:=-\sin\sigma>0$ and integrating $\sigma'=-A\sin\sigma+B\cos\sigma=-|A|\eta+B\cos\sigma$ over $[u_k,\tau_k]$,
\[
\begin{aligned}
\int_{u_k}^{\tau_k}\left(B\cos\sigma-|A|\eta\right) ds
&=\sigma(\tau_k)-\sigma(u_k)<0,
\end{aligned}
\]
that is,
\[
\int_{u_k}^{\tau_k}|A|\eta ds>\int_{u_k}^{\tau_k}B\cos\sigma ds .
\]
On the tail $y\le y(s_1)<y_*$, and $B$ decreases on $(0,y_*)$, so $B(y)\ge b_0:=B(y(s_1))>0$. Using $\int_{u_k}^{\tau_k}\cos\sigma ds=\hat x_k-x_*$ on the right and $|A|\le|c|\hat x_k$ on the left, and setting $\Delta_k:=\int_{u_k}^{\tau_k}\eta ds=y(u_k)-y(\tau_k)$, we obtain
\[
\begin{aligned}
|c| \hat x_k\Delta_k\ \ge\ \int_{u_k}^{\tau_k}|A|\eta ds\ >\ b_0(\hat x_k-x_*),
\end{aligned}
\]
hence
\[
\Delta_k>\frac{b_0}{|c|}\left(1-\frac{x_*}{\hat x_k}\right)\ \ge\ \frac{b_0}{2|c|} .
\]
The intervals $[u_k,\tau_k]$ are pairwise disjoint and $y$ decreases along each of them, so the sum of the $\Delta_k$ is at most $y(s_1)$. This is impossible, because there are infinitely many of them and each exceeds $b_0/(2|c|)$.

\emph{Step 4: the case $\mathcal S^{+}$.} Here $B<0$ on the strip, so $I$ is decreasing, and by Step~2 it decreases to $0$. Thus $I>0$, that is, $\sin\sigma>0$, and $y$ increases to some $y_\infty\in(y_*,\infty]$. The value $y_\infty=\infty$ gives $\rho\to\infty$ and is excluded by the non-escaping assumption, so $y_\infty<\infty$.

At an extremum $t_k$ we now have $\sin\sigma=1$, hence $\sigma'=-A(x)$ and $x''=-\sin\sigma\sigma'=A(x)$, and the argument of Step~3 applies verbatim. The maxima occur at $x>x_*$, the minima at $x<x_*$, and the maximum values $\hat x_k$ tend to $\infty$. With $T_k<u_k<\tau_k$ as before, on $[u_k,\tau_k)$ both $\cos\sigma>0$ and $\sin\sigma>0$, so the lift of $\sigma$ may be chosen in $(0,\tfrac\pi2)$, with $\sigma(\tau_k)=\tfrac\pi2$. Writing now $\eta:=\sin\sigma>0$ and integrating $\sigma'=|A|\eta-|B|\cos\sigma$,
\[
\int_{u_k}^{\tau_k}\left(|A|\eta-|B|\cos\sigma\right) ds=\sigma(\tau_k)-\sigma(u_k)>0 .
\]
On the tail $y\ge y(s_1)>y_*$, and $|B(y)|=|c|y-q/y$ increases on $(y_*,\infty)$, so $|B(y)|\ge b_0:=|B(y(s_1))|>0$. Exactly as in Step~3, with $\Delta_k:=\int_{u_k}^{\tau_k}\eta ds=y(\tau_k)-y(u_k)$,
\[
|c| \hat x_k\Delta_k\ \ge\ \int_{u_k}^{\tau_k}|A|\eta ds\ >\ b_0(\hat x_k-x_*),
\]
and therefore $\Delta_k>b_0/(2|c|)$ again.
Since $y$ increases and the intervals are disjoint, $y_\infty\ge y(s_1)+\sum_k\Delta_k=\infty$, contradicting $y_\infty<\infty$.
\end{proof}

\subsection{Terminal excursions}\label{sec:terminais}

Theorem~\ref{thm:unico} bounds the multiplicity of loops in a strip. The proposition of this subsection eliminates them entirely in a geometrically natural class of excursions, by confronting the sign that the zero of $I$ imposes with the limit value of $I$ at the endpoint.

\begin{definition}[Terminal and intermediate excursions]\label{def:terminal}
Let $c<0$ and let $\gamma$ be a trajectory with maximal interval $(s_-,s_+)$. A horizontal excursion is a maximal subarc contained in a strip $\mathcal S^{\pm}$. It is called terminal if one of its endpoints is an endpoint of $\gamma$ (i.e.\ $s_-$ or $s_+$), and intermediate if both its endpoints are crossings of the level $\{y=y_*\}$. Analogously for
vertical excursions in $\mathcal V^{\pm}$, with $J$ and the level $\{x=x_*\}$.
\end{definition}

In a terminal excursion, the limit value of $I$ at the endpoint of $\gamma$ fixes the sign of $\sin\sigma$, and this leaves no room for a loop.

\begin{proposition}[Absence of loops in terminal excursions]\label{thm:terminal}
Let $c<0$ and let $\gamma$ be a trajectory distinct from the cylinders. Then no terminal
excursion contains a loop.
Consequently, every loop confined to a strip is contained in an intermediate excursion.
\end{proposition}

\begin{proof}
Let $\mathcal E$ be a terminal excursion containing a loop. By Lemma~\ref{prop:restricoes}, there is $t_I$ in the interior of the loop's subarc with $\sin\sigma(t_I)=0$, hence
\begin{equation}\label{eq:zeroJ}
I(t_I)=f\left(x(t_I)\right)\sin\sigma(t_I)=0 .
\end{equation}
Denote by $I_-$ and $I_+$ the limit values of $I$ at the endpoints $\alpha$ and $\omega$, i.e., $I_\pm=\lim_{s\to s_\pm}I(s).$ Proposition~\ref{lem:Jfront} enumerates the possible values. Scenario (F2) does not occur, (F1) gives the value $I_{\pm}=0$, and (F3), the endpoint on the $x$-axis, is only possible in $\mathcal S^{-}$ (since in $\mathcal S^{+}$ one has $y>y_*>0$) and gives $f(x_0)>0$ at the endpoint $\alpha$ and $-f(x_0)<0$ at the endpoint $\omega$. In $\mathcal S^{+}$, the vertical escape, the only remaining scenario, is excluded by Propositions~\ref{lem:fuga} and \ref{prop:nao-eixo} and therefore there the limit value is necessarily $0$.

\smallskip
\emph{Case $\mathcal E\subset\mathcal S^{-}$.} There $y<y_*$ and $I$ is strictly increasing along $\mathcal E$ (Corollary~\ref{cor:estrita}, item (i)). Combining this
with \eqref{eq:zeroJ}:
\begin{itemize}
\item if $\mathcal E$ reaches the endpoint $\alpha$, then $I(s)<0$ for all $s<t_I$ in $\mathcal E$, whence
$I_-\le I(s)<0$. But the admissible values at an endpoint $\alpha$ in $\mathcal S^{-}$ are $0$ (F1) and $f(x_0)>0$ (F3), both $\ge0$. A contradiction.
\item if $\mathcal E$ reaches the endpoint $\omega$, then $I(s)>0$ for all $s>t_I$ in $\mathcal E$, whence
$I_+\ge I(s)>0$. But the admissible values at an endpoint $\omega$ in $\mathcal S^{-}$ are $0$ (F1) and $-f(x_0)<0$ (F3), both $\le0$. A contradiction.
\end{itemize}

\smallskip
\emph{Case $\mathcal E\subset\mathcal S^{+}$.} There $y>y_*$ and $I$ is strictly decreasing along $\mathcal E$. As observed above, any limit value of $I$ at an endpoint contained
in $\mathcal S^{+}$ is $0$. Combining this with \eqref{eq:zeroJ}:
\begin{itemize}
\item if $\mathcal E$ reaches the endpoint $\omega$, then $I(s)<0$ for $s>t_I$, whence $I_+\le I(s)<0\ne0$. A contradiction.
\item if $\mathcal E$ reaches the endpoint $\alpha$, then $I(s)>0$ for $s<t_I$, whence $I_-\ge I(s)>0\ne0$. A contradiction.
\end{itemize}

\smallskip
Since the four alternatives exhaust the cases, no terminal excursion contains a loop. The last assertion follows by exclusion: every excursion is terminal or intermediate
(Definition~\ref{def:terminal}). The mirror statement is obtained by the symmetry of
Lemma~\ref{lem:simetrias}(ii).
\end{proof}

By Proposition~\ref{thm:terminal}, the terminal excursions have no self-intersections. This means that the portions of $\Sigma$ corresponding to these excursions are embedded. So we have

\begin{corollary}[The ends are embedded]\label{cor:end-mergulhado}
Let $c<0$ and let $\gamma$ be a trajectory distinct from the cylinders. Then the restriction of $\gamma$ to each terminal excursion is injective. Consequently, each end of $\Sigma$, in the sense of the portion corresponding to a terminal excursion, is embedded.
\end{corollary}
\begin{proof}
The injectivity follows from Proposition~\ref{thm:terminal}. Let $\mathcal E$ be a terminal excursion in a strip $\mathcal S^{\pm}$. We claim that $\gamma|_\mathcal E$ is proper as a map into the open strip $\mathcal S^{\pm}\subset\mathcal Q$. At the nonterminal endpoint of $\mathcal E$, the curve approaches the level $\{y=y_*\}$, which is outside the open strip. At the terminal endpoint, if the arclength is infinite, then $\rho\to\infty$ by Theorem~\ref{prop:recorrente}, since there are no further crossings of $\{y=y_*\}$. If the arclength is finite, the unit-speed curve has a limit in the closed quadrant. By maximality of the solution of \eqref{eq:flow}, that limit must lie on a coordinate axis, possibly at the origin. Thus neither endpoint can accumulate in a compact subset of $\mathcal S^{\pm}$. Consequently $\gamma|_\mathcal E$ is a proper injective immersion into the strip and hence an embedding. The product parametrization \eqref{eq:param} then gives an embedding of the corresponding regular portion of $\Sigma$, since $x,y>0$ there. At an endpoint on an axis, the smooth closure is provided by Proposition~\ref{lem:suave}. The Clifford vertex, when present, is not part of the regular hypersurface.
\end{proof}

Forbidding loops in the terminal excursions localizes every confined loop in the interior of the sequence of excursions, and therefore forces crossings on both sides of it.

\begin{corollary}[Localization of confined loops]\label{cor:localiza}
Let $c<0$, $\gamma$ be a trajectory distinct from the cylinders, and $\ell$ be a
loop of $\gamma$ confined to a horizontal strip $\mathcal S^{\pm}$. Then,
\begin{enumerate}
\item[(a)] $\ell$ is contained in an intermediate excursion of that strip, never in a
terminal one;
\item[(b)] consequently, $\gamma$ crosses the level $\{y=y_*\}$ at least twice, once
before and once after $\ell$ (the two endpoints of the excursion containing
it);
\item[(c)] if, moreover, $\ell$ is confined to a vertical strip $\mathcal V^{\pm}$, that
is, to a quadrant, then $\gamma$ also crosses $\{x=x_*\}$ at least twice, by the
mirror statement applied to $J$.
\end{enumerate}
In particular, a trajectory with $N_y\le1$ has no loop confined to a horizontal strip,
and one with $\min(N_x,N_y)\le1$ has no loop confined to a quadrant.
\end{corollary}
\begin{proof}
(a) By Proposition~\ref{thm:terminal}, no terminal excursion contains a loop. Since every excursion is terminal or intermediate (Definition~\ref{def:terminal}), the loop $\ell$ is in an intermediate one.

(b) By definition, the two endpoints of an intermediate excursion are crossings of the level
$\{y=y_*\}$ (Definition~\ref{def:terminal}). Hence, there is a crossing before and another after
$\ell$, and $N_y\ge2$. 

(c) If $\ell$ is also confined to a vertical strip, the mirror statement of Proposition~\ref{thm:terminal}, with $J$ in place of $I$ and $\{x=x_*\}$ in
place of $\{y=y_*\}$, places $\ell$ in an intermediate vertical excursion, whose endpoints are crossings of $\{x=x_*\}$. Hence $N_x\ge2$. The final assertions are the contrapositive: $N_y\le1$ prevents step (b), and $\min(N_x,N_y)\le1$ prevents (b) or (c).
\end{proof}

\subsection{Counting by crossings}

The results above say where a loop may sit, not how many turns a trajectory may make. Inside a strip the tangent sweeps less than a full turn, so winding can only be accumulated by passing from one strip to the next, and counting the crossings bounds the winding.

\begin{definition}[Winding and crossings]\label{def:winding}
For a compact subarc $\gamma|_{[a,b]}$ of a trajectory of \eqref{lambda-s} we set
\[
n[a,b]:=\left\lfloor\frac{|\sigma(b)-\sigma(a)|}{2\pi}\right\rfloor,
\]
the number of full turns of the net rotation of the tangent along it. If $\gamma$ is a closed orbit, $n(\gamma)$ denotes this quantity over one fundamental period. For a nonclosed trajectory we write $n(\gamma):=\left\lfloor|\Delta\sigma_{\mathrm{total}}|/2\pi\right\rfloor$, where $\Delta\sigma_{\mathrm{total}}$ is the total rotation between the two endpoints, whenever the tangent angle has limits at both of them. Otherwise no global winding number is attached to $\gamma$. Following Hsiang \cite[\S3]{Hsiang}, we call $n(\gamma)$ the winding number of $\gamma$. In the constant mean curvature case the limit directions are known a priori, so $\Delta\sigma$ itself determines an integer, whereas here the integer part is needed.
\end{definition}

The first step bounds the rotation of the tangent inside a single excursion, using only the monotonicity of $I$ in a strip.

\begin{lemma}[Bound per excursion]\label{lem:excursao}
Let $c<0$ and consider a horizontal excursion of $\gamma$, i.e., a maximal subarc contained in a strip $\mathcal S^{\pm}$. Then $|\Delta\sigma|<2\pi$ along the excursion. The same holds for vertical excursions in $\mathcal V^{\pm}$.
\end{lemma}
\begin{proof}
In the strip $I$ is strictly monotone (Corollary~\ref{cor:estrita}, item (i)) and $\sgn I=\sgn(\sin\sigma)$ changes at most once. If $|\Delta\sigma|\ge2\pi$, then $\sigma$ sweeps a closed interval of length at least $2\pi$, which contains two distinct multiples of $\pi$. Hence $\sin\sigma$, and therefore $I$, would vanish at two distinct instants of the excursion, which is a contradiction. The vertical case is the mirror image, with $J$ and $\cos\sigma$.
\end{proof}

At the two ends of a nonclosed trajectory the bound improves to half a turn, because the limit values of $I$ given by Proposition~\ref{lem:Jfront} fix the sign of $\sin\sigma$ along a terminal excursion.

\begin{lemma}[Bound per terminal excursion]\label{lem:excterminal}
Let $c<0$ and consider a terminal horizontal excursion of $\gamma$, that is, one that
reaches one of the endpoints of $\gamma$. Then, along it, $\sin\sigma$ has constant sign and
\[
|\Delta\sigma|< \pi .
\]
The same holds, with $J$ and $\cos\sigma$, for terminal vertical excursions.
\end{lemma}
\begin{proof}
Let $\mathcal E$ be the excursion and $I_{\mathrm{pt}}$ the limit value of $I$ at the endpoint $\mathcal E$
reaches. As in the proof of Proposition~\ref{thm:terminal}, Proposition~\ref{lem:Jfront} determines the
admissible values, and the monotonicity of $I$ on $\mathcal E$ propagates the sign from the endpoint.
Precisely, in the four cases:
\begin{itemize}
\item $\mathcal E\subset\mathcal S^{-}$ with endpoint $\alpha$: $I$ is increasing, hence $I\ge
I_{\mathrm{pt}}$ on $\mathcal E$, and $I_{\mathrm{pt}}\in\{0,f(x_0)\}$ is $\ge0$. Hence $I\ge0$ on $\mathcal E$.
\item $\mathcal E\subset\mathcal S^{-}$ with endpoint $\omega$: $I$ is increasing, hence
$I\le I_{\mathrm{pt}}$, and $I_{\mathrm{pt}}\in\{0,-f(x_0)\}$ is $\le0$. Hence $I\le0$ on $\mathcal E$.
\item $\mathcal E\subset\mathcal S^{+}$ with endpoint $\alpha$: $I$ is decreasing, hence
$I\le I_{\mathrm{pt}}=0$.
\item $\mathcal E\subset\mathcal S^{+}$ with endpoint $\omega$: $I$ is decreasing, hence
$I\ge I_{\mathrm{pt}}=0$.
\end{itemize}
Since $f>0$, one has $\sgn I=\sgn(\sin\sigma)$, and in all cases $\sin\sigma$ has constant sign (in the weak sense) along $\mathcal E$.

Suppose $\sin\sigma\ge0$ on $\mathcal E$ (the other case is identical). The set $\{\vartheta\in\R:\sin\vartheta\ge0\}$ is the disjoint union of the closed intervals $[2k\pi,(2k+1)\pi]$, $k\in\Z$, each of length $\pi$. Since $\sigma$ is continuous on $\mathcal E$ and does not leave this set, it remains in a single component $[2k\pi,(2k+1)\pi]$. Hence $|\Delta\sigma|\le\pi$ immediately.

The inequality is strict: the endpoint of $\mathcal E$ that is not an endpoint of $\gamma$ is a
crossing of the level $\{y=y_*\}$, and the crossings are transverse (Lemma~\ref{lem:transv}), i.e.\ there $y'=\sin\sigma\ne0$. Hence at that endpoint $\sigma$ is in the interior of $[2k\pi,(2k+1)\pi]$, and the variation to the other endpoint is less than the length of the interval.
\end{proof}

Adding these bounds over the excursions, which are separated by the crossings, gives the main estimate of this section. Figure~\ref{fig:duas} shows the crossings counted in it.

\begin{figure}[htbp]
\centering
\includegraphics[width=0.62\textwidth]{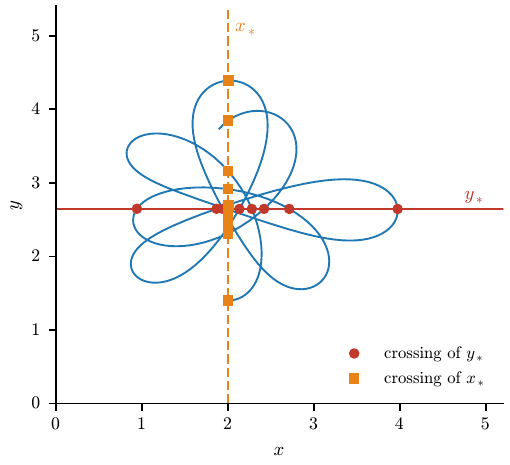}
\caption{A trajectory crossing both separatrices, for $p=4$, $q=7$, $c=-1$, with initial condition $(2,1.4,0)$ and $s\in[0,25.96]$. The crossings of $\{y=y_*\}$ (red circles) and of $\{x=x_*\}$ (orange squares) are marked. The initial point lies on $\{x=x_*\}$, since $x_*=2$.}
\label{fig:duas}
\end{figure}

\begin{theorem}[Counting: winding vs.\ crossings]\label{prop:contagem}
Let $c<0$ and let $\gamma$ be a trajectory distinct from the cylinders, with $N_y$ crossings of
$\{y=y_*\}$ and $N_x$ of $\{x=x_*\}$, and with a winding number in the sense of Definition~\ref{def:winding}, that is, either a closed orbit or a trajectory whose tangent angle has limits at both endpoints. For a closed orbit, $N_x$ and $N_y$ count the crossings over one period. Then
\begin{equation}\label{eq:contagem}
n(\gamma) \le \min(N_x,N_y)-1 .
\end{equation}
Equivalently, arbitrarily large winding requires an arbitrarily large number of crossings of both separatrices. In particular, a trajectory with $\min(N_x,N_y)\le1$ makes no full
turn.
\end{theorem}
\begin{proof}
To prove the bound by $N_y$ we may assume $N_y<\infty$, the bound being void otherwise. Note first that $N_y\ge1$, by Theorem~\ref{thm:cross} below. The $N_y$ crossings, transverse and
isolated (Lemma~\ref{lem:transv}), divide $\gamma$ into horizontal excursions, and we sum the winding excursion by excursion.

\emph{Case 1: $\gamma$ is not closed.} There are exactly $N_y+1$ excursions, of which two are terminal (the one containing the endpoint $\alpha$ and the one containing the endpoint $\omega$) and $N_y-1$ are intermediate. By Lemma~\ref{lem:excursao}, each intermediate one contributes $|\Delta\sigma|<2\pi$, and by Lemma~\ref{lem:excterminal}, each terminal one contributes $|\Delta\sigma|<\pi$. Adding,
\[
|\Delta\sigma_{\text{total}}|<2\pi(N_y-1)+2\cdot\pi =2\pi N_y,
\]
whence $|\Delta\sigma_{\text{total}}|/2\pi<N_y$ and the number of full turns is $n(\gamma)\le N_y-1$.

\emph{Case 2: $\gamma$ is closed.} There are no endpoints and the $N_y$ crossings divide the closed curve
into $N_y$ arcs, each contained in a strip. Thus, Lemma~\ref{lem:excursao} gives
$|\Delta\sigma_{\text{total}}|<2\pi N_y$, with the same conclusion.

In both cases, the mirror argument with $J$ and the vertical excursions gives
$n(\gamma)\le N_x-1$ when $N_x<\infty$, that bound being void otherwise (see Figure~\ref{fig:winding}). Combining the two bounds gives \eqref{eq:contagem}.
\end{proof}

Since the crossings are transverse, they correspond to the connected components of the intersection of $\Sigma$ with the cylinders, and Theorem~\ref{prop:contagem} becomes a statement about $\Sigma$.

\begin{corollary}[Winding costs intersections]\label{cor:custo}
Let $c<0$ and let $\Sigma$ be a bi-rotational self-shrinker whose profile curve has a winding number in the sense of Definition~\ref{def:winding}. If $n(\gamma)=k$, then $\Sigma$
meets \emph{each} of the two cylinders in at least $k+1$ distinct connected components. In particular, a
self-shrinker meeting one of the cylinders in a single connected component has a profile curve with total turning
less than $2\pi$.
\end{corollary}

\begin{proof}
Immediate from $n(\gamma)\le\min(N_x,N_y)-1$ (Theorem~\ref{prop:contagem}), the
crossings of $\{x=x_*\}$ and $\{y=y_*\}$ being transverse and hence in bijection with the connected components in which $\Sigma$ meets $\mathcal C_{\mathrm v}$ and $\mathcal C_{\mathrm h}$, respectively.
\end{proof}

The same argument applies to a compact subarc, with no condition at its ends, at the cost of one more turn.

\begin{corollary}[Finite portions]\label{cor:porcao}
Let $c<0$ and let $\gamma|_{[a,b]}$ be a compact subarc of a trajectory distinct from the cylinders. Then the tangent angle satisfies
\[
|\Delta\sigma|<2\pi\left(\min(N_x[a,b],N_y[a,b])+1\right)
\]
on $[a,b]$. In particular, it makes at most $\min(N_x[a,b],N_y[a,b])$ full turns.
\end{corollary}
\begin{proof}
Let $k:=N_y[a,b]$. The $k$ crossings divide $[a,b]$ into $k+1$ closed subintervals, each contained in the closure of a horizontal excursion. On such a subinterval $I$ is strictly monotone in the interior, hence, by continuity, on the whole closed subinterval. Hence $I$ vanishes at most once there, and, as in the proof of Lemma~\ref{lem:excursao}, $|\Delta\sigma|<2\pi$ on it. Summing, $|\Delta\sigma|<2\pi(k+1)$. The mirror argument, with $J$, gives $|\Delta\sigma|<2\pi(N_x[a,b]+1)$. Finally, $m$ full turns require $|\Delta\sigma|\ge2\pi m$. Unlike in the proof of Theorem~\ref{prop:contagem}, the extreme subintervals are not terminal excursions of $\gamma$, so the sharper bound of Lemma~\ref{lem:excterminal} is not available for them.
\end{proof}

\begin{figure}[ht]
\centering
\includegraphics[width=\textwidth]{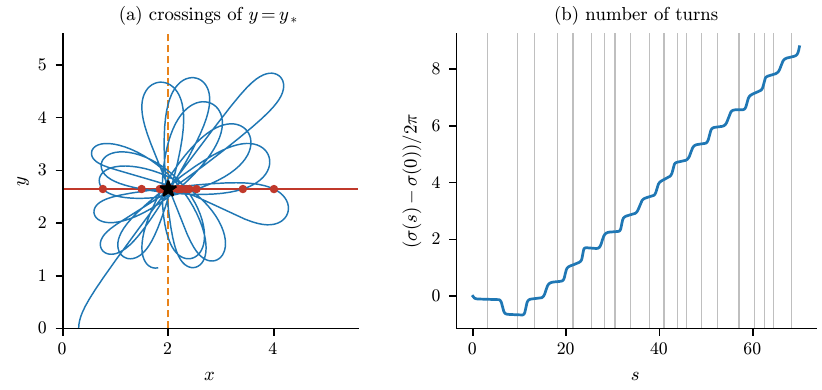}
\caption{A trajectory with high winding, for $p=4$, $q=7$, $c=-1$, leaving the $x$-axis at $x_0=0.3$, with $s\in[0,70]$. (a) The trajectory, with its $20$ crossings of $y_*$ marked (red dots) and the two cylinders. (b) The number of turns $(\sigma(s)-\sigma(0))/2\pi$ as a function of arclength. The vertical lines mark the crossings. Each strip between crossings corresponds to less than one turn
(Lemma~\ref{lem:excursao}), which produces \eqref{eq:contagem}.}
\label{fig:winding}
\end{figure}

\section{The crossing theorem}\label{sec:cross}

We now prove Theorems~\ref{thmA} and~\ref{thmB}. The limit values of $I$ and the exclusion of the escape scenarios in Section~\ref{sec:fronteira} give the crossing theorem, and the analysis at infinity in Section~\ref{sec:polar} classifies the unbounded ends.

\subsection{Crossing of the cylinders}\label{sec:cruzamento}

We also record a consequence of Proposition~\ref{lem:Jfront}, in the language of dynamical
systems.

\begin{corollary}[Localization of the limits]\label{prop:limites}
Let $c<0$ and let $\gamma$ be confined to $\mathcal S^{+}$ (resp. $\mathcal S^{-}$), distinct
from the vertical cylinder, on the maximal interval $(s_-,s_+)$. Then,
\begin{enumerate}
\item[(i)] $I$ is strictly monotone and bounded, $|I|\le M_f$. Hence there exist
$I_\pm=\lim_{s\to s_\pm}I$, finite, with $I_-\ne I_+$;
\item[(ii)] $\gamma$ leaves every compact subset of $\mathcal S^{+}$ (resp. $\mathcal S^{-}$),
that is, its $\alpha$- and $\omega$-limit sets are contained in $\partial\mathcal S^{+}$ (resp. 
$\partial\mathcal S^{-}$).
\end{enumerate}
\end{corollary}
\begin{proof}
(i) It is Corollary~\ref{cor:estrita}, item (i), with the bound of Lemma~\ref{lem:bound}. (ii) Suppose that $\gamma$ meets some compact subset of the strip along a sequence $s_k\to s_+$. Passing to a subsequence and using the compactness of $\s^1$, we may assume that $\Gamma(s_k)\to z$, with $z$ lying over an interior point of the strip. By maximality, $s_+=+\infty$. Since $I(s)\to I_+$, continuous dependence implies that $I$ is constant on the local orbit through $z$. As $B(y)\ne0$ in the strip, \eqref{eq:dJ} gives $\cos\sigma=0$ along that orbit. Lemma~\ref{lem:Z}(i) then shows that it is the vertical cylinder. If its forward direction points away from $\{y=y_*\}$, then $I_+=M_f$ in $\mathcal S^+$ or $I_+=-M_f$ in $\mathcal S^-$, contrary to the strict monotonicity and boundedness of $I$ (item (i)). In the opposite direction the cylinder crosses $\{y=y_*\}$ in finite forward time, contradicting confinement by continuous dependence. The argument at $s_-$ follows by reversibility (Lemma~\ref{lem:simetrias}(i)).
\end{proof}

Corollary~\ref{prop:limites} is already the crossing theorem in disguise: a trajectory confined to a strip leaves every compact subset of it. Since the escapes toward $\partial\mathcal Q$ were excluded in Section~\ref{sec:fronteira}, it must cross the level $\{y=y_*\}$. This is the first half of Theorem~\ref{thmA}, the second half being its mirror, Corollary~\ref{cor:crossI}.

\begin{theorem}[Crossing of the horizontal cylinder]\label{thm:cross}
Let $c<0$. Every trajectory of \eqref{lambda-s}, distinct from the horizontal cylinder, intersects the level $\{y=y_*\}$, and does so transversely (Lemma~\ref{lem:transv}). Equivalently, no non-cylindrical trajectory is confined to $\mathcal S^{+}$ or to $\mathcal S^{-}$.
\end{theorem}
\begin{proof}
If $\gamma$ never touches $\{y=y_*\}$, it is confined, by continuity, to $\mathcal S^{+}$ or to $\mathcal S^{-}$. The vertical cylinder touches the level at the vertex (Rem.~\ref{rem:vertex}), so we may suppose $\gamma$ distinct from both cylinders. By Corollary~\ref{cor:estrita}(i), $I$ is strictly monotone on the whole strip, and Corollary~\ref{prop:limites} applies. In the analysis below, we use the items (F1) to (F4) of Proposition~\ref{lem:Jfront}.

\emph{Step 1. Confinement to $\mathcal S^{+}$ is impossible.} By Proposition~\ref{lem:Jfront}, in $\mathcal S^{+}$ only the cases (F1) and (F4) can occur. Proposition~\ref{lem:fuga} excludes (F4). Since $I$ decreases strictly in $\mathcal S^{+}$, we get $I_->I_+$, but this contradicts (F1), which says that $I_-=0=I_+$. 

\emph{Step 2. Confinement to $\mathcal S^{-}$ is also impossible.} By Proposition~\ref{lem:Jfront}, in $\mathcal S^{-}$ only the cases (F1) and (F3) can occur. Since $I$ is strictly increasing, $I_-<I_+$ and, thus, (F1) is impossible. We divide the analysis of (F3) into two cases:
\begin{itemize}
\item[(a)] the endpoint $\alpha$ lies on the $x$-axis, at $(x_0,0)$. Then $I_-=f(x_0)>0$ by (F3), hence
$I_+>f(x_0)>0$, but the available endpoints $\omega$ have $I_+=0$ by (F1) or $I_+=-f(x_1)<0$ (endpoint $\omega$ on the $x$-axis), which is a contradiction;
\item[(b)] only the endpoint $\omega$ lies on the $x$-axis. Then $I_+=-f(x_0)<0$ and $I_-=0$, contradicting
$I_-<I_+$.
\end{itemize}
Therefore, neither of the situations can actually occur, and this proves the theorem.
\end{proof}

Applying Theorem~\ref{thm:cross} to the mirrored system yields the second half of
Theorem~\ref{thmA}. Figure~\ref{fig:cross} illustrates the crossing theorem.

\begin{figure}[htbp]
\centering
\includegraphics[width=\textwidth]{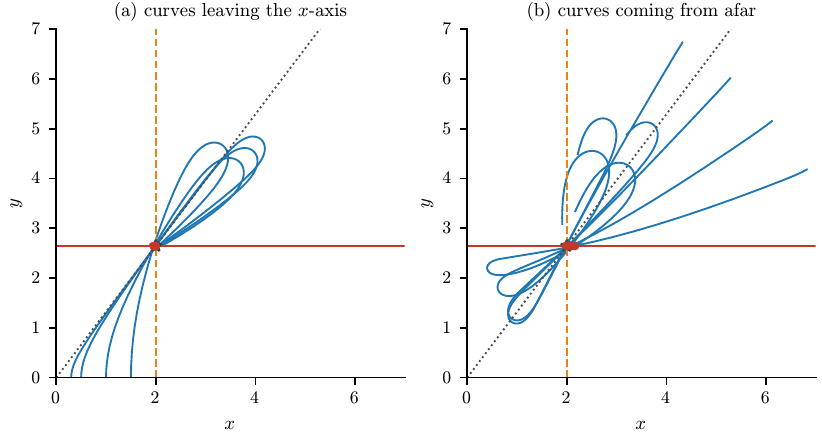}
\caption{The crossing theorem, for $p=4$, $q=7$, $c=-1$. (a) Curves leaving the $x$-axis at $x_0\in\{0.3,0.5,1,1.5\}$, up to their second crossing of $y_*$: all cross $y_*$ (red dots) near the vertex $V$. (b) Trajectories coming from afar align anti-radially (see \eqref{eq:wP}). Started at $\rho=8$ in the directions $\psi\in\{0.55,0.7,0.85,1\}$ with $\sigma=\psi+\pi$ and followed for $s\in[0,12.5]$, they cross $y_*$ (red dots) in the core. In red the horizontal cylinder, dashed orange the
vertical one, dotted the Clifford ray (Proposition~\ref{lem:origem}).}
\label{fig:cross}
\end{figure}

\begin{corollary}[Mirror: crossing of the vertical cylinder]\label{cor:crossI}
Let $c<0$. Every trajectory of \eqref{lambda-s} distinct from the vertical cylinder intersects the level $\{x=x_*\}$, transversely. Equivalently, no non-cylindrical trajectory is confined to $\mathcal V^{+}$ or to $\mathcal V^{-}$.
\end{corollary}
\begin{proof}
It is Theorem~\ref{thm:cross} under the mirror symmetry (Lemma~\ref{lem:simetrias}(ii)), with $J$ in place of $I$. The role of Proposition~\ref{lem:fuga} in the proof is taken by Corollary~\ref{cor:fugaH}.
\end{proof}

The two crossings may now be combined.

\begin{corollary}[Non-confinement and synthesis]\label{cor:sintese}
The only bi-rotational self-shrinker whose profile curve does not cross the level $y_*$ is the horizontal cylinder $\s^{q}(y_*)\times\R^{p+1}$ itself. In particular, every trajectory leaving the $x$-axis at $(x_0,0)$ crosses $\{y=y_*\}$, and the arcs from the $x$-axis contained in $\mathcal S^{-}$ exist only as terminal excursions, between the axis and the adjacent crossing.
\end{corollary}
\begin{proof}
The first assertion is a reformulation of Theorem~\ref{thm:cross}: a profile curve that does not cross $\{y=y_*\}$ is confined to $\mathcal S^{+}$ or to $\mathcal S^{-}$, which is only possible for the horizontal cylinder. For the second, let $\gamma$ leave the $x$-axis at $(x_0,0)$. It cannot be the horizontal cylinder, whose profile lies on $\{y=y_*\}$ and hence does not emanate from $\{y=0\}$. Thus Theorem~\ref{thm:cross} implies that $\gamma$ crosses $\{y=y_*\}$.
\end{proof}

\subsection{Polar coordinates and the ends at infinity}\label{sec:polar}

The remaining unbounded ends are those along which $\rho\to\infty$. To study them we write the system in polar coordinates and isolate the alignment mechanism in the far region. Lemmas~\ref{lem:polar} and \ref{lem:rate} hold for every $c$. From \eqref{eq:wP} on, $c<0$.

\begin{lemma}[System in polar coordinates]\label{lem:polar}
Write $x=\rho\cos\psi$, $y=\rho\sin\psi$, with $\rho>0$ and $\psi\in(0,\tfrac\pi2)$. Along
\eqref{eq:flow} we have
\begin{equation}\label{eq:polar}
\left\{\begin{aligned}
\rho'&=\cos(\psi-\sigma);\\
\psi'&=-\dfrac{\sin(\psi-\sigma)}{\rho};\\
\sigma'&=c\rho\sin(\psi-\sigma)+\dfrac{E(\psi,\sigma)}{\rho},\\
\end{aligned}\right.
\end{equation}
where
\[
E(\psi,\sigma):=\frac{q\cos\sigma}{\sin\psi}-\frac{p\sin\sigma}{\cos\psi}.
\]
\end{lemma}
\begin{proof}
From $\rho^2=x^2+y^2$ we get 
\[
\begin{aligned}
\rho\rho'&=xx'+yy'=x\cos\sigma+y\sin\sigma\\
&=\rho(\cos\psi\cos\sigma+\sin\psi\sin\sigma)\\
&=\rho\cos(\psi-\sigma),
\end{aligned}
\]
whence the first equation. From $\psi=\arctan(y/x)$ we get
\[
\psi'=\frac{xy'-yx'}{\rho^2}=\frac{\rho\cos\psi\sin\sigma-\rho\sin\psi\cos\sigma}{\rho^{2}}
=\frac{-\sin(\psi-\sigma)}{\rho},
\]
which is the second. For the third, we separate in \eqref{eq:flow} the part with $c$ and the
singular part:
\[
-\sin\sigma(cx)+\cos\sigma(cy)=c\rho\left(-\cos\psi\sin\sigma+\sin\psi\cos\sigma\right)
=c\rho\sin(\psi-\sigma)
\]
and
\[
-\frac{p\sin\sigma}{x}+\frac{q\cos\sigma}{y}
=\frac{1}{\rho}\left(\frac{q\cos\sigma}{\sin\psi}-\frac{p\sin\sigma}{\cos\psi}\right)
=\frac{E(\psi,\sigma)}{\rho}.
\]
\end{proof}

Away from the axes the function $E$ is bounded, so that near infinity the rotation of the tangent is driven by the term $c\rho\sin(\psi-\sigma)$.

\begin{lemma}[Asymptotic rotation rate]\label{lem:rate}
Fix $0<\delta<\tfrac\pi4$. For $\psi\in[\delta,\tfrac\pi2-\delta]$ one has
\begin{equation}\label{eq:rate}
|E(\psi,\sigma)| \le E_0:=\frac{p+q}{\sin\delta},
\end{equation}
and hence
\begin{equation}\label{eq:rate-2}
\sigma'=c\rho\sin(\psi-\sigma)+O(\rho^{-1})
\end{equation}
uniformly as $\rho\to\infty$. In particular, $\sigma'=O(\rho)$ uniformly for $\psi\in[\delta,\tfrac\pi2-\delta]$ as $\rho\to\infty$.
\end{lemma}
\begin{proof}
On $[\delta,\tfrac\pi2-\delta]$ one has $\sin\psi\ge\sin\delta$ and $\cos\psi\ge\sin\delta$,
whence $|E|\le(q+p)/\sin\delta$. The rest is \eqref{eq:polar}.
\end{proof}

The relative phase $w:=\psi-\sigma$ organizes the rotation. From \eqref{eq:polar}, and writing
$c=-|c|$, we have
\begin{equation}\label{eq:wP}
w'=\lambda(\rho)\sin w+P,
\end{equation}
where
\[
\lambda(\rho):=|c|\rho-\frac1\rho,\qquad P:=-\frac{E(\psi,\sigma)}{\rho}=\frac{p\sin\sigma}{x}-\frac{q\cos\sigma}{y},
\]
and the last expression follows from $x=\rho\cos\psi$ and $y=\rho\sin\psi$. In particular,
\begin{equation}\label{eq:Pbound}
|P|\le\frac px+\frac qy .
\end{equation}

Note that $\lambda>0$ for $\rho^{2}>1/|c|$, and $\lambda\ge\tfrac12|c|\rho$ for $\rho^{2}\ge2/|c|$. This means that the equilibrium $w=\pi$ of \eqref{eq:wP} is attracting and $w=0$ is repelling, both with a rate of order $|c|\rho$. Here, we say ``equilibrium'' in an approximate sense: for $\rho$ large, $|P|$ is small and $w'$ is near zero for $w$ near $0$ or $\pi.$

We first establish that both coordinates diverge, which makes the perturbation of the relative phase small.

\begin{proposition}[Both coordinates diverge and the ends are graphical]\label{lem:ambas}
Let $c<0$ and let $\gamma$ be a trajectory distinct from the cylinders, with an endpoint along which $\rho\to\infty$. Then, at that endpoint,
\[
x\to\infty,\qquad y\to\infty,
\]
and, near it, $\cos\sigma>0$ and $\sin\sigma>0$. In particular, both coordinates are eventually strictly monotone increasing and the end is graphical.
\end{proposition}

\begin{proof}
Since $\rho\to\infty$, at least one coordinate is unbounded. By the mirror symmetry
(Lemma~\ref{lem:simetrias}, item (ii)) we may assume $\limsup x=\infty$.

We use Theorem~\ref{thm:esfera}, below, whose proof is independent of the present asymptotic analysis.

\emph{Step 1: $\rho'\ge0$ near the endpoint.} By Theorem~\ref{thm:esfera}, every local minimum of $\rho$ satisfies $\rho\le r_0$ and every local maximum $\rho\ge r_0$. Since $\rho\to\infty$, there is $s_1$ with $\rho>r_0$ on $(s_1,s_+)$, so $\rho$ has no local minimum there and, if $\rho'(s_2)<0$ for some $s_2>s_1$, then $\rho$ could never increase again, contradicting $\rho\to\infty$. Hence
\[
\rho'=\cos w\ge0\qquad\text{on }(s_1,s_+),
\]
that is, $|w|\le\tfrac\pi2$ modulo $2\pi$.

\emph{Step 2: the critical points of $x$ and of $y$.} At an instant where $x'=\cos\sigma=0$ one has $\rho'=y\sin\sigma/\rho\ge0$ by Step~1, hence $\sin\sigma=1$ and $\sigma=\tfrac\pi2$. Moreover, at this point, $\sigma'=-A(x)$ and
\[
x''=(\cos\sigma)'=-\sin\sigma\,\sigma'=A(x),
\]
so a local maximum of $x$ forces $A(x)\le0$, that is, $x\ge x_*$, and a local minimum forces $x\le x_*$. Symmetrically, at a critical point of $y$ one has $\sigma=0$ and $y''=B(y)$, so local maxima of $y$ occur where $y\ge y_*$ and local minima where $y\le y_*$.

\emph{Step 3: near the endpoint, the tangent is almost radial in $S_\varepsilon$.} Since $\rho\to\infty$, the endpoint is reached at $s_+=\infty$. Fix $\varepsilon\in(0,1)$ \textup{(}it stays fixed until the end of Step~5\textup{)} and set
\[
C_\varepsilon:=\frac{p+q}{\varepsilon},\qquad
U_\varepsilon:=\{x>\varepsilon,\ y>\varepsilon\},\qquad
S_\varepsilon:=\{x>3\varepsilon,\ y>3\varepsilon\}.
\]
On $U_\varepsilon$ one has $|P|\le p/x+q/y<C_\varepsilon$. Choose $R_\varepsilon$ so large that, for $\rho\ge R_\varepsilon$,
\[
\lambda(\rho)\ge\tfrac12|c|\rho>0
\qquad\text{and}\qquad
\frac{2C_\varepsilon}{\lambda(\rho)}\le\frac{\sqrt3}{2};
\]
this is possible because $\lambda(\rho)=|c|\rho-1/\rho$ is increasing and unbounded. We claim that
\begin{equation}\label{eq:trap-setor}
|\sin w(s)|\ \le\ \frac{2C_\varepsilon}{\lambda(\rho(s))}\qquad\text{whenever } s>s_1,\ \gamma(s)\in S_\varepsilon \text{ and } \rho(s)\ge R_\varepsilon .
\end{equation}

Suppose that $\sin w(\bar s)>2C_\varepsilon/\lambda(\rho(\bar s))$ at such an instant $\bar s$. The case $\sin w(\bar s)<-2C_\varepsilon/\lambda(\rho(\bar s))$ is treated in the same way, with $w\in[-\tfrac\pi2,0)$ and all monotonicities reversed. By Step~1, $|w|\le\tfrac\pi2$ modulo $2\pi$, so $w(\bar s)\in(0,\tfrac\pi2]$. Let $s^*$ be the supremum of the instants $s\ge\bar s$ such that, on $[\bar s,s]$,
\[
\gamma\in U_\varepsilon,\qquad w\in(0,\tfrac\pi2],\qquad \sin w>\frac{2C_\varepsilon}{\lambda(\rho)} .
\]
On $[\bar s,s^*)$ we have
\[
w'=\lambda(\rho)\sin w+P\ \ge\ \lambda(\rho)\sin w-C_\varepsilon\ \ge\ \tfrac12\lambda(\rho)\sin w\ \ge\ C_\varepsilon .
\]
Three consequences follow. First, $w$ increases, hence so does $\sin w$, since $w\in(0,\tfrac\pi2]$. Moreover, $2C_\varepsilon/\lambda(\rho)$ does not increase, since $\lambda$ is increasing and $\rho'\ge0$ by Step~1. Thus
\[
\sin w(s)\ \ge\ \sin w(\bar s)\ >\ \frac{2C_\varepsilon}{\lambda(\rho(\bar s))}\ \ge\ \frac{2C_\varepsilon}{\lambda(\rho(s))}\qquad\text{on }[\bar s,s^*],
\]
with a margin that does not shrink. Second, $w$ cannot remain in $(0,\tfrac\pi2]$ for an arclength larger than $\pi/(2C_\varepsilon)=\pi\varepsilon/(2(p+q))<\varepsilon$, so $s^*-\bar s<\varepsilon$. Third, since $|x'|,|y'|\le1$ and $\gamma(\bar s)\in S_\varepsilon$, we have $x,y>2\varepsilon$ on $[\bar s,s^*]$, so $\gamma$ stays in $U_\varepsilon$ with room to spare. Hence the only condition that can fail at $s^*$ is $w\in(0,\tfrac\pi2]$, which means $w(s^*)=\tfrac\pi2$. But then $\rho'(s^*)=\cos w(s^*)=0$ and $w'(s^*)\ge C_\varepsilon>0$, so $\rho'=\cos w<0$ immediately after $s^*$, contradicting Step~1. This proves \eqref{eq:trap-setor}.

\emph{Step 4: the variation of $\psi$ inside $S_\varepsilon$ is small.} For $R\ge R_\varepsilon$, let $s_R>s_1$ be an instant at which $\rho=R$. Since $\rho$ is nondecreasing after $s_1$, we have $\rho\ge R$ on $[s_R,\infty)$. At the instants after $s_R$ at which $\gamma\in S_\varepsilon$, \eqref{eq:trap-setor} and the choice of $R_\varepsilon$ give $|\sin w|\le\sqrt3/2$, hence $\cos w\ge\tfrac12$ by Step~1, that is, $\rho'\ge\tfrac12$ and $ds\le2\,d\rho$. Since $\psi'=-\sin w/\rho$ and $\lambda(\rho)\ge\tfrac12|c|\rho$, the variation of $\psi$ accumulated inside $S_\varepsilon$ after $s_R$ satisfies
\begin{equation}\label{eq:var-setor}
\int_{\{s\ge s_R:\ \gamma(s)\in S_\varepsilon\}}|\psi'|\,ds
\ \le\ \int_R^{\infty}\frac{2C_\varepsilon}{\lambda(\rho)\,\rho}\,2\,d\rho
\ \le\ \int_R^{\infty}\frac{8C_\varepsilon}{|c|\rho^{2}}\,d\rho=\frac{8C_\varepsilon}{|c|R}=:V_R .
\end{equation}

\emph{Step 5: $\psi$ converges.} At a point with $\rho\ge R$ outside $S_\varepsilon$ one has $x\le3\varepsilon$ or $y\le3\varepsilon$, that is, $\cos\psi\le3\varepsilon/R$ or $\sin\psi\le3\varepsilon/R$. Since $\psi\in(0,\tfrac\pi2)$ and $\arcsin t\le\pi t/3\le4t/3$ for $0\le t\le\tfrac12$, by convexity, for $R\ge6\varepsilon$ this places $\psi$ in one of the two intervals
\[
W_0:=[0,\,4\varepsilon/R],\qquad W_{\pi/2}:=[\tfrac\pi2-4\varepsilon/R,\,\tfrac\pi2].
\]
Going from $W_0$ to $W_{\pi/2}$, or back, after $s_R$ requires a variation of $\psi$ of at least $\tfrac\pi2-8\varepsilon/R$, all of it accumulated inside $S_\varepsilon$, and for $R$ large this exceeds $V_R$. Hence, after $s_R$, every point of $\gamma$ outside $S_\varepsilon$ has $\psi$ in the same interval, say $W$.

Now let $s_R\le s<s'$. If $\gamma$ stays in $S_\varepsilon$ on $[s,s']$, then, by \eqref{eq:var-setor},
\[
|\psi(s')-\psi(s)|\ \le\ \int_{s}^{s'}|\psi'|\,d\tau\ \le\ V_R .
\]
Otherwise, the set of instants of $[s,s']$ at which $\gamma\notin S_\varepsilon$ is closed and nonempty. Let $u_1\le u_2$ be its first and last elements. Then $\psi(u_1),\psi(u_2)\in W$, so that $|\psi(u_2)-\psi(u_1)|\le4\varepsilon/R$, the length of $W$, while $\gamma\in S_\varepsilon$ on $[s,u_1)$ and on $(u_2,s']$, so that, again by \eqref{eq:var-setor},
\[
|\psi(u_1)-\psi(s)|\le\int_{s}^{u_1}|\psi'|\,d\tau\le V_R,
\qquad
|\psi(s')-\psi(u_2)|\le\int_{u_2}^{s'}|\psi'|\,d\tau\le V_R .
\]
By the triangle inequality,
\[
\begin{aligned}
|\psi(s')-\psi(s)|
&\le |\psi(u_1)-\psi(s)|+|\psi(u_2)-\psi(u_1)|+|\psi(s')-\psi(u_2)|\\
&\le 2V_R+\frac{4\varepsilon}{R}
= \frac{16C_\varepsilon}{|c|R}+\frac{4\varepsilon}{R}=:\delta_R .
\end{aligned}
\]
In both cases $|\psi(s')-\psi(s)|\le\delta_R$ for all $s,s'\ge s_R$. Hence
\[
0\ \le\ \limsup_{s\to\infty}\psi(s)-\liminf_{s\to\infty}\psi(s)\ \le\ \delta_R
\]
for every sufficiently large $R$, and since $\delta_R\to0$ as $R\to\infty$, the upper and lower limits coincide. Thus $\psi$ converges to some $\psi_\infty\in[0,\tfrac\pi2]$.

\emph{Step 6: $x\to\infty$ and $y\to\infty$.} If $\psi_\infty\in(0,\tfrac\pi2)$, then $x=\rho\cos\psi\to\infty$ and $y=\rho\sin\psi\to\infty$, and we are done. Suppose $\psi_\infty=\tfrac\pi2$. The case $\psi_\infty=0$ is symmetric. Then $y=\rho\sin\psi\to\infty$, so $\liminf x>0$ by Proposition~\ref{prop:nao-eixo} and $x$ is not contained in a compact subinterval of $(0,\infty)$ near the endpoint: there $y>y_*$, so the trajectory is confined to $\mathcal S^{+}$ and $\limsup y=\infty$, and Proposition~\ref{lem:fuga} applies. Hence $\limsup x=\infty$.

It remains to exclude the possibility that $x$ oscillates, that is, that $\liminf x<\limsup x=\infty$. In that case $x$ has infinitely many local maxima $s_k$, which we may choose with $x(s_k)\to\infty$, and after each of them a local minimum $t_k$. By Step~2, $\sigma=\tfrac\pi2$ at all these instants, so that
\[
I(s_k)=f(x(s_k))\longrightarrow0,\qquad I(t_k)=f(x(t_k))\ \ge\ \min_{[m,x_*]}f>0,
\]
where $m>0$ is a lower bound for $x$ near the endpoint and $x(t_k)\le x_*$, again by Step~2. On the other hand $y\to\infty$ gives $y>y_*$ near the endpoint, so $B(y)<0$ and $I'=B(y)f(x)\cos^{2}\sigma\le0$ there: $I$ is nonincreasing. Since $t_k>s_k$, this gives $0<\min_{[m,x_*]}f\le I(t_k)\le I(s_k)\to0$, a contradiction. Hence $x\to\infty$.

Finally, near the endpoint, we have $y>y_*$, so $I$ is strictly decreasing and $|I|\le f(x)\to0$, so $I\to0$ and therefore $I>0$, that is, $\sin\sigma>0$. Moreover $x>x_*$ near the endpoint, since $x\to\infty$, so $x$ has no local minimum there by Step~2. As $x$ tends to $\infty$, it cannot be eventually decreasing, hence $x'=\cos\sigma>0$ near the endpoint as well.
\end{proof}

This is the technical core of the section: the tangent becomes radial, and the polar angle has finite total variation.

\begin{proposition}[Radial alignment]\label{prop:align}
Let $c<0$ and let $\gamma$ be a trajectory distinct from the cylinders. Let $x=\rho\cos\psi$ and $y=\rho\sin\psi$ be its polar coordinates and $w=\psi-\sigma$ its relative phase. Suppose $\gamma$ has an endpoint along which $\rho\to\infty$. Then, near that endpoint,
\begin{equation}\label{eq:sinw}
|\sin w|\ \le\ \frac{4}{|c|\rho}\left(\frac px+\frac qy\right),
\end{equation}
the tangent becomes radial ($w\to0$ at $\omega$, $w\to\pi$ at $\alpha$), and the polar angle $\psi$ has finite total variation. In particular, the polar angle $\psi$ of $\gamma$ converges to a limit $\psi_\infty\in[0,\tfrac\pi2]$.
\end{proposition}

\begin{proof}
Set
\[
\Pi:=\frac px+\frac qy.
\]
By Proposition~\ref{lem:ambas}, we have $x\to\infty$ and $y\to\infty$ at the endpoint, so
$\Pi\to0$. Fix once and for all an instant beyond which
\begin{equation}\label{eq:align-setup}
\rho^{2}\ge\frac{2}{|c|},
\quad
\lambda(\rho)=|c|\rho-\frac1\rho \ge \tfrac12|c|\rho > 0,
\quad\mbox{and}\quad \Pi \le \tfrac14\lambda.
\end{equation}
This is possible because $\Pi\to0$ while $\lambda\to\infty$. We take the
endpoint to be $\omega$. The case $\alpha$ follows by reversibility (Lemma~\ref{lem:simetrias}, item (i)), which replaces $w$ by $w-\pi$. By \eqref{eq:wP} and \eqref{eq:Pbound}, the relative phase satisfies
\begin{equation}\label{eq:align-wode}
w'=\lambda(\rho)\sin w+P,\qquad |P|\le\Pi.
\end{equation}

\smallskip
\emph{Step 1: the trajectory is confined to $|\sin w|\le2\Pi/\lambda$.}

By Proposition~\ref{lem:ambas}, $\sigma\in(0,\tfrac\pi2)$ near the endpoint. Since also $\psi\in(0,\tfrac\pi2)$, the lift of $w=\psi-\sigma$ can be chosen in $(-\tfrac\pi2,\tfrac\pi2)$ there.

Suppose that $\sin w(s_1)>2\Pi/\lambda$ at some instant $s_1$ beyond \eqref{eq:align-setup}. At such a point,
\eqref{eq:align-wode} gives
\[
w' \ge \lambda\sin w-\Pi \ge \lambda\sin w-\tfrac12\lambda\sin w
= \tfrac12\lambda\sin w > 0,
\]
the middle inequality being $\Pi\le\tfrac12\lambda\sin w$, which is exactly the assumption
$\sin w>2\Pi/\lambda$. We claim this persists. Write $b:=2\Pi/\lambda$, so that the assumption reads $\sin w>b$. The barrier is nonincreasing while $w\in(0,\tfrac\pi2)$: there $\rho'=\cos w>0$, hence $\lambda'=(|c|+\rho^{-2})\cos w>0$, while $\Pi'=-\left(p\cos\sigma/x^{2}+q\sin\sigma/y^{2}\right)\le0$ by Proposition~\ref{lem:ambas}, so that $b'=2(\Pi'\lambda-\Pi\lambda')/\lambda^{2}\le0$. And at a crossing $\sin w=b$ one has $w'\ge\lambda b-\Pi=\Pi>0$, whence $\tfrac{d}{ds}\sin w=\cos w w'>0\ge b'$. Thus, the barrier cannot be crossed downward. 

We now prove that $w$ reaches $\pi/2$ at a finite arclength. Indeed, on that stretch $w$ increases, so $\sin w\ge\sin w(s_1)$ and $\lambda$ admits the uniform lower bound
\[
\lambda(\rho) \ge \tfrac12|c|\rho \ge \tfrac12|c|\sqrt{2/|c|}=\sqrt{|c|/2}=:\lambda_0>0
\]
by \eqref{eq:align-setup}. Therefore,
\[
w'\ge \tfrac12\lambda\sin w \ge \tfrac12\lambda_0\sin w(s_1)=:\kappa>0,
\]
a constant, and integrating, $w(s)\ge w(s_1)+\kappa(s-s_1)$. Hence $w$ reaches $\tfrac\pi2$ after arclength at most $\tfrac\pi{2\kappa}$, which is impossible since $w<\tfrac\pi2$. The case $\sin w<-2\Pi/\lambda$ is symmetric, with $w$ decreasing to $-\tfrac\pi2$.

Therefore $|\sin w|\le2\Pi/\lambda$ near the endpoint, and since
$\lambda\ge\tfrac12|c|\rho$ by \eqref{eq:align-setup}, we obtain \eqref{eq:sinw}.

\smallskip
\emph{Step 2: the tangent becomes radial.} By \eqref{eq:sinw} and $\Pi\to0$ we get $\sin w\to0$, and since $|w|<\tfrac\pi2$, $w\to0$: the tangent becomes radial, and in particular
\begin{equation}\label{eq:align-cosw}
\rho'=\cos w \ge \tfrac12
\end{equation}
near the endpoint.

\smallskip
\emph{Step 3: the polar angle has finite total variation.}
By \eqref{eq:align-cosw} the function $\rho$ is strictly increasing, so it may be used as the
variable, with $ds=d\rho/\rho'\le2 d\rho$. From the polar system \eqref{eq:polar}, we have $\psi'=-\sin w/\rho$, and by using \eqref{eq:sinw} we obtain
\[
|\psi'|\le \frac{|\sin w|}{\rho} \le \frac{4\Pi}{|c|\rho^2}.
\]
Since $\Pi\to0$, in particular $\Pi\le\Pi_0:=\sup\Pi<\infty$, and integrating from some
$\rho_1$ on,
\[
\int|\psi'|ds \le \int_{\rho_1}^{\infty}\frac{4\Pi_0}{|c|\rho^2}\cdot 2d\rho = \frac{8\Pi_0}{|c|\rho_1} < \infty.
\]
Hence $\psi'$ is integrable near the endpoint, and $\psi(s)=\psi(s_0)+\int_{s_0}^{s}\psi'\,d\tau$ converges, as $s\to s_+$, to some $\psi_\infty$. Finally $\psi_\infty\in[0,\tfrac\pi2]$, because $\psi(s)\in(0,\tfrac\pi2)$ for every $s$, the trajectory remaining in the open quadrant.
\end{proof}

Proposition~\ref{prop:align} leaves open whether the limit direction is degenerate, that is, whether $\psi_\infty=0$ or $\pi/2$. Propositions~\ref{lem:fuga} and~\ref{prop:nao-eixo} do not settle this. They rule out an end with one coordinate bounded while the other diverges, whereas $\psi_\infty=0$ only means $y/x\to0$, which is compatible with $y\to\infty$, for instance with $y\sim\sqrt{x}$. Such an end would be asymptotic to no cylinder and would not be conical. The next theorem excludes it.

\begin{theorem}[Conical ends, direct form]\label{cor:conico-direto}
Let $c<0$ and let $\gamma$ be a trajectory distinct from the cylinders, with an endpoint along which $\rho\to\infty$. Then that end is asymptotically conical, i.e., the tangent is asymptotically radial and the limit $\psi_\infty$ lies in the open interval $(0,\tfrac\pi2)$.
\end{theorem}

\begin{proof}
By Proposition~\ref{prop:align} the limit $\psi_\infty\in[0,\tfrac\pi2]$ exists and the tangent is
radial. We need to show that $\psi_\infty\notin\{0,\tfrac\pi2\}$. By the mirror symmetry
(Lemma~\ref{lem:simetrias}, item (ii)) it suffices to exclude $\psi_\infty=0$.

Suppose $\psi\to0$, and work near the endpoint, where $\psi\le\tfrac\pi4.$ By Proposition~\ref{prop:align}, $\cos w\ge\tfrac12$ and
\[
|\psi'|=\frac{|\sin w|}{\rho}\le \frac{4}{|c|\rho^{2}}\left(\frac px+\frac qy\right).
\]
We now estimate the quantity $\psi\left(\tfrac px+\tfrac qy\right)$. Since $y=\rho\sin\psi$ and
$\sin\psi\ge\tfrac2\pi\psi$ on $[0,\tfrac\pi2]$, we get $\tfrac\psi y\le\tfrac{\pi}{2\rho}$. On the other hand, since $x=\rho\cos\psi$ with $\cos\psi\ge\tfrac{\sqrt2}2$ for $\psi\le\tfrac\pi4$, we get $\tfrac\psi x\le\tfrac{\pi\sqrt{2}}{4\rho}$. Hence
\[
\psi\left(\frac px+\frac qy\right) \le \frac{\pi\sqrt{2}p}{4\rho}+\frac{\pi q}{2\rho} \le \frac{2(p+q)}{\rho},
\]
and therefore
\[
\left|\frac{d}{ds}\psi^{2}\right|=2\psi|\psi'| \le \frac{8}{|c|\rho^{2}}\cdot\frac{2(p+q)}{\rho}=\frac{16(p+q)}{|c|\rho^3}.
\]
Since $\rho'=\cos w\ge\tfrac12$, we may integrate in $\rho$ with $ds\le2 d\rho$. Thus, for every $\rho\ge\rho_1$,
\[
\psi^2(\rho) \ge \psi^2(\rho_1)-\int_{\rho_1}^{\infty}\frac{32(p+q)}{|c|t^3}dt = \psi^2(\rho_1)-\frac{16(p+q)}{|c|\rho_1^2}.
\]
Finally $\psi(\rho_1)\ge\sin\psi(\rho_1)=y(\rho_1)/\rho_1$, so
\[
\psi^2(\rho)\ge\frac{1}{\rho_1^2}\left(y(\rho_1)^2-\frac{16(p+q)}{|c|}\right),
\]
which is a positive constant as soon as $y(\rho_1)^{2}>16(p+q)/|c|$, which holds for
$\rho_1$ large, because $y\to\infty$ by Proposition~\ref{lem:ambas}. Hence $\psi$ is bounded below by a positive constant, contradicting $\psi\to0$.
\end{proof}

Figure~\ref{fig:conico} shows two such ends. We can combine the results above into the following corollary.

\begin{figure}[htbp]
\centering
\includegraphics[width=\textwidth]{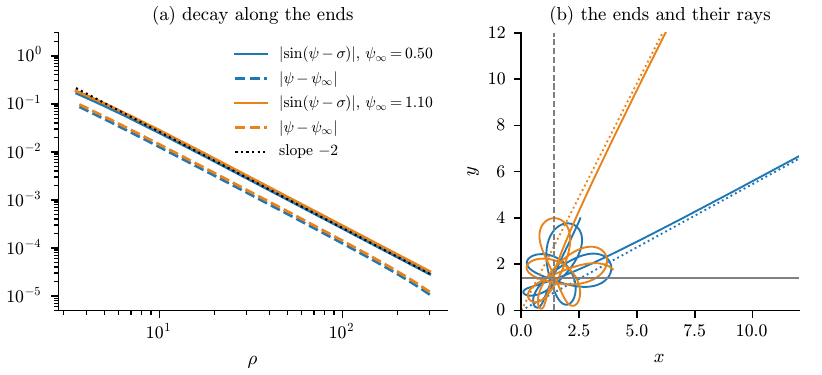}
\caption{Asymptotically conical ends (Theorem~\ref{cor:conico-direto} and Proposition~\ref{prop:decaimento}), for $p=q=1$, $c=-\tfrac12$. The ends are computed backward from $\rho=600$ with $\sigma=\psi$, for $\psi=0.5$ and $\psi=1.1$. Since the radial direction is repelling for the forward flow, the backward flow selects the conical ends. (a) On $\rho\in[30,200]$ the fitted slopes are $-2.00$ for $|\sin(\psi-\sigma)|$ and $-2.06$ for $|\psi-\psi_\infty|$, where $\psi_\infty$ is approximated by the value at $\rho=600$. The dotted line has slope $-2$. (b) Continued inward, both ends enter the core around the vertex and wind there.}
\label{fig:conico}
\end{figure}

\begin{corollary}[Conical ends are nondegenerate bigraphs]\label{cor:grafico}
Let $c<0$ and let $\gamma$ be a trajectory distinct from the cylinders, with an endpoint along which $\rho\to\infty$ and limit direction $\psi_\infty\in(0,\tfrac\pi2)$. Then, near that endpoint and parametrizing toward it, $x$ and $y$ are both strictly increasing, so that the end is simultaneously a graph $y=y(x)$ over the $x$-axis and a graph $x=x(y)$ over the $y$-axis. Moreover,
\[
\frac{dy}{dx}=\tan\sigma\longrightarrow\tan\psi_\infty\in(0,\infty),
\quad\mbox{and}\quad
\frac{dx}{dy}=\cot\sigma\longrightarrow\cot\psi_\infty\in(0,\infty),
\]
so that neither graph degenerates, i.e., the slopes stay bounded away from $0$ and from $\infty$.
\end{corollary}
\begin{proof}
By Proposition~\ref{lem:ambas}, $\cos\sigma>0$ and $\sin\sigma>0$ near the endpoint, so $x$ and $y$ are strictly increasing and the end is a graph over each axis. By Proposition~\ref{prop:align}, $w=\psi-\sigma\to0$, hence $\sigma\to\psi_\infty$, which gives the limits of $\tan\sigma$ and $\cot\sigma$.
\end{proof}

The alignment also controls the second fundamental form along a conical end.

\begin{proposition}[Decay at the ends and weighted integrability]\label{prop:decaimento} Let $c<0$ and let $\gamma$ be a trajectory with an asymptotically conical end. Then, along
this end,
\begin{equation}\label{eq:A2decay}
|\mathcal A|^{2}=O(\rho^{-2})\quad\mbox{as}\quad \rho\to\infty.
\end{equation}
Moreover,
\begin{enumerate}
\item[(i)] $|\mathcal A|^{2}\ge C\rho^{-2}$ along the end, for some $C>0$, and the unweighted total $L^2$ curvature of the portion of $\Sigma$ corresponding to the end is infinite, and a fortiori $\int_\Sigma|\mathcal A|^{2} dV=\infty$;
\item[(ii)] the total $L^2$ curvature of that portion, weighted by the Gaussian weight, is finite. If every end of $\gamma$ is asymptotically conical, then
\begin{equation}\label{eq:A2finite}
\int_\Sigma|\mathcal A|^{2} e^{c|X|^{2}/2} dV<\infty .
\end{equation}
\end{enumerate}
\end{proposition}
\begin{proof}
Since $\psi\to\psi_\infty\in(0,\tfrac\pi2)$, there is $\delta\in(0,\tfrac\pi4)$ with $\psi\in[\delta,\tfrac\pi2-\delta]$ along the end, so that $x=\rho\cos\psi\ge\rho\sin\delta$ and $y\ge\rho\sin\delta$. By Proposition~\ref{prop:align},
\[
|\sin(\psi-\sigma)|\le\frac{4}{|c|\rho}\left(\frac px+\frac qy\right)\le\frac{4(p+q)}{|c|\sin\delta}\,\rho^{-2},
\]
and by \eqref{eq:polar} and \eqref{eq:rate},
\[
|\sigma'|\le|c|\rho|\sin(\psi-\sigma)|+\frac{|E|}{\rho}\le\left(\frac{4(p+q)}{\sin\delta}+\frac{p+q}{\sin\delta}\right)\rho^{-1}.
\]
Moreover, $\sin^2\sigma/x^2$ and
$\cos^2\sigma/y^2$ are both $O(\rho^{-2})$. By Eq. \eqref{eq:A2}, $|\mathcal A|^{2}=O(\rho^{-2})$,
which is \eqref{eq:A2decay}. The volume element of the hypersurface is
$dV=\omega_p\omega_q x^{p}y^{q} ds$, with $\omega_p=\mathrm{vol}(\s^p)$. Along each end,
$x^py^q\asymp\rho^{p+q}$ and $ds\asymp d\rho$ (since $\rho'=\cos(\psi-\sigma)\to1$). For the lower bound, $\sigma\to\psi_\infty\in(0,\tfrac\pi2)$ gives $\sin^{2}\sigma\ge\frac12\sin^{2}\psi_\infty$ along the end, and $x\le\rho$, so that, by \eqref{eq:A2},
\[
|\mathcal A|^{2}\ \ge\ p\frac{\sin^{2}\sigma}{x^{2}}\ \ge\ \frac{p\sin^{2}\psi_\infty}{2}\rho^{-2}=:C\rho^{-2} .
\]
Hence, denoting by $\Sigma_{\rm end}$ the portion of $\Sigma$ over the end,
\[
\int_{\Sigma_{\rm end}}|\mathcal A|^{2} dV\ \ge\ C\omega_p\omega_q\int^{\infty}\rho^{-2}\rho^{p+q} d\rho
=C\omega_p\omega_q\int^{\infty}\rho^{p+q-2} d\rho=+\infty
\]
since $p+q\ge2$, and the integrand is nonnegative, so the integral over $\Sigma$ is infinite as well. This proves (i). With the weight $e^{c|X|^2/2}=e^{-|c|\rho^2/2}$, the upper bound $|\mathcal A|^{2}=O(\rho^{-2})$ gives
\[
\int_{\Sigma_{\rm end}}|\mathcal A|^{2}e^{c|X|^{2}/2} dV\ \lesssim\ \int^{\infty}\rho^{p+q-2} e^{-|c|\rho^2/2} d\rho<\infty .
\]
If every end is asymptotically conical, the complement of the ends in $\Sigma$ is compact, on which the integrand is bounded, and \eqref{eq:A2finite} follows. This proves (ii).
\end{proof}

The limit direction in Theorem~\ref{cor:conico-direto} is only constrained to lie in $(0,\tfrac\pi2)$. The next result shows that each direction is realized by exactly one end. For $\psi\in(0,\tfrac\pi2)$ we write
\[
C_\psi:=\{(X,Y)\in\R^{p+1}\times\R^{q+1}:\ |Y|=\tan\psi\,|X|\}
\]
for the cone over $\s^p(\cos\psi)\times\s^q(\sin\psi)$, so that $C_{\psi_0}$ is the Clifford cone.

\begin{proposition}[Classification of the escaping ends]\label{prop:fins}
Let $c<0$.
\begin{enumerate}
\item[(i)] If two maximal trajectories have endpoints along which $\rho\to\infty$ with the same limit direction $\psi_\infty$, then they coincide, up to the reversal of Lemma~\ref{lem:simetrias}, item (i).
\item[(ii)] The Clifford ray is the only maximal trajectory, up to reversal, with an endpoint along which $\rho\to\infty$ with limit direction $\psi_0$.
\item[(iii)] For every $\psi_\infty\in(0,\tfrac\pi2)\setminus\{\psi_0\}$ there is a maximal trajectory with an endpoint along which $\rho\to\infty$ with limit direction $\psi_\infty$. Near that endpoint, the corresponding part of $\Sigma$ is, up to a homothety, the self-shrinking end asymptotic to $C_{\psi_\infty}$ constructed by Chang and Spruck \cite[Theorem~1.1]{ChangSpruck}.
\end{enumerate}
In particular, for every $\psi_\infty\in(0,\tfrac\pi2)$ there is, up to reversal, exactly one maximal trajectory with an escaping end of limit direction $\psi_\infty$.
\end{proposition}

\begin{proof}
(i) By Theorem~\ref{cor:conico-direto}, either both trajectories are the same cylinder, or $\psi_\infty\in(0,\tfrac\pi2)$ and neither is a cylinder. In the second case, applying the reversal if necessary, we may assume that $\rho\to\infty$ at the endpoint $\omega$ of both trajectories. Set $\mu:=\tan\psi_\infty$. By Corollary~\ref{cor:grafico}, both ends are graphs $y=u_1(x)$ and $y=u_2(x)$ over a common half-line $[x_1,\infty)$, with $u_i'\to\mu$, and hence $u_i/x\to\mu$. Along such a graph, \eqref{lambda-s} reads
\begin{equation}\label{eq:graph-end}
u''=(1+u'^2)\,\mathcal G(x,u,u'),\qquad \mathcal G(x,u,u'):=|c|\,(xu'-u)-\frac{p\,u'}{x}+\frac{q}{u}.
\end{equation}
Enlarging $x_1$, we may assume that $u_i'\le2\mu$ and $u_i\ge\tfrac\mu2x$ on $[x_1,\infty)$.

We use the following elementary fact. Let $a$ be continuous with $a(x)\ge|c|x$ on $[x_1,\infty)$, let $\Upsilon$ be bounded, and let $W$ be a bounded $C^1$ solution of $W'=aW+\Upsilon$. Then
\[
W(x)=-\int_x^\infty e^{K(x)-K(\tau)}\Upsilon(\tau)\,d\tau,\qquad K(x):=\int_{x_1}^xa,
\]
and therefore
\begin{equation}\label{eq:fato-W}
|W(x)|\le\frac{1}{|c|\,x}\,\sup_{\tau\ge x}|\Upsilon(\tau)|.
\end{equation}
Indeed, by variation of constants,
\[
e^{-K(x)}W(x)=W(x_1)+\int_{x_1}^xe^{-K}\Upsilon.
\]
The left-hand side tends to $0$ because $W$ is bounded and $K(x)\to\infty$, and the integral converges. The estimate follows from $K(\tau)-K(x)\ge|c|x(\tau-x)$ for $\tau\ge x$.

Set $v_i:=xu_i'-u_i$. Multiplying \eqref{eq:graph-end} by $x$ gives
\[
v_i'=|c|\,x\,(1+u_i'^2)\,v_i+(1+u_i'^2)\left(\frac{qx}{u_i}-p\,u_i'\right),
\]
whose last term is bounded. Since $v_i/x=u_i'-u_i/x\to0$, the same argument as in the fact, with $|v_i|\le Cx$ in place of boundedness, shows that $v_i(x)e^{-K_i(x)}\to0$, and we obtain $v_i=O(x^{-1})$. Hence $\mathcal G(x,u_2,u_2')=O(x^{-1})$.

Now let $D:=u_1-u_2$ and $z:=xD'-D=v_1-v_2=O(x^{-1})$. Subtracting the two equations \eqref{eq:graph-end} and multiplying by $x$, a direct computation gives
\[
z'=|c|\,x\,(1+u_1'^2)\,z+\chi,
\]
where
\[
\chi:=-(1+u_1'^2)\left(p\,D'+\frac{q\,x\,D}{u_1u_2}\right)+x\,(u_1'+u_2')\,D'\,\mathcal G(x,u_2,u_2').
\]
By the bounds above,
\[
|\chi|\le C_0\left(|D'|+\frac{|D|}{x}\right)\quad\mbox{on }[x_1,\infty).
\]
Let $Z(x):=\sup_{\tau\ge x}|z(\tau)|$, which is finite and nonincreasing. Since $(D/x)'=z/x^2$ and $D/x\to0$, we have
\[
\frac{|D(x)|}{x}\le\int_x^\infty\frac{|z(\tau)|}{\tau^{2}}\,d\tau\le\frac{Z(x)}{x},
\]
so that
\[
|D(x)|\le Z(x)\qquad\mbox{and}\qquad |D'(x)|=\frac{|z+D|}{x}\le\frac{2Z(x)}{x}.
\]
Hence, for $\tau\ge x$,
\[
|\chi(\tau)|\le\frac{3C_0\,Z(\tau)}{\tau}\le\frac{3C_0\,Z(x)}{x},
\]
and \eqref{eq:fato-W} gives
\[
|z(x)|\le\frac{3C_0}{|c|\,x^2}\,Z(x).
\]
Since the right-hand side is nonincreasing in $x$, taking the supremum yields
\[
Z(x)\le\frac{3C_0}{|c|\,x^2}\,Z(x)\qquad\mbox{for every }x\ge x_1.
\]
Thus $Z(x)=0$ for $x$ large, so $D/x$ is constant there, hence zero, and $u_1=u_2$ near infinity. By uniqueness for \eqref{eq:flow}, the two maximal trajectories coincide.

(ii) The Clifford ray $y=\tan\psi_0\,x$ solves \eqref{eq:graph-end}, since then $xu'-u=0$ and
\[
\frac qu-\frac{p\,u'}{x}=\frac{q-p\tan^2\psi_0}{x\tan\psi_0}=0.
\]
The claim follows from (i).

(iii) Let $\mathcal P=\s^p(\cos\psi_\infty)\times\s^q(\sin\psi_\infty)\subset\s^{p+q+1}$. It belongs to the isoparametric family of the products of spheres, and it is not minimal because $\psi_\infty\ne\psi_0$. By \cite[Theorem~1.1]{ChangSpruck}, after a homothety normalizing $c$, there is a self-shrinking end smoothly asymptotic to $C_{\psi_\infty}$, given as a radial graph
\[
\left\{e^{\varphi(d(z))}z\right\}
\]
over an annular region of the sphere adjacent to $\mathcal P$, where $d$ is the distance to $\mathcal P$ in $\s^{p+q+1}$. Since $d$ depends only on the polar angle of $z$, this end is $\Gpq$-invariant, and its profile is the polar graph
\[
\rho=e^{\varphi(|\psi-\psi_\infty|)}
\]
over an interval of angles in $(0,\tfrac\pi2)$ adjacent to $\psi_\infty$, along which $\rho\to\infty$ as $\psi\to\psi_\infty$. This profile is an arc of a trajectory, and its maximal extension has an endpoint along which $\rho\to\infty$ with limit direction $\psi_\infty$. The last assertion of (iii) and the final statement follow from (i) and (ii).
\end{proof}

\begin{remark}[On Proposition~\ref{prop:fins}]\label{rem:angulos}
Item (i) is a bi-rotational counterpart of the uniqueness theorem of Wang \cite[Theorem~1.3]{Wang} for smooth, properly embedded self-shrinking ends asymptotic to a regular cone. Within the symmetry class it follows from the profile equation alone, using only the first-order asymptotics of Corollary~\ref{cor:grafico}. In particular, the bi-rotational trumpets constructed by Park \cite[Proposition~3.1]{Park-fol} coincide near infinity with the ends of \cite{ChangSpruck}. The proposition concerns the ends themselves, and the global behavior of the maximal trajectories carrying them is a separate question. Combined with Corollary~\ref{thmD}, it shows that a properly immersed complete noncompact bi-rotational self-shrinker is determined, up to reversal, by the limit direction of any one of its ends. Finally, the trajectories with an escaping end form part of the exceptional set of Theorem~\ref{thmH}(a) and the limit direction parametrizes them surjectively, a trajectory with two escaping ends of distinct directions corresponding to both.
\end{remark}

Together with the description of the endpoints, Proposition~\ref{prop:fins} shows that each self-shrinking end constructed by Chang and Spruck extends to a unique complete bi-rotational self-shrinker.

\begin{corollary}[Completion of the Chang--Spruck ends]\label{cor:completion}
Let $c<0$ and $\psi_\infty\in(0,\tfrac\pi2)\setminus\{\psi_0\}$. Up to a homothety, the self-shrinking end asymptotic to $C_{\psi_\infty}$ constructed by Chang and Spruck \cite[Theorem~1.1]{ChangSpruck} is contained in a unique complete bi-rotational self-shrinker $\Sigma_{\psi_\infty}$ without boundary, namely the one generated by the maximal trajectory of Proposition~\ref{prop:fins}\textup{(iii)}. Exactly one of the following holds:
\begin{enumerate}
\item[(i)] the other endpoint of the profile curve lies on an axis, and $\Sigma_{\psi_\infty}$ is proper and diffeomorphic to $\R^{q+1}\times\s^{p}$ or to $\R^{p+1}\times\s^{q}$;
\item[(ii)] the other end is asymptotically conical, with limit direction different from $\psi_\infty$, and $\Sigma_{\psi_\infty}$ is proper and diffeomorphic to $\R\times\s^{p}\times\s^{q}$;
\item[(iii)] the other end is non-escaping, and $\Sigma_{\psi_\infty}$ is not proper, has infinite entropy, and meets both cylinders in infinitely many connected components.
\end{enumerate}
\end{corollary}
\begin{proof}
Uniqueness of the trajectory is Proposition~\ref{prop:fins}\textup{(i)}. It is not the Clifford ray, by Proposition~\ref{prop:fins}\textup{(ii)}, so $\Sigma_{\psi_\infty}$ is complete by Theorem~\ref{thm:completude}, and its profile curve does not reach the origin, by Proposition~\ref{lem:origem}. The profile curve is not closed, since it has an escaping end. By Theorem~\ref{thmB}, the other extremity is an endpoint on an axis, an asymptotically conical end, or a non-escaping end. In the first two cases $\Sigma_{\psi_\infty}$ is proper by Corollary~\ref{cor:dicotomia}, and the topology is given by Theorem~\ref{thm:completas}. In the third case the conclusions are Corollaries~\ref{cor:dicotomia} and~\ref{cor:entropia}.

It remains to see that in case (ii) the two limit directions differ. If both ends had direction $\psi_\infty$, then, by Corollary~\ref{cor:grafico} and the argument in the proof of Proposition~\ref{prop:fins}\textup{(i)}, both ends would be the same graph $y=u(x)$ near infinity, traversed in opposite directions. Hence $\Gamma(s_1+t)=R(\Gamma(s_2-t))$ for some $s_1,s_2$ and all admissible $t$, where $R(x,y,\sigma)=(x,y,\sigma+\pi)$, and, as in Remark~\ref{rem:umlauf}\textup{(ii)}, the midpoint $m$ of $s_1$ and $s_2$ would satisfy $\Gamma(m)=R(\Gamma(m))$, which is absurd.
\end{proof}


\begin{remark}[Numerical examples]\label{rem:numerical-escaping}
Numerical evidence suggests that alternatives (i) and (ii) of Corollary~\ref{cor:completion} occur. For $p=q=1$ and $c=-\tfrac12$, the trajectory with $\psi_\infty\approx0.505$ appears to leave the $x$-axis at $x_0\approx1.334$, and those with $\psi_\infty\approx0.684$ and $\psi_\infty\approx0.352$ appear to cross the diagonal orthogonally. Since the reflection of Lemma~\ref{lem:simetrias}(ii) composed with the reversal then preserves the trajectory, the other end is conical, with limit direction $\tfrac\pi2-\psi_\infty$. These examples are analogous to the immersed rotational self-shrinkers of Drugan and Kleene \cite{DrKl}, and the same search produced many further candidates. Figure~\ref{fig:escaping} shows the three trajectories.
\end{remark}

\begin{figure}[htbp]
\centering
\includegraphics[width=\textwidth]{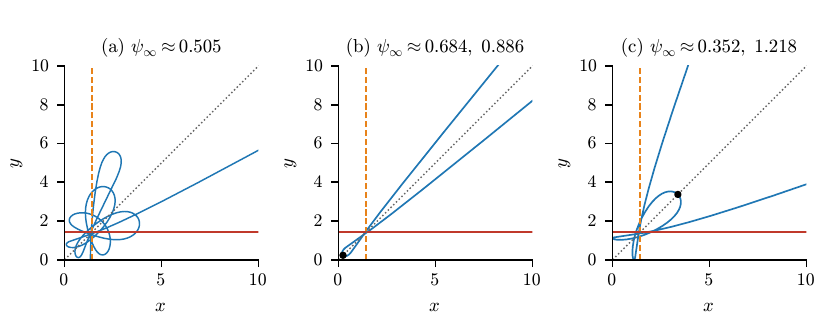}
\caption{Numerical examples for Corollary~\ref{cor:completion}, for $p=q=1$ and $c=-\tfrac12$ (Remark~\ref{rem:numerical-escaping}). (a) A trajectory leaving the $x$-axis at $x_0\approx1.334$, with an asymptotically conical end of limit direction $\psi_\infty\approx0.505$. (b) and (c) Trajectories with two asymptotically conical ends, symmetric with respect to the diagonal, which they cross orthogonally at the black dot. The ends are computed backward from $\rho=40$. The separatrices are the solid and the dashed lines, and the Clifford ray is dotted.}
\label{fig:escaping}
\end{figure}

\section{Complete hypersurfaces}\label{sec:completas}

This section proves Theorem~\ref{thmG}, completes the proof of Theorem~\ref{thmB}, and proves Corollary~\ref{thmD} and Theorem~\ref{thmF}. Before describing the topological types, we establish that the hypersurfaces in question are always complete, with the single exception of the Clifford minimal cone. Proposition~\ref{lem:suave}, Proposition~\ref{lem:origem} and Theorem~\ref{thm:completude} hold for every $c\ne0$, and Lemma~\ref{lem:terceira} for every $c\in\R$, since their proofs do not use the sign of $c$. The remaining results of this section assume $c<0$.

\begin{proposition}[Smooth closure at an axis]\label{lem:suave}
Let $c\ne0$ and let $\gamma$ be a trajectory of \eqref{lambda-s} whose endpoint reaches the
$x$-axis at $(x_0,0)$, $x_0>0$ (orthogonal endpoint, Proposition~\ref{lem:cusp}). Then, near the
endpoint,
\begin{equation}\label{eq:parabola}
x-x_0=\frac{A(x_0)}{2(q+1)} y^{2}+o(y^{2}).
\end{equation}
In particular, writing the profile curve for $y>0$ as $x=\tilde x(y)$, one has $\tilde x'(y)\to0$ and the even extension $\tilde x(-y)=\tilde x(y)$ is initially of class $C^{1,1}$. This extension is in fact $C^\infty$, and consequently the immersion \eqref{eq:param} closes off smoothly over the singular orbit $\s^{p}(x_0)\times\{0\}$, which becomes a regular interior subset of $\Sigma$. The mirror statement holds on the $y$-axis, with $p$ in place of $q$.
\end{proposition}

\begin{proof}
By Proposition~\ref{lem:cusp}, near the endpoint one has $\sin\sigma\to-1$. In particular, $y'=\sin\sigma<0$ and $y$ is strictly decreasing to $0$, so that $y$ serves as the independent variable.

\emph{Step 1: the asymptotics of $J$.} By Lemma~\ref{prop:dIdJ}, $J'=A(x)g(y)\sin^{2}\sigma$. Dividing by $y'=\sin\sigma$, we obtain
\begin{equation}\label{eq:dIdy}
\frac{dJ}{dy}=A(x)g(y)\sin\sigma.
\end{equation}
Since $g(y)=y^{q}e^{cy^{2}/2}\to0$ and $|\cos\sigma|\le1$, one has $J\to0$ at the endpoint.
Integrating \eqref{eq:dIdy} from $0$ to $y$, we obtain
\[
J(y)=\int_0^yA(x(t))g(t)\sin\sigma(t)dt.
\]
Now, as $x(t)\to x_0$ and $A$ is continuous at $x_0>0$, we have $A(x(t))\to A(x_0)$ and
$\sin\sigma(t)\to-1$. Writing $A(x(t))\sin\sigma(t)=-A(x_0)+o(1)$ and $g(t)=t^q(1+O(t^2))$, and using that
\[
\int_0^{y}(L+o(1))t^{q}dt=(L+o(1))\tfrac{y^{q+1}}{q+1}
\] 
for any constant $L$, we obtain
\begin{equation}\label{eq:Iasym}
J(y)=-\frac{A(x_0)}{q+1}y^{q+1}+o(y^{q+1}),\quad\mbox{as}\quad y\to0^{+}.
\end{equation}

\emph{Step 2: the asymptotics of $\cos\sigma$ and the parabola.} From $J=g(y)\cos\sigma$ and
$g(y)=y^{q}(1+O(y^{2}))$ we get, by \eqref{eq:Iasym},
\[
\cos\sigma=\frac{J(y)}{g(y)}=-\frac{A(x_0)}{q+1}y+o(y).
\]
Hence, since $\sin\sigma=-1+o(1)$,
\[
\frac{dx}{dy}=\frac{\cos\sigma}{\sin\sigma}=\frac{-\frac{A(x_0)}{q+1}y+o(y)}{-1+o(1)} =\frac{A(x_0)}{q+1}y+o(y),
\]
and integrating from $0$ to $y$ (with $x(0)=x_0$) one obtains exactly \eqref{eq:parabola}. In particular, $dx/dy\to0$, i.e., the curve meets the axis perpendicularly, with the quadratic expansion stated above.

\emph{Step 3: smoothness across the singular orbit.}
By Step~2 and \eqref{eq:flow},
\[
\sigma'=-A(x)\sin\sigma+\left(\frac qy+cy\right)\cos\sigma
\longrightarrow A(x_0)-\frac{qA(x_0)}{q+1}=\frac{A(x_0)}{q+1},
\]
so the curvature of the profile curve stays bounded up to the axis. Write the curve near the endpoint as a graph $x=\tilde x(y)$, $0<y<\delta$, and extend $\tilde x$ evenly to $(-\delta,\delta)$. Since $\tilde x''=\pm\sigma'(1+\tilde x'^2)^{3/2}$ is bounded and $\tilde x'(y)\to0$, the even extension satisfies $\tilde x\in C^{1,1}(-\delta,\delta)$ and $|\tilde x'(t)|\le L_0|t|$.

Set $U(v):=\tilde x(|v|)$ for $v\in\R^{q+1}$. For $v\ne0$,
\[
\nabla U(v)=\phi(|v|)v,
\qquad
\phi(t):=\frac{\tilde x'(t)}{t}.
\]
The bound $|\tilde x'(t)|\le L_0|t|$ gives $|\phi|\le L_0$, and for $0<t\le r$,
\[
t|\phi(r)-\phi(t)|
=\left|\frac tr\left(\tilde x'(r)-\tilde x'(t)\right)
+\tilde x'(t)\left(\frac tr-1\right)\right|
\le2L_0(r-t).
\]
Together with the radial formula for $\nabla U$, this shows that $\nabla U$ is locally Lipschitz at the origin. Hence $U\in C^{1,1}$ near $0$.

Near a point of $\s^p(x_0)\times\{0\}$, the immersed hypersurface can therefore be written, after a fixed Euclidean rotation, as a graph of a function
\[
u\in C^{1,1}=W^{2,\infty}_{\mathrm{loc}}
\]
over its tangent hyperplane. Let $\mathcal Z$ denote the subset of the graph domain corresponding to the singular orbit. It has Lebesgue measure zero (indeed codimension $q+1\ge2$ in the graph domain). On the complement of $\mathcal Z$, the original smooth immersion satisfies \eqref{eq:soliton}. In graph coordinates this is a quasilinear equation
\begin{equation}\label{eq:graph-shrinker}
a^{ij}(z,u,Du)D_{ij}u=b(z,u,Du),
\end{equation}
whose coefficients are smooth (in fact analytic) in their arguments. Since $u\in W^{2,\infty}_{\mathrm{loc}}$, equation \eqref{eq:graph-shrinker} holds almost everywhere off $\mathcal Z$, and therefore almost everywhere on the whole graph domain. Thus the $C^{1,1}$ extension is a strong solution across $\mathcal Z$.

We now make the regularity step explicit. Shrinking the graph neighborhood if necessary, $Du$ is bounded, and the matrix $a^{ij}(z,u,Du)$ is uniformly elliptic there. Because $u\in C^{1,1}$ and the coefficients in \eqref{eq:graph-shrinker} are smooth, the functions
\[
\widetilde a^{ij}(z):=a^{ij}(z,u(z),Du(z)),
\qquad
\widetilde b(z):=b(z,u(z),Du(z))
\]
are locally Lipschitz. Apply difference quotients to \eqref{eq:graph-shrinker} in a coordinate direction. More explicitly, fixing a coordinate direction $e_k$ and setting
\[
\delta_{h,k}v(z):=\frac{v(z+he_k)-v(z)}{h},
\qquad \vartheta_h:=\delta_{h,k}u
\]
for $h\ne0$ sufficiently small, we obtain, almost everywhere on a smaller neighborhood,
\begin{equation}\label{eq:GT-difference-quotients}
\widetilde a^{ij}(z+he_k)D_{ij}\vartheta_h(z)
=\delta_{h,k}\widetilde b(z)-\left(\delta_{h,k}\widetilde a^{ij}(z)\right)D_{ij}u(z).
\end{equation}
Thus the difference quotients satisfy uniformly elliptic linear equations whose right-hand sides are locally bounded uniformly in $h$, since $\widetilde a^{ij}$ and $\widetilde b$ are locally Lipschitz and $D^2u\in L^\infty_{\mathrm{loc}}$. For each fixed nonzero increment, these difference quotients belong to $W^{2,r}_{\mathrm{loc}}$ for every $1<r<\infty$, since $u\in W^{2,\infty}_{\mathrm{loc}}$. Moreover, the ellipticity constants and moduli of continuity of the translated leading coefficients, as well as the local $L^r$ bounds of the right-hand sides, are uniform in the increment. The $L^r$ norms of the difference quotients are uniformly bounded because $u$ is Lipschitz. Hence the interior $W^{2,r}$ estimates \cite[Theorem~9.11]{GT} apply uniformly to the difference quotients. Passing to the limit as the increment tends to zero yields
\[
u\in W^{3,r}_{\mathrm{loc}}
\qquad\text{for every }1<r<\infty.
\]
Choosing $r$ larger than the dimension of the graph and using Sobolev embedding gives $u\in C^{2,\alpha}_{\mathrm{loc}}$ for some $\alpha\in(0,1)$. Equation \eqref{eq:graph-shrinker} is then a uniformly elliptic equation with $C^{1,\alpha}$ coefficients, and the interior Schauder estimates bootstrap successively to $u\in C^{k,\alpha}_{\mathrm{loc}}$ for every $k$, hence $u\in C^\infty$.

Therefore the even radial extension closes the immersion smoothly over $\s^p(x_0)\times\{0\}$. Once this smoothness has been established, the radial symmetry of $U(v)=\tilde x(|v|)$ implies, in particular, the usual vanishing of the odd derivatives of the one-variable even profile $X$ at $0$.

\end{proof}

We need a third monotone quantity. Recall that $n=p+q+1$, and set
\[
h(\psi):=\cos^{p}\psi\sin^{q}\psi,\qquad \psi\in(0,\tfrac\pi2).
\]
Since $(\log h)'=q\cot\psi-p\tan\psi$, the function $h$ attains its maximum exactly at the Clifford angle $\psi_0$.

\begin{lemma}[A third monotone quantity]\label{lem:terceira}
Let $c\in\R$ and, along a solution of \eqref{eq:flow}, let
\[
\mathcal F:=\Phi(x,y)\langle\gamma,\gamma'\rangle=\Phi(x,y)(x\cos\sigma+y\sin\sigma).
\]
Then
\begin{equation}\label{eq:F}
\mathcal F'=\Phi(n+c\rho^{2}).
\end{equation}
Moreover, with $w=\psi-\sigma$, the function $L:=h(\psi)\cos w=\mathcal F/(\rho^{n}e^{c\rho^{2}/2})$ satisfies
\begin{equation}\label{eq:L}
\rho\frac{dL}{ds}=(n+c\rho^{2}) h(\psi)\sin^{2}w .
\end{equation}
\end{lemma}
\begin{proof}
Using $\Phi'=\Phi(A\cos\sigma+B\sin\sigma)$, the identity
\[
\frac{d}{ds}(x\cos\sigma+y\sin\sigma)=1+\sigma'(-x\sin\sigma+y\cos\sigma)
\]
and \eqref{eq:flow}, we obtain
\[
\begin{aligned}
\frac{\mathcal F'}{\Phi}&=(A\cos\sigma+B\sin\sigma)(x\cos\sigma+y\sin\sigma)\\
&\quad+1+(-A\sin\sigma+B\cos\sigma)(-x\sin\sigma+y\cos\sigma)\\
&=1+Ax+By,
\end{aligned}
\]
the mixed terms canceling, and $1+Ax+By=1+p+q+c(x^{2}+y^{2})$, which is \eqref{eq:F}. Next, $\langle\gamma,\gamma'\rangle=\rho\cos w$ by \eqref{eq:polar}, and $\Phi=\rho^{p+q}h(\psi)e^{c\rho^{2}/2}$, so that $\mathcal F=\rho^{n}e^{c\rho^{2}/2}L$. Since $\rho'=\cos w$, differentiation gives
\[
\rho \mathcal F'=(n+c\rho^{2})\cos w\;\rho^{n}e^{c\rho^{2}/2}L+\rho^{n}e^{c\rho^{2}/2}\rho L',
\]
while \[\rho \mathcal F'=(n+c\rho^{2})\rho\Phi=(n+c\rho^{2})\rho^{n}e^{c\rho^{2}/2}h\] by \eqref{eq:F}. Hence \[\rho L'=(n+c\rho^{2})(h-L\cos w)=(n+c\rho^{2})h\sin^{2}w.\]
\end{proof}

The function $\mathcal F$ turns the round sphere into a third separatrix, playing for the radial motion the role that the lines $\{y=y_*\}$ and $\{x=x_*\}$ play for $I$ and $J$. Let $c<0$ and $r_0:=\sqrt{n/|c|}$, so that $\s^{p+q+1}(r_0)$ is the round sphere solution. Its profile curve is the arc $\{\rho=r_0\}$ of $\QQ$, traversed with $\sigma=\psi\pm\frac\pi2$. By \eqref{eq:F},
\[
\mathcal F'=|c|\Phi\left(r_0^{2}-\rho^{2}\right),
\]
so $\mathcal F$ is strictly increasing along every subarc inside the open ball $\{\rho<r_0\}$ and strictly decreasing along every subarc outside the closed ball $\{\rho\le r_0\}$. Since $\mathcal F=\Phi\rho\rho'$ with $\Phi,\rho>0$, the critical points of $\rho$ along $\gamma$ are exactly the zeros of $\mathcal F$.

\begin{theorem}[The sphere as a third separatrix]\label{thm:esfera}
Let $c<0$ and let $\gamma$ be a profile curve distinct from that of the round sphere. Then:
\begin{enumerate}
\item[(i)] every contact of $\gamma$ with the circle $\{\rho=r_0\}$ is transverse, that is, $\rho'\ne0$ there;
\item[(ii)] every critical point of $\rho$ along $\gamma$ is nondegenerate: it is a strict local minimum if it lies in $\{\rho<r_0\}$ and a strict local maximum if it lies in $\{\rho>r_0\}$. In particular, every local minimum of $\rho$ lies in $\{\rho<r_0\}$ and every local maximum in $\{\rho>r_0\}$;
\item[(iii)] between two consecutive crossings of $\{\rho=r_0\}$, $\rho$ has exactly one critical point, a minimum if the arc lies inside the sphere and a maximum if it lies outside. Between an endpoint of $\gamma$ and the adjacent crossing, $\rho$ has at most one critical point. In particular, the crossings of $\{\rho=r_0\}$ and the critical points of $\rho$ alternate along $\gamma$;
\item[(iv)] every loop of $\gamma$ contained in $\{\rho<r_0\}$ contains exactly one critical point of $\rho$, a minimum, and every loop contained in $\{\rho>r_0\}$ contains exactly one, a maximum;
\item[(v)] if $\gamma$ is closed, it crosses $\{\rho=r_0\}$ an even number of times, and at least twice.
\end{enumerate}
\end{theorem}

\begin{proof}
We first show that a maximal trajectory distinct from the sphere must meet $\{\rho=r_0\}$. Since $c<0$ and $p,q\ge1$, the function $\Phi\rho$ is bounded on $\mathcal Q$, and we set $K:=\sup_{\mathcal Q}\Phi\rho<\infty$. Suppose, to the contrary, that $\gamma$ never meets $\{\rho=r_0\}$. Then either $\rho<r_0$ everywhere or $\rho>r_0$ everywhere. By \eqref{eq:F}, $\mathcal F$ is strictly increasing in the first case and strictly decreasing in the second, so we may choose $s_0$ with $\mathcal F(s_0)\ne0$.

If $\rho<r_0$, follow the trajectory forward when $\mathcal F(s_0)>0$, or backward when $\mathcal F(s_0)<0$. The monotonicity of $\mathcal F$ gives $|\mathcal F(s)|\ge |\mathcal F(s_0)|$ in the chosen direction, and $\rho$ increases toward $r_0$, with $|\rho'|=|\mathcal F|/(\Phi\rho)\ge |\mathcal F(s_0)|/K$. If $\rho>r_0$, follow the trajectory backward when $\mathcal F(s_0)>0$, or forward when $\mathcal F(s_0)<0$. Again $|\mathcal F(s)|\ge |\mathcal F(s_0)|$, while $\rho$ decreases toward $r_0$ at speed at least $|\mathcal F(s_0)|/K$. In both cases, since no crossing occurs, the chosen endpoint must be attained in finite arclength, and $\rho$ remains bounded. The unit-speed curve then has a limit in the closed quadrant. By maximality and smoothness of the vector field in $\mathcal Q\times\mathbb S^1$, this limit must lie on a coordinate axis, possibly at the origin. But $\mathcal F=\Phi\rho\rho'\to0$ there, since $|\rho'|\le1$, $\rho$ is bounded and $\Phi=x^p y^q e^{c\rho^2/2}\to0$. This contradicts $|\mathcal F(s)|\ge |\mathcal F(s_0)|>0$. Hence every such trajectory meets $\{\rho=r_0\}$.

(i) If a contact is not transverse, then $\rho=r_0$ and $\rho'=\cos(\psi-\sigma)=0$, so $\sigma=\psi\pm\frac\pi2$ and $\Gamma$ lies on the trajectory of the round sphere. By uniqueness of solutions, $\gamma$ would be the profile curve of the sphere.

(ii) Since $\mathcal F'=\Phi'\rho\rho'+\Phi\left(\rho'^{2}+\rho\rho''\right)$, at a critical point $s_0$ of $\rho$ we get, from \eqref{eq:F},
\[
\rho\rho''=n+c\rho^{2}=|c|\left(r_0^{2}-\rho^{2}\right).
\]
By (i), $\rho(s_0)\ne r_0$, so $\rho''(s_0)\ne0$, and the second derivative test gives the conclusion.

(iii) By (ii), the critical points of $\rho$ are isolated strict local extrema. On a maximal interval where $\rho<r_0$, $\mathcal F$ is strictly increasing and vanishes at most once and, as $\mathcal F=\Phi\rho\rho'$, we have that $\rho$ has at most one critical point there. If the interval is bounded by two crossings, then $\rho$ equals $r_0$ at both ends and is smaller in between, so it attains an interior minimum, which is a critical point. This gives exactly one critical point, a minimum. Outside the sphere the argument is the same, with $\mathcal F$ strictly decreasing and an interior maximum. For the alternation, two consecutive critical points are a minimum and a maximum, since between two local minima there is a local maximum, and by (ii) $\rho$ passes from $\rho<r_0$ to $\rho>r_0$ or back between them. As $\rho$ is strictly monotone between consecutive critical points, it crosses $\{\rho=r_0\}$ exactly once there.

(iv) Let $\gamma|_{[a,b]}$ be a loop contained in $\{\rho<r_0\}$. Then $\rho(a)=\rho(b)$, and $\rho$ is not constant on $[a,b]$, since otherwise $\mathcal F\equiv0$ there, whereas $\mathcal F'>0$. Hence $\rho$ has an interior extremum on $[a,b]$, which is a critical point, and by (iii) it is the only one, a minimum. The case $\{\rho>r_0\}$ is analogous.

(v) Along a closed profile curve $\rho$ attains a global maximum and a global minimum, which are local extrema. By (ii) the first exceeds $r_0$ and the second is smaller than $r_0$, so $\gamma$ crosses $\{\rho=r_0\}$. By (i) the crossings are transverse, so $\rho-r_0$ changes sign at each of them, and it returns to its initial value after one period, so their number is even.

\end{proof}

Figure~\ref{fig:rho-critico} illustrates Theorem~\ref{thm:esfera} along a numerically computed trajectory.

\begin{figure}[ht]
\centering
\includegraphics[width=\textwidth]{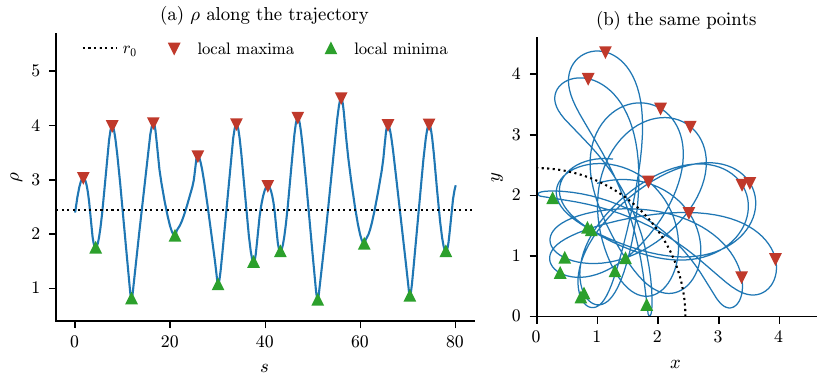}
\caption{Theorem~\ref{thm:esfera} for $p=q=1$, $c=-\tfrac12$, along the trajectory with initial condition $(x,y,\sigma)=(1,2.2,0.3)$, $s\in[0,80]$. (a) The ten local maxima of $\rho$ satisfy $\rho\ge2.89$ and the ten local minima satisfy $\rho\le1.97$, on either side of $r_0=\sqrt{n/|c|}=\sqrt6\approx2.45$ (dotted). (b) The same points in the plane, with the circle of radius $r_0$, which is the profile of the round sphere.}
\label{fig:rho-critico}
\end{figure}

It remains to describe the endpoints at the origin. As in the minimal case, only the Clifford ray reaches it.

\begin{proposition}[The origin is reached only by the Clifford ray]\label{lem:origem}
Let $c\ne0$ and let $\gamma$ be a maximal solution of \eqref{eq:flow} whose profile curve
satisfies $x(s)\to0$ and $y(s)\to0$ as $s\to s_\pm$. Then $\gamma$ is the Clifford ray
\[
\gamma_{\mathrm C}(s)=\left(s\cos\psi_0,\ s\sin\psi_0\right),\quad \sigma\equiv\psi_0,
\quad s\in(0,\infty),
\quad
\tan\psi_0:=\sqrt{q/p},
\]
which gives rise to the Clifford minimal cone $\mathcal C:=\left\{(u,v):\ |u|/|v|=\sqrt{p/q}\right\}$, a cone over $\s^{p}\left(\sqrt{p/(p+q)}\right)\times\s^{q}\left(\sqrt{q/(p+q)}\right),$ minus the vertex.
\end{proposition}
\begin{proof}
By reversibility we may assume that the origin is reached at the endpoint $\alpha$, as $s\to s_-$.

\emph{Step 1: monotone approach.} Near the origin, $x<x_*$ and $y<y_*$ if $c<0$, so $A,B>0$ there. For $c>0$ this holds everywhere. By \eqref{eq:dJ}, and Lemma~\ref{lem:Z}(ii), $I$ and $J$ are strictly increasing near $\alpha$, and $|I|\le f(x)\to0$, $|J|\le g(y)\to0$. Hence $I>0$ and $J>0$ near $\alpha$, that is, $\sigma\in(0,\tfrac\pi2)$, and $x$, $y$ are strictly increasing. In particular the arclength near $\alpha$ is finite, $s_->-\infty$, and, since also $\psi\in(0,\tfrac\pi2)$, the relative phase $w=\psi-\sigma$ lies in $(-\tfrac\pi2,\tfrac\pi2)$.

\emph{Step 2: logarithmic time.} Fix $s_1$ close to $s_-$ and define
\[
t(s):=\int_{s_1}^{s}\frac{dr}{\rho(r)},\qquad s\in(s_-,s_1].
\]
Since $|\gamma'|=1$ and $\gamma(r)\to0$ as $r\to s_-$, the length of $\gamma|_{(s_-,r]}$ is $r-s_-$, so that
\[
\rho(r)=|\gamma(r)-0|\le r-s_-,
\]
so that
\[
-t(s)=\int_{s}^{s_1}\frac{dr}{\rho(r)}\ \ge\ \int_{s}^{s_1}\frac{dr}{r-s_-}=\log\frac{s_1-s_-}{s-s_-}\ \xrightarrow[s\to s_-]{}\ \infty .
\]
Thus $t$ is an increasing diffeomorphism from $(s_-,s_1]$ onto $(-\infty,0]$, and $t\to-\infty$ as $s\to s_-$. By \eqref{eq:polar}, $(\psi,w,\rho)$ solves, in the variable $t$, the autonomous system
\begin{equation}\label{eq:Z}
\left\{
\begin{aligned}
\frac{d\psi}{dt}&=-\sin w,\\
\frac{dw}{dt}&=-(1+c\rho^{2})\sin w-E(\psi,\psi-w),\\
\frac{d\rho}{dt}&=\rho\cos w,
\end{aligned}\right.
\end{equation}
whose right-hand side is analytic on the open set $(0,\tfrac\pi2)\times\R\times\R$ of the variables $(\psi,w,\rho)$. Note that $\rho$ is now merely a coordinate of this phase space, in which the system extends analytically across $\rho=0$, although the corresponding curve in $\QQ$ degenerates to the origin there. The coordinate plane $\{\rho=0\}$ is invariant, since $d\rho/dt=\rho\cos w$ vanishes on it, and on that plane \eqref{eq:Z} is the system of the minimal case $c=0$. Along our solution, $\rho(t)\to0$ as $t\to-\infty$, so the orbit accumulates on $\{\rho=0\}$.

\emph{Step 3: a Lyapunov function.} By Lemma~\ref{lem:terceira}, $L=h(\psi)\cos w$ satisfies 
\[
\frac{dL}{dt}=(n+c\rho^{2})h(\psi)\sin^{2}w\ge0
\] as long as $c\rho^{2}>-n$, which holds near $\alpha$. Moreover, along our solution, \[0<L\le h(\psi)\le h(\psi_0),\] since $h$ attains its maximum at the Clifford angle. Hence $L(t)$ converges, as $t\to-\infty$, to some $L_{-\infty}\in[0,h(\psi_0)]$.

\emph{Step 4: interior limit points.} By Step~1, $(\psi(t),w(t))$ belongs to the compact set $K:=[0,\tfrac\pi2]\times[-\tfrac\pi2,\tfrac\pi2]$ for every $t\le0$. Let
\[
\Lambda:=\bigcap_{T\le0}\overline{\{(\psi(t),w(t)):\ t\le T\}}\ \subset\ K
\]
be its set of limit points as $t\to-\infty$. The sets in the intersection are nonempty compact subsets of $K$, they decrease as $T$ decreases, and each of them is connected, being the closure of a continuous image of an interval. A nested intersection of nonempty compact connected sets is nonempty, compact and connected, so $\Lambda$ has these three properties. Moreover $L$ is continuous on $K$ and $L(t)\to L_{-\infty}$, so
\begin{equation}\label{eq:Lconst}
L\equiv L_{-\infty}\quad\text{on }\Lambda .
\end{equation}

Suppose that $\Lambda$ contains a point $z=(\psi_1,w_1)$ with $\psi_1\in(0,\tfrac\pi2)$, say $(\psi,w)(t_k)\to z$ with $t_k\to-\infty$. Since $\rho(t)\to0$, the full state satisfies
\[
Z(t_k):=(\psi,w,\rho)(t_k)\ \longrightarrow\ (z,0).
\]
The point $(z,0)$ lies in the domain of analyticity of \eqref{eq:Z}, so the solution $\zeta$ of \eqref{eq:Z} with $\zeta(0)=(z,0)$ is defined on $[-\tau_0,\tau_0]$ for some $\tau_0>0$, and it stays in $\{\rho=0\}$ by the invariance of that plane. By continuous dependence on the initial condition, uniform on the compact interval $[-\tau_0,\tau_0]$,
\[
Z(t_k+\tau)\ \longrightarrow\ \zeta(\tau)\qquad\text{uniformly for }|\tau|\le\tau_0 .
\]
Since $t_k+\tau\to-\infty$ for each fixed $\tau$, the $(\psi,w)$-part of $\zeta(\tau)$ belongs to $\Lambda$ for every $|\tau|\le\tau_0$. In other words, $\Lambda$ contains a whole trajectory of the reduced system.

By \eqref{eq:Lconst}, the function $\tau\mapsto L(\zeta(\tau))$ is constant on $[-\tau_0,\tau_0]$. Lemma~\ref{lem:terceira} gives the derivative of $L$ along the flow of \eqref{eq:Z} only where $\rho>0$, which is where that flow is the flow of the profile curves:
\[
\frac{dL}{d\tau}=(n+c\rho^{2})h(\psi)\sin^{2}w .
\]
Both sides of this identity are continuous on the whole domain of \eqref{eq:Z}, and $\{\rho>0\}$ is dense in it, so the identity also holds on $\{\rho=0\}$. Along $\zeta$ we have $\rho\equiv0$, hence
\[
0=\frac{d}{d\tau}L(\zeta(\tau))=n h(\psi)\sin^{2}w\qquad\text{on }[-\tau_0,\tau_0].
\]
Shrinking $\tau_0$ if necessary, $\psi$ stays in $(0,\tfrac\pi2)$ along $\zeta$, where $h>0$. Hence $\sin w\equiv0$ and, since $|w|\le\tfrac\pi2$, $w\equiv0$ along $\zeta$. The second equation of \eqref{eq:Z}, with $\rho=0$ and $w\equiv0$, then gives
\[
0=\frac{dw}{d\tau}=-E(\psi,\psi)=p\tan\psi-q\cot\psi,
\]
so $\psi\equiv\psi_0$ along $\zeta$. In particular $z=(\psi_0,0)$ and, by \eqref{eq:Lconst},
\[
L_{-\infty}=h(\psi_0)\cos0=h(\psi_0).
\]

For every $t\le0$,
\[
h(\psi_0)=L_{-\infty}\le L(t)=h(\psi(t))\cos w(t)\le h(\psi(t))\le h(\psi_0),
\]
where the first inequality holds because $L$ is nondecreasing and $L_{-\infty}$ is its limit at $-\infty$. Hence equality holds throughout, that is, $\cos w(t)=1$ and $h(\psi(t))=h(\psi_0)$ for every $t\le0$. Since $h$ attains its maximum only at $\psi_0$, this gives $w\equiv0$ and $\psi\equiv\psi_0$, that is, $\sigma\equiv\psi_0$ and $\gamma$ is a ray through the origin. By uniqueness of solutions, $\gamma=\gamma_{\mathrm C}$.

\emph{Step 5: boundary limit points.} There remains the case $\Lambda\subset\{\psi=0\}\cup\{\psi=\tfrac\pi2\}$. By connectedness $\Lambda$ lies in one of them, say $\Lambda\subset\{\psi=0\}$, that is, $y/x\to0$ at $\alpha$. The other case is the mirror image, with $J$ in place of $I$. Near $\alpha$ we have $B(y)\ge q/(2y)$ and $\tfrac12x^{p}\le f(x)\le2x^{p}$. We use $\xi=x$ as the parameter, which is legitimate by Step~1, with $d\xi=\cos\sigma ds$. Since $I\to0$ at $\alpha$, \eqref{eq:dJ}, gives
\begin{equation}\label{eq:borda0}    
2x^{p}\ \ge\ f(x)\ \ge\ I=\int_0^{x}B f\cos\sigma d\xi\ \ge\ \frac q4\int_0^{x}\frac{\xi^{p}\cos\sigma}{y} d\xi.
\end{equation}

Let $\varepsilon(x):=\sup_{0<\xi\le x}y(\xi)/\xi$, so that $\varepsilon(x)\to0$ as $x\to0$. Then $y\le\varepsilon\xi$, and \eqref{eq:borda0} yields
\begin{equation}\label{eq:borda1}
\int_0^{x}\xi^{p-1}\cos\sigma d\xi\ \le\ \frac{8\varepsilon(x)}{q} x^{p}.
\end{equation}
On the other hand, let $\Omega:=\{\xi\in(0,x]:\cos\sigma\ge\tfrac12\}$. On $(0,x]\setminus \Omega$ one has $\sigma>\tfrac\pi3$, so $\tan\sigma>\sqrt3$. Since $dy/d\xi=\tan\sigma>0$ on $(0,x]$ and $y\to0$ at $\alpha$,
\[
\sqrt3\left|(0,x]\setminus \Omega\right|\ \le\ \int_{(0,x]\setminus \Omega}\tan\sigma d\xi\ \le\ \int_0^{x}\tan\sigma d\xi=y(x)\le\varepsilon(x) x,
\]
so the measure of $(0,x]\setminus \Omega$ is at most $\varepsilon(x) x/\sqrt3$. As $\xi^{p-1}$ is nondecreasing,
\[
\begin{aligned}
\int_0^{x}\xi^{p-1}\cos\sigma d\xi\ &\ge\ \frac12\int_\Omega\xi^{p-1}d\xi\ \ge\ \frac12\int_0^{x(1-\varepsilon/\sqrt3)}\xi^{p-1}d\xi\\
&=\frac{x^{p}}{2p}\left(1-\frac{\varepsilon(x)}{\sqrt3}\right)^{p}.
\end{aligned}
\]
Together with \eqref{eq:borda1}, $(1-\varepsilon(x)/\sqrt3)^{p}\le16 p\varepsilon(x)/q$ for all small $x$, which is impossible since $\varepsilon(x)\to0$.
\end{proof}

The descriptions of the endpoints on the axes, at the origin and at infinity combine into the completeness of the hypersurfaces.

\begin{theorem}[Completeness]\label{thm:completude}
Let $c\ne0$. Every bi-rotational hypersurface satisfying \eqref{eq:soliton} and associated with a maximal solution of \eqref{eq:flow}, different from the Clifford cone, is complete.
\end{theorem}
\begin{proof}
The profile curve is parametrized by arclength ($|(x',y')|=1$, Lemma~\ref{lem:noeq}), and this parameter coincides with the arclength of the induced metric along $\gamma$. Hence $\Sigma$ is complete if and only if each endpoint of the maximal solution of \eqref{eq:flow} is either (a) reached at parameter $s\to\pm\infty$, or (b) an orthogonal endpoint on an axis, over which $\Sigma$ closes off smoothly. Consider the endpoint $s_+$. The endpoint $s_-$ is analogous. If $s_+=+\infty$, we are in case (a). If $s_+<\infty$, then, since $|\gamma'|=1$,
\[
|\gamma(s)-\gamma(r)|=\left|\int_r^{s}\gamma'(t) dt\right|\le|s-r|
\]
for $r,s$ near $s_+$, so $\gamma$ is $1$-Lipschitz on a bounded interval, hence uniformly continuous, and extends continuously to $s_+$, converging to some $P\in\overline{\mathcal Q}$. Moreover, $P\notin\mathcal Q$, since otherwise $\Gamma(s)$ would remain, near $s_+$, in the compact set $\overline{B(P,r)}\times\s^1\subset\mathcal Q\times\s^1$ for small $r>0$, whereas a maximal solution with $s_+<\infty$ leaves every compact subset of the domain \cite[Corollary~2.16]{Teschl}. Thus $P\in\partial\mathcal Q$. If $P=(x_0,0)$ with $x_0>0$, Proposition~\ref{lem:cusp} and Proposition~\ref{lem:suave} show that it is an orthogonal endpoint over which $\Sigma$ closes off smoothly, and symmetrically if $P=(0,y_0)$. This is case (b). If $P$ is the origin, $\gamma$ is the Clifford ray by Proposition~\ref{lem:origem}. Hence, away from the Clifford cone, every endpoint is of type (a) or (b), and $\Sigma$ is complete.
\end{proof}

\subsection{The compact case}\label{sec:compactas}

A compact hypersurface is automatically complete. The relevant question is what its topology is and where, in the phase plane, its profile curve lives. As a bounded curve does not escape to infinity, the arguments below are shorter than those of the general case.

\begin{proposition}[Topology of the compact ones]\label{prop:topcompacta}
If $\Sigma$ is a compact bi-rotational self-shrinker, its profile curve is of one of three types, and $\Sigma$ is the corresponding smooth closed manifold:
\begin{enumerate}
\item[\textup{(I)}] a closed loop in the interior of $\mathcal Q$, with $\Sigma\cong\s^{1}\times\s^{p}\times\s^{q}$;
\item[\textup{(IIa)}] an arc with one endpoint on the $x$-axis and the other on the $y$-axis, with
$\Sigma\cong\s^{p+q+1}$;
\item[\textup{(IIb)}] an arc with both endpoints on the $x$-axis, with $\Sigma\cong\s^{q+1}\times\s^{p}$, or both endpoints on the $y$-axis, with $\Sigma\cong\s^{p+1}\times\s^{q}$.
\end{enumerate}
In all cases, Proposition~\ref{lem:suave} guarantees the $C^\infty$ closure at the endpoints on the axes, so that $\Sigma$ is a closed manifold without boundary.
\end{proposition}
\begin{proof}
A compact profile curve in $\overline{\mathcal Q}$ is either a closed loop in the interior (it does not touch the axes) or an arc whose endpoints reach the axes. In the second case, each endpoint is orthogonal to its axis (Proposition~\ref{lem:cusp}), and the two endpoints are on the same axis (IIb) or on distinct axes (IIa). The topological type follows from \eqref{eq:param}: an endpoint on the $x$-axis collapses the factor $\s^q$ (filling in a disk $\R^{q+1}$), and one on the $y$-axis collapses $\s^p$. In (I) no factor collapses, giving $\s^1\times\s^p\times\s^q$. In (IIa) both collapse, giving a sphere, and in (IIb) $\s^q$ collapses at both endpoints, giving $\s^{q+1}\times\s^p$.
\end{proof}

The next result shows that compact bi-rotational self-shrinkers cannot occur just anywhere, but must be contained within a specific region of space.

\begin{corollary}[Compact ones encircle the vertex]\label{thm:compacta-vertice}
Let $c<0$. The profile curve of \emph{every} compact bi-rotational self-shrinker crosses both separatrices $\{y=y_*\}$ and $\{x=x_*\}$. Consequently, the
vertex $V=(x_*,y_*)$ lies in the interior of the bounding box $[x_{\min},x_{\max}]\times[y_{\min},y_{\max}]$ of the profile curve, i.e., the compact one ``encircles'' $V$.
\end{corollary}
\begin{proof}
As $\Sigma$ is non-cylindrical (the cylinders are noncompact) and bounded, by Theorem~\ref{thm:cross}, its profile curve crosses $\{y=y_*\}$. For $\{x=x_*\}$ we apply the mirror statement (Corollary~\ref{cor:crossI}), which is equally unconditional (Corollary~\ref{cor:fugaH}). Since $\gamma$ crosses $x=x_*$, one has $x_{\min}<x_*<x_{\max}$, and analogously $y_{\min}<y_*<y_{\max}$, i.e., $V$ is interior to the box.
\end{proof}

In terms of the hypersurface, Corollary~\ref{thm:compacta-vertice} says that a compact $\Sigma$ meets both cylinders. Combined with Corollaries~\ref{cor:custo} and~\ref{cor:noloop2}, if $n(\gamma)=k$, then $\Sigma$ meets each of $\mathcal C_{\mathrm v}$ and $\mathcal C_{\mathrm h}$ in at least $k+1$ connected components, and if, moreover, $\Sigma\cong\s^{1}\times\s^{p}\times\s^{q}$, the number of these components is even and at least $\max(2,k+1)$.

\begin{remark}
    The round sphere $\s^{p+q+1}(r_0)$, which meets each cylinder in a single connected component (since $x_*,y_*<r_0$), shows that the parity statement does not extend to the other compact types.
\end{remark}

\subsection{The noncompact case}\label{sec:naocompactas}

We turn to the hypersurfaces that are not compact. Their classification is a direct reading of the end structure of Section~\ref{sec:cross}, together with the smooth closure of Proposition~\ref{lem:suave}.

\begin{theorem}[Topological type of the noncompact ones]\label{thm:completas}
Let $c<0$ and let $\Sigma$ be a noncompact bi-rotational self-shrinker, distinct from the Clifford cone. Then $\Sigma$ is complete and diffeomorphic to one of the following models, according to how many endpoints of the profile curve reach an axis, and to which axis:
\[
\R\times\s^{p}\times\s^{q},
\qquad
\R^{p+1}\times\s^{q},
\qquad
\R^{q+1}\times\s^{p},
\]
according as its profile curve has no endpoint on a symmetry axis, one endpoint on the
$y$-axis, or one endpoint on the $x$-axis. 

Moreover, each end of $\Sigma$ is asymptotically conical or non-escaping.
\end{theorem}
\begin{proof}
Completeness is Theorem~\ref{thm:completude}. The profile curve is not closed, since $\Sigma$ is noncompact, and by the proof of Theorem~\ref{thm:completude} each endpoint of its maximal interval is either reached as $s\to\pm\infty$ or is an orthogonal endpoint on an axis, reached at a finite parameter. If both endpoints were on the axes, the parameter interval would be bounded and $\Sigma$ would be one of the closed manifolds of Proposition~\ref{prop:topcompacta}. Hence at most one endpoint lies on an axis. With no endpoint on an axis, $\Sigma$ is parametrized by $\R\times\s^{p}\times\s^{q}$. With one endpoint on the $x$-axis, the factor $\s^{q}$ collapses there, and by Proposition~\ref{lem:suave} $\Sigma$ is a smooth manifold diffeomorphic to $\R^{q+1}\times\s^{p}$. The $y$-axis is analogous. The description of the ends is Theorem~\ref{thmB}.
\end{proof}

\subsection{Tame or wild}

Our classification splits the ends into conical and non-escaping ones. The intersections with the cylinders distinguish the two cases: the only alternative to conical ends is meeting both cylinders infinitely often.

\begin{corollary}[Tame or wild]\label{cor:dicotomia}
Let $c<0$ and let $\Sigma$ be a complete, noncompact bi-rotational
self-shrinker, distinct from the two cylinders. Then only one of the following holds:
\begin{enumerate}
\item[(i)] Every end of $\Sigma$ is asymptotically conical, with limit direction $\psi_\infty\in(0,\tfrac\pi2)$, and it is embedded. Along each of them the tangent sweeps less than half a turn and $|\mathcal A|^{2}=O(\rho^{-2})$. In this case $\Sigma$ is proper, and it meets each cylinder in finitely many connected components;
\item[(ii)] $\Sigma$ has a non-escaping end. In this case $\Sigma$ is not proper, and it meets both cylinders transversely in infinitely many connected components.
\end{enumerate}
In particular, a properly immersed complete noncompact bi-rotational self-shrinker has only asymptotically conical embedded ends.
\end{corollary}

\begin{proof}
By Theorem~\ref{thmB} each end is of type (a) or (b), and $\Sigma$ has at least one end since it is noncompact. The two alternatives are ``no end of type (b)'' and ``some end of type (b)'', which are exclusive and exhaustive.

In case (i) the description of the ends is Theorem~\ref{thmB}. Properness follows because $\rho\to\infty$ along each of the finitely many ends, so the preimage of a compact set
is a compact subset of the parameter interval. Finiteness of the number of intersections with each cylinder is then immediate: crossings are isolated (Lemma~\ref{lem:transv}) and cannot accumulate at an end where $\rho\to\infty$.

In case (ii), let the end be $s\to s_+$ with infinite arclength and $\rho\not\to\infty$. Then $\liminf\rho<\infty$, so there are $s_k\to s_+$ and $R<\infty$ with $\rho(s_k)\le R$. Hence the preimage under \eqref{eq:param} of the compact set generated by $\overline{B(0,2R)}\cap\mathcal Q$ contains the points $s_k$, and it is a closed unbounded subset of the parameter interval. Therefore it is not compact and $\Sigma$ is not proper. The fact that the end crosses both separatrices infinitely often is Theorem~\ref{prop:recorrente} and each crossing is transverse by Lemma~\ref{lem:transv}. Crossings of $\{y=y_*\}$ and $\{x=x_*\}$ are precisely the intersections of $\Sigma$ with $\mathcal C_{\mathrm h}$ and $\mathcal C_{\mathrm v}$.
\end{proof}

The dichotomy may now be restated in the language in which such hypotheses usually appear in the literature.

\begin{corollary}[Properness, volume growth and entropy]\label{cor:entropia}
Let $c<0$ and let $\Sigma$ be a complete, noncompact bi-rotational
self-shrinker, distinct from the two cylinders. Then the following are equivalent:
\begin{enumerate}
\item[(a)] every end of $\Sigma$ is asymptotically conical;
\item[(b)] $\Sigma$ is proper;
\item[(c)] $\Sigma$ has polynomial volume growth;
\item[(d)] $\Sigma$ has finite Gaussian area at its natural shrinker scale, equivalently finite entropy;
\item[(e)] $\Sigma$ meets each of the two cylinders in finitely many connected components.
\end{enumerate}
When they fail, $\Sigma$ has a non-escaping end, meets both cylinders transversely in
infinitely many connected components, is not proper, and has infinite entropy.
\end{corollary}

\begin{proof}
The equivalence of (a) and (b) is Corollary~\ref{cor:dicotomia}, as is the
equivalence with (e). Along a conical end $x\asymp\rho$ and $y\asymp\rho$, so that
$dV=\omega_p\omega_q x^{p}y^{q} ds\asymp\rho^{p+q}d\rho$ and the volume of
$\Sigma\cap B(0,\rho)$ grows like $\rho^{p+q+1}$. Hence (a)$\Rightarrow$(c).

To compare (b)--(d) with the standard shrinker literature, set
\[
t_c:=-\frac1{2c}>0.
\]
Then $H=c\langle X,\nu\rangle=-\frac1{2t_c}\langle X,\nu\rangle$, and the dilation $\widetilde X=t_c^{-1/2}X$ satisfies the standard normalization $\widetilde H=-\frac12\langle\widetilde X,\widetilde\nu\rangle$. Properness and polynomial volume growth are invariant under this fixed dilation. Hence the theorem of Cheng and Zhou \cite[Theorem~1.3]{ChengZhou}, applied after normalization, gives the equivalence between properness, polynomial volume growth, and finiteness of the natural weighted volume
\[
\int_\Sigma e^{c|X|^2/2}\,dV.
\]
At the natural center and scale $(0,t_c)$, the Gaussian area is
\[
F_{0,t_c}(\Sigma)
=(4\pi t_c)^{-n/2}\int_\Sigma e^{-|X|^2/(4t_c)}\,dV
=(4\pi t_c)^{-n/2}\int_\Sigma e^{c|X|^2/2}\,dV.
\]
If $\Sigma$ has polynomial volume growth, its entropy is finite. Indeed, by \cite[\S7.2, Lemma~7.10]{CM} the entropy of such a self-shrinker is attained at its natural center and scale, so it equals $F_{0,t_c}(\Sigma)$ (see also \cite[\S0]{CMP}, and \cite[Appendix~A]{Stone} for the values on generalized cylinders). Although \cite{CM} works with embedded hypersurfaces, the proof of \cite[Lemma~7.10]{CM} uses only the self-shrinker equation and integrations by parts on $\Sigma$, which are justified by the polynomial volume growth. It therefore applies verbatim to the immersed self-shrinkers considered here. Conversely, finite entropy implies $F_{0,t_c}(\Sigma)<\infty$ because the entropy is the supremum of the $F$-functionals. Therefore (b), (c), and (d) are equivalent. Finally, if (a) fails, then $\Sigma$ is not proper by Corollary~\ref{cor:dicotomia}, so Cheng--Zhou gives $F_{0,t_c}(\Sigma)=\infty$, and hence the entropy is infinite.
\end{proof}

\section{Almost every trajectory is recurrent}\label{sec:generico}

This section proves Theorem~\ref{thmH}. Figure~\ref{fig:recorrencia} shows a trajectory with the behavior that the theorem predicts for almost every initial condition. Let $c<0$, $M:=\mathcal Q\times\s^1$, and let $\varphi_s$ be the local flow of \eqref{eq:flow}. For $v\in M$ let $(s_-(v),s_+(v))$ be the maximal interval of the solution through $v$, and
\[
C^{\pm}:=\{v\in M:\ s_\pm(v)\ \text{is finite}\},\qquad M_\infty:=M\setminus(C^{+}\cup C^{-}).
\]
We use the following standard facts and results. The set $\mathcal D=\{(s,v):s\in(s_-(v),s_+(v))\}$ is open, and each $\varphi_s$ is a diffeomorphism from $\{v:(s,v)\in\mathcal D\}$ onto its image \cite[Theorem~6.1 and the discussion following it]{Teschl}. We also apply the Poincaré recurrence theorem for measure-preserving maps of finite measure spaces \cite[Theorem~1.4]{Walters}.

\begin{figure}[ht]
\centering
\includegraphics[width=0.66\textwidth]{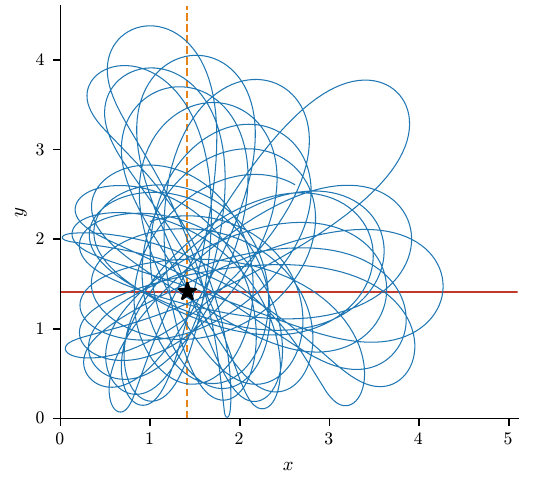}
\caption{The trajectory of $(x,y,\sigma)=(1,2.2,0.3)$ for $p=q=1$, $c=-\tfrac12$ and $s\in[0,200]$, with the behavior that Theorem~\ref{thmH} predicts for almost every initial condition. It stays in $0.55\le\rho\le5.23$, crosses $\{y=y_*\}$ (solid) $54$ times and $\{x=x_*\}$ (dashed) $55$ times, and returns at $s\approx92.3$ to a distance of about $0.14$ from its initial phase point. The star is the vertex $V$.}
\label{fig:recorrencia}
\end{figure}

The functions $s_\pm$ are semicontinuous, and therefore $C^{\pm}$ and $M_\infty$ are Borel sets. Indeed, if $s_+(v_0)>T$, the solution through $v_0$ is defined on $[0,T]$ and its trace is a compact subset of $M$, so by the continuous dependence on the initial condition every solution starting close enough to $v_0$ is also defined on $[0,T]$, that is, $\{s_+>T\}$ is open. Thus $s_+$ is lower semicontinuous, $\{s_+\le T\}$ is closed, and
\[
C^{+}=\{s_+<\infty\}=\bigcup_{n\in\N}\{s_+\le n\}
\]
is an $F_\sigma$ set, hence Borel. The same argument, applied to the reversed flow, shows that $s_-$ is upper semicontinuous and that $C^{-}$ is Borel, and $M_\infty$ is the complement of $C^{+}\cup C^{-}$.

\begin{proposition}[Invariant measure of finite mass]\label{prop:medida}
The flow of \eqref{eq:flow} preserves $\mathfrak m=\Phi dx dy d\sigma$, $\Phi=x^{p}y^{q}e^{c(x^{2}+y^{2})/2}$, and
\[
\mathfrak m(M)=2\pi\frac{2^{\frac{p+q}{2}}\Gamma\left(\tfrac{p+1}{2}\right)\Gamma\left(\tfrac{q+1}{2}\right)}{2|c|^{\frac{p+q}{2}+1}}<\infty .
\]
\end{proposition}
\begin{proof}
Since $\partial_x\log\Phi=A$, $\partial_y\log\Phi=B$ and $\Phi$ does not depend on $\sigma$,
\[
\begin{aligned}
&\dive\left(\Phi(\cos\sigma,\sin\sigma,-A\sin\sigma+B\cos\sigma)\right)\\
&\qquad=\Phi\left(A\cos\sigma+B\sin\sigma-A\cos\sigma-B\sin\sigma\right)=0,
\end{aligned}
\]
the divergence being taken in the coordinates $(x,y,\sigma)$.
To justify measure preservation directly, let $F$ be the vector field of
\eqref{eq:flow}. For every $s$ for which the local flow is defined,
Liouville--Jacobi's determinant identity gives
\[
\frac{d}{ds}\left[\Phi(\varphi_s(v))\det D_v\varphi_s\right]
=\left(\dive(\Phi F)\right)(\varphi_s(v))\det D_v\varphi_s=0.
\]
At $s=0$ the expression in brackets is $\Phi(v)$. Hence the local flow
preserves $\Phi\,dx\,dy\,d\sigma$ on its domain of definition.
The mass factors as
\[
\mathfrak m(M)=2\pi\int_0^\infty x^{p}e^{-|c|x^{2}/2}dx\int_0^\infty y^{q}e^{-|c|y^{2}/2}dy,
\]
and
\[
\int_0^\infty x^{p}e^{-|c|x^{2}/2}dx=2^{\frac{p-1}{2}}|c|^{-\frac{p+1}{2}}\,\Gamma\left(\tfrac{p+1}{2}\right).
\]
\end{proof}

\begin{remark}[$\mathfrak m$ as a Liouville measure]\label{rem:liouville}

The measure $\mathfrak m$ is the Liouville measure of the geodesic flow of $g_\Phi$, written in the time $s$. In the coordinates $(x,y,\sigma)$, the Liouville measure of $(\QQ,g_\Phi)$ is $\Phi^{2}dx\,dy\,d\sigma$, and it is invariant under the geodesic flow parametrized by $g_\Phi$-arclength $\tau$. Since $d\tau=\Phi\,ds$, the field of \eqref{eq:flow} is $\Phi$ times the geodesic field, so its invariant density is $\Phi^{2}/\Phi=\Phi$.

\end{remark}

The trajectories that reach the boundary of the phase space in finite arclength form a negligible set.

\begin{proposition}[Incomplete trajectories are negligible]\label{lem:nulo}
$\mathfrak m(C^{+})=\mathfrak m(C^{-})=0$.
\end{proposition}
\begin{proof}
Let $\ell:=\min(p,q)\ge1$. Fix $\varepsilon\in(0,1)$ and let $G_\varepsilon:=\{u\in C^{-}:-s_-(u)\le\varepsilon\}$. Let $u\in G_\varepsilon$. As in the proof of Theorem~\ref{thm:completude}, applied in the backward time direction, the profile curve $\gamma_u$ of $u$, with $\gamma_u(0)=(x(u),y(u))$, converges as $s\to s_-(u)$ to a point $P\in\partial\mathcal Q$, and, again because $|\gamma_u'|=1$,
\[
\left|P-(x(u),y(u))\right|=\lim_{s\to s_-(u)}\left|\gamma_u(s)-\gamma_u(0)\right|\le\lim_{s\to s_-(u)}|s|=-s_-(u)\le\varepsilon .
\]
Hence $y(u)\le\varepsilon$ or $x(u)\le\varepsilon$, and
\[
\begin{aligned}
\mathfrak m(G_\varepsilon)&\le 2\pi\left(\int_0^\infty x^{p}e^{cx^{2}/2}dx\int_0^{\varepsilon}y^{q}dy+\int_0^\infty y^{q}e^{cy^{2}/2}dy\int_0^{\varepsilon}x^{p}dx\right)\\
&\le K\varepsilon^{\ell+1},
\end{aligned}
\]
with $K$ depending only on $p,q,c$.

Fix $T>0$, let $C^{-}_T:=\{v\in C^{-}:-s_-(v)\le T\}$ and $N:=\lceil T/\varepsilon\rceil$. Given $v\in C^{-}_T$, let $k\in\{0,\dots,N-1\}$ be such that $-s_-(v)\in(k\varepsilon,(k+1)\varepsilon]$. Then $u:=\varphi_{-k\varepsilon}(v)$ is defined, $-s_-(u)=-s_-(v)-k\varepsilon\in(0,\varepsilon]$ and $s_+(u)=s_+(v)+k\varepsilon>k\varepsilon$. Hence
\[
C^{-}_T\ \subset\ \bigcup_{k=0}^{N-1}\varphi_{k\varepsilon}\left(G_\varepsilon^{k}\right),
\qquad G_\varepsilon^{k}:=\{u\in G_\varepsilon:\ s_+(u)>k\varepsilon\},
\]
where $\varphi_{k\varepsilon}$ is defined on $G_\varepsilon^{k}$ and preserves $\mathfrak m$ by Proposition~\ref{prop:medida}. Therefore
\[
\mathfrak m(C^{-}_T)\ \le\ N\mathfrak m(G_\varepsilon)\ \le\ \left(\frac T\varepsilon+1\right)K\varepsilon^{\ell+1}\ \xrightarrow[\varepsilon\to0]{}\ 0,
\]
because $\ell+1\ge2$. Thus $\mathfrak m(C^{-}_T)=0$ for every $T$, and $\mathfrak m(C^{-})=0$, since $C^{-}=\bigcup_{T\in\N}C^{-}_T$. Finally, $R(x,y,\sigma)=(x,y,\sigma+\pi)$ preserves $\mathfrak m$ and, by Lemma~\ref{lem:simetrias}(i), maps $C^{-}$ onto $C^{+}$. Hence $\mathfrak m(C^{+})=0$.
\end{proof}

Since $\mathfrak m$ is finite and almost every solution is defined for all time, the Poincaré recurrence theorem applies.

\begin{theorem}[Recurrence]\label{thm:recorrencia}
For $\mathfrak m$-almost every $v\in M$, the solution through $v$ is defined on $\R$, and there are sequences $n_k,m_k\to\infty$ of integers with $\varphi_{n_k}(v)\to v$ and $\varphi_{-m_k}(v)\to v$.
\end{theorem}
\begin{proof}
The set $M_\infty$ is invariant, and $T:=\varphi_1$ is a bijection of $M_\infty$ onto itself, with inverse $\varphi_{-1}$, which preserves $\mathfrak m$. By Proposition~\ref{lem:nulo} and Proposition~\ref{prop:medida}, $\mathfrak m(M_\infty)=\mathfrak m(M)<\infty$. Since $M$ has a countable basis of open sets, the Poincaré recurrence theorem applied to $T$ and to $T^{-1}$ gives the conclusion for $\mathfrak m$-almost every $v\in M_\infty$, hence for $\mathfrak m$-almost every $v\in M$.
\end{proof}

Recurrence alone does not exclude closed orbits. The analyticity of the flow shows that they are either negligible or fill an open set.

\begin{proposition}[Closed trajectories]\label{prop:fechadas}
The union of the closed orbits of \eqref{eq:flow} is either $\mathfrak m$-null or contains a nonempty open subset of $M$.
\end{proposition}
\begin{proof}
Let $S:=\{(x,y_*,\sigma):x>0,\ 0<\sigma<\pi\}$, a surface transverse to the flow, since $y'=\sin\sigma>0$ on it. By Corollary~\ref{cor:noloop2} and Lemma~\ref{lem:transv}, every closed orbit crosses $\{y=y_*\}$ upward, hence meets $S$. For $j\ge1$, let $D_j\subset S$ be the set of points whose forward trajectory meets $S$ at least $j$ more times, and $P^{j}:D_j\to S$ the $j$-th return map. These maps are real-analytic. Indeed, the field \eqref{eq:flow} is real-analytic on $M$, so its local flow $(s,v)\mapsto\varphi_s(v)$ is real-analytic in both variables: locally, one complexifies the analytic coefficients and applies \cite[Theorem~4.2]{Teschl}, whose Picard-iteration proof applies componentwise to systems. Dependence at any other time in the domain follows by composing short-time analytic flow maps. Given $z\in S$ whose forward trajectory returns to $S$ at the instant $\tau_0>0$, the function $F(v,s):=y(\varphi_s(v))-y_*$ is analytic near $(z,\tau_0)$, vanishes there, and satisfies $\partial_sF=\sin\sigma>0$ at that point, by the transversality of the crossings (Lemma~\ref{lem:transv}). The analytic implicit function theorem \cite[\S2.3]{KrantzParks} gives an analytic first return time $\tau(v)$ near $z$, and $P(v)=\varphi_{\tau(v)}(v)$ is analytic near $z$. The maps $P^{j}$ are compositions of such maps, hence analytic on the open set $D_j$ where they are defined.

Let $\mathrm{Fix}(P^{j}):=\{z\in D_j: P^{j}(z)=z\}$. A point $z\in S$ lies on a closed orbit if and only if the trajectory of $z$ returns to the phase point $z$ itself, that is, if and only if $z\in\mathrm{Fix}(P^{j})$ for some $j\ge1$. Hence the points of $S$ on closed orbits form $\bigcup_{j\ge1}\mathrm{Fix}(P^{j})$.

If, for some $j$ and some connected component $W$ of $D_j$, $P^{j}\equiv\mathrm{id}$ on $W$, then the flow-out $\{\varphi_t(z):z\in W,\ t\in\R\}$ is a nonempty open subset of $M$ consisting of closed orbits. Otherwise, on each component $W$ the analytic map $P^{j}-\mathrm{id}$ does not vanish identically, so at least one of its two scalar components does not, and $\mathrm{Fix}(P^{j})\cap W$ is contained in the zero set of that component. The zero set of a real-analytic function which does not vanish identically on a connected open subset of $\R^{2}$ has Lebesgue measure zero \cite[Proposition~1]{Mityagin}. Since $S$ is second countable, each $D_j$ has at most countably many connected components, so $N:=\bigcup_j\mathrm{Fix}(P^{j})$ is a countable union of null sets, hence a null subset of $S$. The union of the closed orbits is the image of $N\times\R$ under the flow map $(z,t)\mapsto\varphi_t(z)$, which is $C^{1}$ between manifolds of dimension $3$. For each $k$, the set $N\times[-k,k]$ is null in $S\times\R$, so its image is null, a $C^{1}$ map being locally Lipschitz (see \cite[Lemma~7.25]{RudinRCA}). Taking the union over $k\in\N$ gives that the whole flow-out is $\mathfrak m$-null.
\end{proof}

We can now prove Theorem~\ref{thmH}.

\begin{theorem}[Theorem~\ref{thmH}]\label{thm:generico}
Let $c<0$. The conclusions of Theorem~\ref{thmH} hold.
\end{theorem}
\begin{proof}
The invariance and the finiteness of $\mathfrak m$ are Proposition~\ref{prop:medida} and the statement on closed orbits is Proposition~\ref{prop:fechadas}. Let $v$ be as in Theorem~\ref{thm:recorrencia}, and let $\gamma$ be its profile curve, defined on $\R$. If $\gamma$ is closed, $\Sigma_v\cong\s^{1}\times\s^{p}\times\s^{q}$ by Proposition~\ref{prop:topcompacta}. Otherwise, both ends of $\gamma$ have infinite arclength, and $\rho(\varphi_{n_k}(v))\to\rho(v)$ and $\rho(\varphi_{-m_k}(v))\to\rho(v)$, so $\rho\not\to\infty$ along either end: both ends are non-escaping. The trajectory is not the Clifford ray, along which $\rho\to\infty$ as $s\to\infty$, nor a cylinder, along which $\varphi_{n}(v)$ leaves every compact set. Hence $\Sigma_v$ is complete by Theorem~\ref{thm:completude} and noncompact. By Corollaries~\ref{cor:dicotomia}, and~\ref{cor:entropia}, $\Sigma_v$ is not proper and has infinite entropy, and by Theorem~\ref{prop:recorrente}, each end crosses both separatrices infinitely often, transversely by Lemma~\ref{lem:transv}. Finally, if the trajectory of $v$ reaches an axis or the origin, then $v\in C^{+}\cup C^{-}$. If it has an asymptotically conical end, then $\rho\to\infty$ along that end, so that $v$ is not recurrent in the corresponding time direction. By Proposition~\ref{lem:nulo} and Theorem~\ref{thm:recorrencia}, these initial conditions form an $\mathfrak m$-null set.
\end{proof}

\FloatBarrier


\begin{thebibliography}{99}

\bibitem{AbreschLanger} U.~Abresch, J.~Langer, \emph{The normalized curve shortening flow and homothetic solutions}, J.\ Differential Geom.\ \textbf{23} (1986), 175--196. \href{https://doi.org/10.4310/jdg/1214440025}{\nolinkurl{doi:10.4310/jdg/1214440025}}.
\bibitem{Angenent} S.~B.~Angenent, \emph{Shrinking doughnuts}, in: Nonlinear Diffusion Equations
and their Equilibrium States, 3 (Gregynog, 1989), Progr.\ Nonlinear Differential Equations Appl.\
\textbf{7}, Birkhäuser, 1992, 21--38. \href{https://doi.org/10.1007/978-1-4612-0393-3_2}{\nolinkurl{doi:10.1007/978-1-4612-0393-3_2}}.
\bibitem{BdCH} A.~Back, M.~do Carmo, and W.~Y.~Hsiang, \emph{On some fundamental equations of
equivariant Riemannian geometry}, Tamkang J.\ Math.\ \textbf{40} (2009), 343--376. \href{https://doi.org/10.5556/j.tkjm.40.2009.601}{\nolinkurl{doi:10.5556/j.tkjm.40.2009.601}}.
\bibitem{CE} M.~Cavalcante, J.~M.~Espinar, \emph{Halfspace type theorems for self-shrinkers},
Bull.\ Lond.\ Math.\ Soc.\ \textbf{48} (2016), 242--250. \href{https://doi.org/10.1112/blms/bdv099}{\nolinkurl{doi:10.1112/blms/bdv099}}.
\bibitem{ChangSpruck} P.-Y.~Chang, J.~Spruck, \emph{Self-shrinkers to the mean curvature flow asymptotic to isoparametric cones}, Trans.\ Amer.\ Math.\ Soc.\ \textbf{369} (2017), 6565--6582. \href{https://doi.org/10.1090/tran/6957}{\nolinkurl{doi:10.1090/tran/6957}}.
\bibitem{ChengOgata} Q.-M.~Cheng, S.~Ogata, \emph{$2$-dimensional complete self-shrinkers in $\R^{3}$}, Math.\ Z.\ \textbf{284} (2016), 537--542. \href{https://doi.org/10.1007/s00209-016-1665-2}{\nolinkurl{doi:10.1007/s00209-016-1665-2}}.
\bibitem{ChengZhou} X.~Cheng, D.~Zhou, \emph{Volume estimate about shrinkers}, Proc.\ Amer.\
Math.\ Soc.\ \textbf{141} (2013), 687--696. \href{https://doi.org/10.1090/S0002-9939-2012-11922-7}{\nolinkurl{doi:10.1090/S0002-9939-2012-11922-7}}.
\bibitem{CMP} T.~H.~Colding, T.~Ilmanen, W.~P.~Minicozzi II, B.~White, \emph{The round
sphere minimizes entropy among closed self-shrinkers}, J.\ Differential Geom.\ \textbf{95}
(2013), 53--69. \href{https://doi.org/10.4310/jdg/1375124609}{\nolinkurl{doi:10.4310/jdg/1375124609}}.
\bibitem{CM} T.~H.~Colding, W.~P.~Minicozzi II, \emph{Generic mean curvature flow I: generic
singularities}, Ann.\ of Math.\ \textbf{175} (2012), 755--833. \href{https://doi.org/10.4007/annals.2012.175.2.7}{\nolinkurl{doi:10.4007/annals.2012.175.2.7}}.
\bibitem{CMfrankel} T.~H.~Colding, W.~P.~Minicozzi II, \emph{A strong Frankel theorem for
shrinkers}, Compositio Math.\ \textbf{161} (2025), 947--958. \href{https://doi.org/10.1112/S0010437X24007747}{\nolinkurl{doi:10.1112/S0010437X24007747}}.
\bibitem{doCarmoCurvas} M.~P.~do Carmo, \emph{Differential Geometry of Curves and Surfaces},
2nd ed., Dover, 2016. ISBN 978-0-486-80699-0.
\bibitem{DrKl} G.~Drugan, S.~J.~Kleene, \emph{Immersed self-shrinkers}, Trans.\ Amer.\ Math.\ Soc.\ \textbf{369} (2017), 7213--7250. \href{https://doi.org/10.1090/tran/6907}{\nolinkurl{doi:10.1090/tran/6907}}.
\bibitem{DLN} G.~Drugan, H.~Lee, X.~H.~Nguyen, \emph{A survey of closed self-shrinkers with
symmetry}, Results Math.\ \textbf{73} (2018), Paper No.~32, 32 pp. \href{https://doi.org/10.1007/s00025-018-0763-3}{\nolinkurl{doi:10.1007/s00025-018-0763-3}}.
\bibitem{DN} G.~Drugan and X.~H.~Nguyen, \emph{Shrinking doughnuts via variational methods},
J.~Geom.\ Anal.\ \textbf{28} (2018), 3725--3746. \href{https://doi.org/10.1007/s12220-017-9976-z}{\nolinkurl{doi:10.1007/s12220-017-9976-z}}.
\bibitem{GT} D.~Gilbarg, N.~S.~Trudinger, \emph{Elliptic Partial Differential Equations of Second Order}, reprint of the 1998 ed., Classics in Mathematics, Springer, 2001. \href{https://doi.org/10.1007/978-3-642-61798-0}{\nolinkurl{doi:10.1007/978-3-642-61798-0}}.
\bibitem{Hale} J.~K.~Hale, \emph{Ordinary Differential Equations}, 2nd ed., Robert E.~Krieger
Publishing, 1980. ISBN 0-89874-011-8.
\bibitem{Halldorsson} H.~P.~Halldorsson, \emph{Self-similar solutions to the curve shortening flow}, Trans.\ Amer.\ Math.\ Soc.\ \textbf{364} (2012), 5285--5309. \href{https://doi.org/10.1090/S0002-9947-2012-05632-7}{\nolinkurl{doi:10.1090/S0002-9947-2012-05632-7}}.
\bibitem{HTY} W.-Y.~Hsiang, Z.-H.~Teng, W.-C.~Yu, \emph{New examples of constant mean curvature immersions of $(2k-1)$-spheres into Euclidean $2k$-space}, Ann.\ of Math.\ (2) \textbf{117} (1983), 609--625. \href{https://doi.org/10.2307/2007036}{\nolinkurl{doi:10.2307/2007036}}.
\bibitem{Hsiang} W.-Y.~Hsiang, \emph{Generalized rotational hypersurfaces of constant mean curvature in the Euclidean spaces.~I}, J.\ Differential Geom.\ \textbf{17} (1982), 337--356. \href{https://doi.org/10.4310/jdg/1214436924}{\nolinkurl{doi:10.4310/jdg/1214436924}}.
\bibitem{Huisken} G.~Huisken, \emph{Asymptotic behavior for singularities of the mean curvature
flow}, J.\ Differential Geom.\ \textbf{31} (1990), 285--299. \href{https://doi.org/10.4310/jdg/1214444099}{\nolinkurl{doi:10.4310/jdg/1214444099}}.
\bibitem{Ilmanen95} T.~Ilmanen, \emph{Singularities of mean curvature flow of surfaces}, preprint, 1995 (arXiv version posted 2026). \href{https://arxiv.org/abs/2601.21133}{arXiv:2601.21133}. \href{https://doi.org/10.48550/arXiv.2601.21133}{\nolinkurl{doi:10.48550/arXiv.2601.21133}}.
\bibitem{IPR} D.~Impera, S.~Pigola, M.~Rimoldi, \emph{The Frankel property for self-shrinkers
from the viewpoint of elliptic PDE's}, J.\ Reine Angew.\ Math.\ \textbf{773} (2021), 1--20. \href{https://doi.org/10.1515/crelle-2020-0044}{\nolinkurl{doi:10.1515/crelle-2020-0044}}.
\bibitem{KM} S.~Kleene, N.~M.~Møller, \emph{Self-shrinkers with a rotational symmetry}, Trans.\
Amer.\ Math.\ Soc.\ \textbf{366} (2014), 3943--3963. \href{https://doi.org/10.1090/S0002-9947-2014-05721-8}{\nolinkurl{doi:10.1090/S0002-9947-2014-05721-8}}.
\bibitem{KrantzParks} S.~G.~Krantz, H.~R.~Parks, \emph{A Primer of Real Analytic Functions}, 2nd ed., Birkh\"auser Advanced Texts, Birkh\"auser, Boston, 2002. \href{https://doi.org/10.1007/978-0-8176-8134-0}{\nolinkurl{doi:10.1007/978-0-8176-8134-0}}.
\bibitem{LaSalle} J.~P.~LaSalle, \emph{The Stability of Dynamical Systems}, with an appendix by Z.~Artstein, CBMS-NSF Regional
Conference Series in Applied Mathematics \textbf{25}, SIAM, Philadelphia, 1976. \href{https://doi.org/10.1137/1.9781611970432}{\nolinkurl{doi:10.1137/1.9781611970432}}.
\bibitem{McGrath} P.~McGrath, \emph{Closed mean curvature self-shrinking surfaces of generalized
rotational type}, preprint, arXiv:1507.00681. \href{https://doi.org/10.48550/arXiv.1507.00681}{\nolinkurl{doi:10.48550/arXiv.1507.00681}}.
\bibitem{Mityagin} B.~S.~Mityagin, \emph{The zero set of a real analytic function}, Math.\ Notes \textbf{107} (2020), 529--530. \href{https://doi.org/10.1134/S0001434620030189}{\nolinkurl{doi:10.1134/S0001434620030189}}.
\bibitem{Park-Simons} J.~Park, \emph{Ancient mean curvature flow asymptotic to a minimal quadratic cone},
preprint, arXiv:2608.11011. \href{https://doi.org/10.48550/arXiv.2608.11011}{\nolinkurl{doi:10.48550/arXiv.2608.11011}}.
\bibitem{Park-fol} J.~Park, \emph{Foliation by bi-rotational self shrinkers in Euclidean
space}, preprint, arXiv:2608.11046. \href{https://doi.org/10.48550/arXiv.2608.11046}{\nolinkurl{doi:10.48550/arXiv.2608.11046}}.
\bibitem{PR} S.~Pigola, M.~Rimoldi, \emph{Complete self-shrinkers confined into some regions of the space}, Ann.\ Global Anal.\ Geom.\ \textbf{45} (2014), 47--65. \href{https://doi.org/10.1007/s10455-013-9387-8}{\nolinkurl{doi:10.1007/s10455-013-9387-8}}.
\bibitem{Riedler} O.~Riedler, \emph{Closed embedded self-shrinkers of mean curvature flow},
J.~Geom.\ Anal.\ \textbf{33} (2023), Paper No.~172. \href{https://doi.org/10.1007/s12220-023-01217-w}{\nolinkurl{doi:10.1007/s12220-023-01217-w}}.
\bibitem{RudinRCA} W.~Rudin, \emph{Real and Complex Analysis}, 3rd ed., McGraw--Hill, 1987. ISBN 978-0-07-054234-1.
\bibitem{SNexpanders} G.~Silva Neto, \emph{The global structure of bi-rotational self-expanders}, preprint, 2026.
\bibitem{Stone} A.~Stone, \emph{A density function and the structure of singularities of the
mean curvature flow}, Calc.\ Var.\ Partial Differential Equations \textbf{2} (1994),
443--480. \href{https://doi.org/10.1007/BF01192093}{\nolinkurl{doi:10.1007/BF01192093}}.
\bibitem{Teschl} G.~Teschl, \emph{Ordinary Differential Equations and Dynamical Systems},
Graduate Studies in Mathematics \textbf{140}, Amer.\ Math.\ Soc., 2012. \href{https://doi.org/10.1090/gsm/140}{\nolinkurl{doi:10.1090/gsm/140}}.
\bibitem{Velazquez} J.~J.~L.~Vel\'azquez, \emph{Curvature blow-up in perturbations of minimal cones evolving by mean curvature flow}, Ann.\ Scuola Norm.\ Sup.\ Pisa Cl.\ Sci.\ (4) \textbf{21} (1994), 595--628. Available at \url{https://www.numdam.org/item/ASNSP_1994_4_21_4_595_0/}.
\bibitem{Walters} P.~Walters, \emph{An Introduction to Ergodic Theory}, Graduate Texts in Mathematics \textbf{79}, Springer, 1982. \href{https://link.springer.com/book/9780387951522}{\nolinkurl{doi:10.1007/978-1-4612-5775-2}}.
\bibitem{Wang} L.~Wang, \emph{Uniqueness of self-similar shrinkers with asymptotically conical ends}, J.\ Amer.\ Math.\ Soc.\ \textbf{27} (2014), 613--638. \href{https://doi.org/10.1090/S0894-0347-2014-00792-X}{\nolinkurl{doi:10.1090/S0894-0347-2014-00792-X}}.
\bibitem{WangAsym} L.~Wang, \emph{Asymptotic structure of self-shrinkers}, preprint, arXiv:1610.04904. \href{https://doi.org/10.48550/arXiv.1610.04904}{\nolinkurl{doi:10.48550/arXiv.1610.04904}}.
\bibitem{White97} B.~White, \emph{Stratification of minimal surfaces, mean curvature flows, and harmonic maps}, J.\ Reine Angew.\ Math.\ \textbf{488} (1997), 1--35. \href{https://doi.org/10.1515/crll.1997.488.1}{\nolinkurl{doi:10.1515/crll.1997.488.1}}.
\end{thebibliography}
\end{document}